\documentclass[12pt]{amsart}
\pdfoutput=1

\usepackage[utf8]{inputenc}
\usepackage[T1]{fontenc}
\usepackage{latexsym}
\usepackage{amsxtra}
\usepackage{amsmath}
\usepackage{amsfonts}
\usepackage{amssymb}
\usepackage{blkarray}
\usepackage{multirow}
\usepackage{pgf,tikz}
\usepackage{mathrsfs}
\usepackage{mathdots}
\usepackage{pdfpages}
\usepackage[colorlinks,linkcolor=blue,anchorcolor=black,citecolor=blue,filecolor=blue,menucolor=blue,runcolor=blue,urlcolor=blue]{hyperref}
\usepackage{geometry}
\definecolor{red}{rgb}{0.0, 0.0, 0.0}

\usetikzlibrary{arrows}

\tikzstyle{vertex}=[circle, draw, inner sep=0pt, minimum size=6pt]
\newcommand{\vertex}{\node[vertex]}

\newtheorem{theorem}{Theorem}[section]
\newtheorem{app-lemma}{Appendix Lemma}
\newtheorem{lemma}[theorem]{Lemma}
\newtheorem{corollary}[theorem]{Corollary}
\newtheorem{proposition}[theorem]{Proposition}

\newtheorem{conjecture}[theorem]{Conjecture}
\newtheorem{hypothesis}[theorem]{Hypothesis}
\newtheorem{claim}{Claim}[theorem]
\newenvironment{subproof}[1][\proofname]{%
  \begin{proof}[#1]%
}{%
  \end{proof}%
}
\numberwithin{equation}{section}

\theoremstyle{definition}
\newtheorem{definition}[theorem]{Definition}
\newtheorem{example}[theorem]{Example}

\theoremstyle{remark}
\newtheorem{remark}[theorem]{Remark}

\newcommand{\bFp}{\mathbb{F}_p}

\newcommand{\PG}{\mathrm{PG}}
\newcommand{\GF}{\mathrm{GF}}
\newcommand{\F}{\mathbb{F}}
\newcommand{\TM}{\widetilde{M}}
\renewcommand{\P}{\mathbb{P}}
\newcommand{\PT}{\mathbb{PT}}
\newcommand{\Ppap}{\P_{\textnormal{Pap}}}
\newcommand{\GM}{\mathcal{GM}}
\newcommand{\AC}{\mathcal{AC}_4}
\newcommand{\AF}{\mathcal{AF}_4}
\newcommand{\SL}{\mathcal{SL}_4}
\newcommand{\T}{\mathcal{T}}

\DeclareMathOperator{\YT}{YT}

\begin{document}
\title{The templates for some classes of quaternary matroids}

\author{Kevin Grace}
\address{School of Mathematics, University of Bristol, Bristol, BS8 1TW, UK, and the Heilbronn Institute for Mathematical Research, Bristol, UK \newline \newline Current Address: Vanderbilt University, Department of Mathematics, 1326 Stevenson Center, Nashville, TN 37240}
\email{kevin.m.grace@vanderbilt.edu}

\thanks{This research was supported in part by National Science Foundation grant 1500343.}

\subjclass{05B35}
\date{\today}

\begin{abstract}
Subject to hypotheses based on the matroid structure theory of Geelen, Gerards, and Whittle, we completely characterize the highly connected members of the class of golden-mean matroids and several other closely related classes of quaternary matroids. This leads to a determination of the eventual extremal functions for these classes. One of the main tools for obtaining these results is the notion of a \emph{frame template}. Consequently, we also study frame templates in significant depth.
\end{abstract}

\maketitle

\section{Introduction}
\label{sec:Introduction}
A matroid is \emph{quaternary} if it is representable over the field $\GF(4)$. Let $\tau=\frac{1+\sqrt{5}}{2}$ be the \emph{golden mean}, the positive root of $x^2-x-1$ over $\mathbb{R}$. A \emph{golden-mean matroid} is a matroid that can be represented by a matrix over $\mathbb{R}$ where every nonzero subdeterminant is $\pm\tau^i$ for some $i\in\mathbb{Z}$. We will denote the class of golden-mean matroids by $\mathcal{GM}$. Every golden-mean matroid is quaternary. The \emph{extremal function} (also called \emph{growth rate function}) for a minor-closed class $\mathcal{M}$ of matroids, denoted by $h_{\mathcal{M}}(r)$, is the function whose value at an integer $r\geq0$ is given by the maximum number of elements in a simple matroid in $\mathcal{M}$ of rank at most $r$.

Let $\mathcal{AC}_4$ denote the class of quaternary matroids representable over some field of every characteristic. Let $\mathcal{AF}_4$ be the class of matroids representable over all fields of size at least $4$, and let $\mathcal{SL}_4$ denote the class of quaternary matroids $M$ for which there exists a prime power $q$ such that $M$ is representable over all fields of size at least $q$.

We will delay the statements of Hypotheses \ref{hyp:connected-template} and \ref{hyp:clique-template} to Section \ref{sec:Frame-Templates} due to their technical nature. We prove the following.

\begin{theorem}
\label{thm:extremal}
Suppose Hypothesis \ref{hyp:clique-template} holds. For all sufficiently large $r$, we have $h_{\GM}(r)=h_{\AC}(r)=h_{\SL}(r)=h_{\AF}(r)=\binom{r+3}{2}-5$.
\end{theorem}

This verifies, for matroids of sufficiently large rank, a conjecture of Archer \cite{a05} about the golden-mean matroids.

We also completely characterize the highly connected matroids in these classes, subject to Hypotheses \ref{hyp:connected-template} and \ref{hyp:clique-template}. {\color{red}A matroid $M$ is \emph{vertically $k$-connected} if, for every set $X\subseteq E(M)$ with $r(X)+r(E-X)-r(M)<k$, either $X$ or $E-X$ is spanning. If $M$ is vertically $k$-connected, then its dual $M^*$ is \textit{cyclically $k$-connected}.} We prove the following.

\begin{theorem}
\label{thm:3sets}
Suppose Hypothesis \ref{hyp:connected-template} holds. There exists $k\in\mathbb{Z}_+$ such that every $k$-connected member of $\mathcal{AC}_4$ with at least $2k$ elements is contained in exactly one of $\mathcal{AF}_4$, $\mathcal{GM}-\mathcal{AF}_4$, and $\mathcal{SL}_4-\mathcal{AF}_4$ and such that every $k$-connected member of $\mathcal{SL}_4$ with at least $2k$ elements is representable over all fields of size at least $7$. Moreover, suppose Hypothesis \ref{hyp:clique-template} holds. There exist $k,n\in\mathbb{Z}_+$ such that the same result holds for vertically $k$-connected members of $\AC$ with an $M(K_n)$-minor.
\end{theorem}

We also give excluded minor characterizations of the highly connected members of these classes of quaternary matroids; these results are analogous to the main results of \cite{gvz19}. The Pappus matroid and the Fano matroid $F_7$ are well-known. The matroids $V_1$, $V_2$, $V_3$, $P_1$, $P_2$, and $P_3$ are defined in Section \ref{sec:Characteristic Sets}. {\color{red}Let $Y_9$ denote the matroid with the geometric representation known as the Perles configuration. (Figure \ref{fig:Y_9} in Section \ref{sec:The Highly Connected Matroids in AF4 and SL4} shows the Perles configuration.)} We prove the following four theorems. 

\begin{theorem}
\label{thm:AC4exc-minor}
Suppose Hypothesis \ref{hyp:connected-template} holds. There exists $k\in\mathbb{Z}_+$ such that a $k$-connected quaternary matroid with at least $2k$ elements is contained in $\mathcal{AC}_4$ if and only if it contains no minor isomorphic to one of the matroids in the set $\{F_7$, $F_7^*$, $V_1$, $V_1^*$, $V_2$, $V_3$, $V_3^*$, $P_1$, $P_1^*$, $P_2$, $P_3$, $P_3^*\}$. Moreover, suppose Hypothesis \ref{hyp:clique-template} holds; there exist $k,n\in\mathbb{Z}_+$ such that a vertically $k$-connected quaternary matroid with an $M(K_n)$-minor is contained in $\mathcal{AC}_4$ if and only if it contains no minor isomorphic to one of the matroids in the set $\{F_7$, $F_7^*$, $V_1$, $V_2$, $V_3$, $P_1$, $P_2$, $P_3\}$.
\end{theorem}

\begin{theorem}
\label{thm:GMexc-minor}
Suppose Hypothesis \ref{hyp:connected-template} holds. There exists $k\in\mathbb{Z}_+$ such that a $k$-connected quaternary matroid with at least $2k$ elements is golden-mean if and only if it contains no minor isomorphic to one of the matroids in the set $\{F_7$, $F_7^*$, $V_1$, $V_1^*$, $V_2$, $V_3$, $V_3^*$, $P_1$, $P_1^*$, $P_2$, $P_3$, $P_3^*$, $Pappus$, $(Pappus)^*\}$. Moreover, suppose Hypothesis \ref{hyp:clique-template} holds; there exist $k,n\in\mathbb{Z}_+$ such that a vertically $k$-connected quaternary matroid with an $M(K_n)$-minor is golden-mean if and only if it contains no minor isomorphic to one of the matroids in $\{F_7$, $F_7^*$, $V_1$, $V_2$, $V_3$, $P_1$, $P_2$, $P_3$, $Pappus\}$.
\end{theorem}

\begin{theorem}
\label{thm:SL4exc-minor}
Suppose Hypothesis \ref{hyp:connected-template} holds. There exists $k\in\mathbb{Z}_+$ such that a $k$-connected quaternary matroid with at least $2k$ elements is contained in $\SL$ if and only if it contains no minor isomorphic to one of the matroids in the set $\{F_7$, $F_7^*$, $V_1$, $V_1^*$, $V_2$, $V_3$, $V_3^*$, $P_1$, $P_1^*$, $P_2$, $P_3$, $P_3^*$, $Y_9$, $Y_9^*\}$. Moreover, suppose Hypothesis \ref{hyp:clique-template} holds; there exist $k,n\in\mathbb{Z}_+$ such that a vertically $k$-connected quaternary matroid with an $M(K_n)$-minor is contained in $\SL$ if and only if it contains no minor isomorphic to one of the matroids in the set $\{F_7$, $F_7^*$, $V_1$, $V_2$, $V_3$, $P_1$, $P_2$, $P_3$, $Y_9\}$.
\end{theorem}

\begin{theorem}
\label{thm:AF4exc-minor}
Suppose Hypothesis \ref{hyp:connected-template} holds. There exists $k\in\mathbb{Z}_+$ such that a $k$-connected quaternary matroid with at least $2k$ elements is contained in $\AF$ if and only if it contains no minor isomorphic to one of the matroids in the set $\{F_7$, $F_7^*$, $V_1$, $V_1^*$, $V_2$, $V_3$, $V_3^*$, $P_1$, $P_1^*$, $P_2$, $P_3$, $P_3^*$, $Y_9$, $Y_9^*$, $Pappus$, $(Pappus)^*\}$. Moreover, suppose Hypothesis \ref{hyp:clique-template} holds; there exist $k,n\in\mathbb{Z}_+$ such that a vertically $k$-connected quaternary matroid with an $M(K_n)$-minor is contained in $\AF$ if and only if it contains no minor isomorphic to one of the matroids in the set $\{F_7$, $F_7^*$, $V_1$, $V_2$, $V_3$, $P_1$, $P_2$, $P_3$, $Y_9$, $Pappus\}$.
\end{theorem}

This paper is based on the work of Geelen, Gerards, and Whittle in \cite{ggw15}. The results announced in \cite{ggw15} rely on the Matroid Structure Theorem by these same authors \cite{ggw06}. Hypotheses \ref{hyp:connected-template} and \ref{hyp:clique-template} are modified versions of hypotheses given in \cite{ggw15}. Hypotheses \ref{hyp:connected-template} and \ref{hyp:clique-template} are believed to be true, but their proofs are still forthcoming in future papers by Geelen, Gerards, and Whittle. The reader is referred to \cite{gvz18} for more details.

In Sections \ref{sec:Preliminaries}--\ref{sec:Characteristic Sets}, we recall concepts---such as extremal functions, represented matroids, characteristic sets, golden-mean matroids, and partial fields---that will be needed later in the paper. Sections \ref{sec:Frame-Templates}--\ref{sec:Y-Templates} provide an analysis of the theory of frame templates, introduced by Geelen, Gerards, and Whittle in \cite{ggw15}. Some results in Sections \ref{sec:Frame-Templates}--\ref{sec:Y-Templates} will also be used in \cite{cgovz-inprep}. Much of Section \ref{sec:Reducing-a-Template} generalizes results proved by Grace and Van Zwam \cite{gvz17} for binary matroids. In Section \ref{sec:Extremal Functions}, we compute the extremal functions for $\AC$ and for the golden-mean matroids. In Sections \ref{sec:Maximal Templates} and \ref{sec:The Highly Connected Matroids in AC4}, we give a complete characterization of the highly connected members of $\AC$ and prove Theorem \ref{thm:AC4exc-minor}. In Section \ref{sec:The Highly Connected Golden-Mean Matroids}, we characterize the highly connected golden-mean matroids and prove Theorem \ref{thm:GMexc-minor}. In Section \ref{sec:The Highly Connected Matroids in AF4 and SL4}, we characterize the highly connected members of $\AF$ and $\SL$, and we prove Theorems \ref{thm:SL4exc-minor} and \ref{thm:AF4exc-minor}. Finally, in Section \ref{sec:Summary}, we prove Theorem \ref{thm:3sets} and give some concluding remarks.

Several of the proofs in this paper involve some extensive case checks. Many of these case checks are aided by {\color{red}Versions 8.6 and 8.9} of the SageMath software system \cite{sage}, in particular making use of the \emph{matroids} component \cite{sage-matroid}. The author used the CoCalc (formerly SageMathCloud) online interface. The SageMath worksheets used for these computations are included with the paper in the ancillary files on arXiv.

\section{Preliminaries}
\label{sec:Preliminaries}

Unexplained notation and terminology in this paper will generally follow Oxley \cite{o11}. One exception is that we denote the vector matroid of a matrix $A$ by $M(A)$ rather than $M[A]$.

For a matroid $M$, we denote $|$si$(M)|$ by $\varepsilon(M)$; that is, $\varepsilon(M)$ is the number of rank-$1$ flats of $M$. The next theorem was proved by Geelen, Kung, and Whittle \cite[Theorem 1.1]{gkw09}, based on work by Geelen and Whittle \cite{gw03} and Geelen and Kabell \cite{gk09}.

\begin{theorem}[Growth Rate Theorem]
\label{thm:growthrate}
If $\mathcal{M}$ is a nonempty minor-closed class of matroids, then there exists $c\in\mathbb{R}$ such that either:
\begin{itemize}
\item[(1)] $h_{\mathcal{M}}(r)\leq cr$ for all $r$,
\item[(2)] $\binom{r+1}{2}\leq h_{\mathcal{M}}(r)\leq cr^2$ for all $r$ and $\mathcal{M}$ contains all graphic matroids,
\item[(3)] there is a prime power $q$ such that $\frac{q^r-1}{q-1}\leq h_{\mathcal{M}}(r)\leq cq^r$ for all $r$ and $\mathcal{M}$ contains all $\mathrm{GF}(q)$-representable matroids, or
\item[(4)] $\mathcal{M}$ contains all simple rank-2 matroids.
\end{itemize}
\end{theorem}

If (2) of the previous theorem holds for $\mathcal{M}$, then $\mathcal{M}$ is \emph{quadratically dense}. If $M$ is a simple rank-$r$ matroid in $\mathcal{M}$ such that $\varepsilon(M)=h_{\mathcal{M}}(r)$, then we call $M$ an \emph{extremal matroid} of $\mathcal{M}$. If $f$ and $g$ are real-valued functions of a real variable, then we write $f(x)\approx g(x)$ to denote that $f(x)=g(x)$ for all $x$ sufficiently large, and we say that $f$ and $g$ are \emph{eventually equal}.

We now clarify some notation and terminology. For a field $\mathbb{F}$ of characteristic $p\neq0$, we denote the prime subfield of $\mathbb{F}$ by $\bFp$. We denote an empty matrix by $[\emptyset]$. We denote a group of one element by $\{0\}$ or $\{1\}$, if it is an additive or multiplicative group, respectively. If $S'$ is a subset of a set $S$ and $G$ is a subgroup of the additive group of the vector space $\mathbb{F}^S$, we denote by $G|S'$ the projection of $G$ into $\mathbb{F}^{S'}$. Similarly, if $\bar{x}\in G$, we denote the projection of $\bar{x}$ into $\mathbb{F}^{S'}$ by $\bar{x}|S'$. Let $A$ be an $m\times n$ matrix. If $A'$ is an $m\times n'$ submatrix of $A$, then $A'$ is a \emph{column submatrix} of $A$. If $A$ is a matrix with rows and columns labeled by sets $B$ and $E$, respectively, and if $\{a_1,a_2,\ldots,a_n\}\subseteq E$, then $[a_1,a_2,\ldots,a_n]$ denotes the column submatrix $A[B,\{a_1,a_2,\ldots,a_n\}]$ of $A$. Let $A$ and $A'$ be matrices with the same dimensions, and suppose the entry in the $i$-th row and $j$-th column of $A$ is nonzero if and only if the entry in the $i$-th row and $j$-th column of $A'$ is nonzero. Then $A$ and $A'$ have the same \emph{zero-nonzero pattern}. The $n\times\binom{n}{2}$ matrix where all columns are distinct and where every column contains exactly two nonzero entries, the first a $1$ and the second a $-1$, is denoted by $D_n$ or simply $D$. Note that $D_n$ is the signed incidence matrix of the complete graph $K_n$. If $U\subseteq\mathbb{F}^E$ and $X\subseteq E$, then $U|X=\{u|X:u\in U\}$. If $\Gamma\subseteq\mathbb{F}$, then $\Gamma U=\{\gamma u|\gamma\in\Gamma, u\in U\}$. 

We use the following notation and terminology, following \cite{nw17}. The \emph{weight} of a vector is its number of nonzero entries. If $\mathbb{F}$ is a field and $A\cap B=\emptyset$, then we identify the vector space $\mathbb{F}^A\times\mathbb{F}^B$ with $\mathbb{F}^{A\cup B}$. If $U$ and $W$ are additive subgroups of $\mathbb{F}^E$, then $U$ and and $W$ are \emph{skew} if $U\cap W=\{0\}$.

In some respects, this paper is about matrices rather than matroids. However, we use these results about matrices to obtain results about their vector matroids. Because a matroid can have inequivalent representations, it will be useful to have a more formal notion of matroid representations.

Let $\mathbb{F}$ be a field. An $\mathbb{F}$-\emph{represented matroid} (or simply \emph{represented matroid} if the field is understood from the context) is a pair $M=(E,U)$, where $U$ is a subspace of $\mathbb{F}^E$. The \emph{dual} of $M$ is $M^*=(E,U^{\perp})$, where $U^{\perp}$ is the subspace consisting of all vectors orthogonal to every vector in $U$. A \emph{representation} of $M=(E,U)$ is a matrix $A$ whose row space is $U$. We write $M=M(A)$. We consider two represented matroids to be \emph{equal} if they have representation matrices that are row equivalent up to column scaling. We denote the vector matroid (in the usual sense) of a representation $A$ of $M$ by $\TM$ or $\widetilde{M}(A)$ and call it the \emph{abstract matroid} associated with $M$. Basic matroid notions such as ground sets, independent sets, bases, circuits, rank, closure, connectivity, etc. are freely carried over from abstract matroids to represented matroids. A \emph{represented frame matroid} is a represented matroid with a representation that has at most two nonzero entries per column. In Sections \ref{sec:Frame-Templates}--\ref{sec:Summary}, we will use the term \emph{matroid} to mean \emph{represented matroid}, unless an abstract matroid is specified.

The notions of restriction, deletion, contraction, and minors of matroids carry over to represented matroids also. If $M=(E,U)$ and $X\subseteq E$, then we define $M|X=(X,U|X)$. We define $M\backslash X=M|(E-X)$ and $M/X=(M^*\backslash X)^*$.

\section{Golden-Mean Matroids}
\label{sec:Golden-Mean Matroids}

Recall from Section \ref{sec:Introduction} that we denote the class of golden-mean matroids by $\mathcal{GM}$. The following characterization of $\mathcal{GM}$ was originally an unpublished result of Dirk Vertigan. Pendavingh and Van Zwam \cite[Theorem 4.9]{pvz10a} published a proof later.

\begin{theorem}
\label{thm:golden-mean-char}
The following are equivalent for a matroid $M$:
\begin{itemize}
\item $M$ is golden-mean.
\item $M$ is representable over $\mathrm{GF}(4)$ and $\mathrm{GF}(5)$.
\item $M$ is representable over $\mathrm{GF}(5)$, over $\mathrm{GF}(p^2)$ for all primes $p$, and over $\mathrm{GF}(p)$ for all primes $p$ such that $p\equiv\pm1\pmod{5}$.
\end{itemize}
\end{theorem} This characterization of the golden-mean matroids immediately leads to the following result that will be fundamental to the rest of the paper.

\begin{corollary}
\label{cor:GM-in-AC4}
The class of golden-mean matroids is contained in $\mathcal{AC}_4$.
\end{corollary}

In $\mathrm{GF}(5)$, the unique root of $x^2-x-1$ is $3$. In fields of characteristic 2, we have $x^2-x-1=x^2+x+1$. We will let $\alpha$ denote one of the roots of this polynomial in $\mathrm{GF}(4)$. Then the other root is $\alpha^2=\alpha+1=\alpha^{-1}$.

Recall from Section \ref{sec:Introduction} the definition of the extremal function $h_{\mathcal{M}}$ for a class $\mathcal{M}$ of matroids. Archer \cite{a05} and Welsh \cite{w14} have studied the extremal function of $\mathcal{GM}$. Archer conjectured that the extremal function for $\mathcal{GM}$ is
\[ h_{\mathcal{GM}}(r)=\begin{cases} 
       \binom{r+3}{2}-5 & if \hspace{12pt} r\neq3\\
11 & if \hspace{12pt} r=3
   \end{cases}.
\] He showed that indeed $h_{\mathcal{GM}}(3)=11$ and that the unique maximum-sized golden-mean matroid of rank 3 is the Betsy Ross matroid $B_{11}$. (Figure \ref{fig:Betsy-Ross}, found in Section \ref{sec:The Highly Connected Matroids in AF4 and SL4}, is a geometric representation of $B_{11}$.) He further conjectured that there are three families of matroids that are the maximum-sized golden-mean matroids for rank $r\neq3$. Welsh denoted these three conjectured maximum-sized golden-mean matroids of rank $r\neq3$ by $T_r^2$, $G_r$, and $HP_r$. Welsh also gave the matrix representations over $\mathrm{GF}(4)$ for $T_r^2$, $G_r$, and $HP_r$ given below.

\begin{definition}
\label{def:families}
 The matroids $T_r^2$, $G_r$, and $HP_r$ are the vector matroids of the following matrices over $\mathrm{GF}(4)$.

\begin{center}
$T_r^2=M\left(\begin{tabular}{|c|c|c|c|c|c|ccc|}
\hline
\multirow{2}{*}{$I_r$}&$0\cdots0$&$1\cdots1$&$\alpha\cdots\alpha$&$\alpha^2\cdots\alpha^2$\\
\cline{2-5}
&$D_{r-1}$&$I_{r-1}$&$I_{r-1}$&$I_{r-1}$\\
\hline
\end{tabular}\right)$

\vspace{12pt}
$G_r=M\left(\begin{tabular}{|c|c|c|c|c|c|ccc|}
\hline
\multirow{3}{*}{$I_r$}&$0\cdots0$&$1\cdots1$&$0\cdots0$&$\alpha\cdots\alpha$&$0\cdots0$&$1$&$1$&$1$\\
&$0\cdots0$&$0\cdots0$&$1\cdots1$&$0\cdots0$&$\alpha\cdots\alpha$&$1$&$\alpha$&$\alpha^2$\\
\cline{2-9}
&$D_{r-2}$&$I_{r-2}$&$I_{r-2}$&$I_{r-2}$&$I_{r-2}$&\multicolumn{3}{c|}{$0$}\\
\hline
\end{tabular}\right)$

\vspace{12pt}
$HP_r=M\left(\begin{tabular}{|c|c|c|c|c|c|ccc|}
\hline
\multirow{3}{*}{$I_r$}&$0\cdots0$&$\alpha\cdots\alpha$&$0\cdots0$&$\alpha\cdots\alpha$&$\alpha^2\cdots\alpha^2$&$1$&$1$&$1$\\
&$0\cdots0$&$0\cdots0$&$1\cdots1$&$\alpha\cdots\alpha$&$1\cdots1$&$1$&$\alpha$&$\alpha^2$\\
\cline{2-9}
&$D_{r-2}$&$I_{r-2}$&$I_{r-2}$&$I_{r-2}$&$I_{r-2}$&\multicolumn{3}{c|}{$0$}\\
\hline
\end{tabular}\right)$
\end{center}
\end{definition}

\section{Partial Fields}
\label{sec:Partial-Fields}
Partial fields were introduced by Semple and Whittle \cite{sw96} to study classes $\mathcal{M}$ of matroids such that a matroid $M\in\mathcal{M}$ if and only if $M$ is representable by a matrix over a field such that every nonzero subdeterminant of that matrix is an element of some multiplicative subgroup of the field. The class of golden-mean matroids is such a class. Other examples include the regular matroids, near-regular matroids, dyadic matroids, and $\sqrt[6]{1}$-matroids.

For the next several definitions, we follow Pendavingh and Van Zwam \cite{pvz10b}.
\begin{definition}
\label{def:partial-field}
A \emph{partial field} is a pair $\mathbb{P}=(R,G)$, where $R$ is a commutative ring with identity and $G$ is a subgroup of the multiplicative group $R^{\times}$ of $R$ such that $-1\in G$. When $\mathbb{P}$ is referred to as a set, it is the set $G\cup\{0\}$.
\end{definition}
A partial field $\P$ behaves very much like a field, except that, for $p,q\in\P$, the sum $p+q$ need not be an element of $\P$. Note that, if $\F$ is a field, then we may view $\F$ as a partial field by considering the partial field $(\F,\F^{\times})$.

\begin{definition}
\label{def:P-representable}
A matrix $A$ with entries in $\mathbb{P}$ is a \emph{$\mathbb{P}$-matrix} if $\det(A')\in\mathbb{P}$ for every square submatrix $A'$ of $A$. If $M$ is a matroid of rank $r$ on ground set $E$ and there exists an $r\times E$ $\mathbb{P}$-matrix $A$ such that $M = M(A)$, then we say that $M$ is \emph{representable} over $\P$ or, more briefly, \emph{$\P$-representable}.
\end{definition}

\begin{lemma}
\label{lem:identity}
If $A$ is a $\P$-matrix, then so is $[I|A]$.
\end{lemma}

\begin{proof}
Every square submatrix of $[I|A]$ has a determinant that is equal to either $0$, $1$, or the determinant of a square submatrix of $A$, up to a sign.
\end{proof}

The \emph{pivot} operation is described by Oxley in \cite[pages 81, 204]{o11} and by Pendavingh and Van Zwam in \cite[Definition 2.3]{pvz10b}. If a matrix is in the standard form $[I|A]$, then the pivot operation can be interpreted as a row operation to transform a column of $A$ into an identity column. This modifies one of the original identity columns. To complete the pivot operation, exchange the columns to obtain a new matrix in standard form. We say that we pivot \emph{on} the entry that becomes the nonzero entry in the new identity column. For representable matroids, the matroid operation of contraction corresponds to deleting the row and column containing a nonzero entry of the identity matrix, after performing a pivot operation. From Definition \ref{def:P-representable}, it is clear that a matrix obtained from a $\P$-matrix by removing rows and columns is also a $\P$-matrix. Semple and Whittle \cite[Proposition 3.3]{sw96} also showed the following.

\begin{proposition}
\label{pro:preserve-P-matrix}
Let $A$ be a $\P$-matrix. If the matrix $B$ is obtained from $A$ by interchanging a pair of rows or columns, by replacing a row or column by a nonzero scalar multiple of that row or column, or by performing a pivot on a non-zero entry of $A$, then $B$ is a $\P$-matrix.
\end{proposition}

The generalized parallel connection of two matroids $M_1$ and $M_2$ along a common restriction $N$, denoted by $P_N(M_1,M_2)$, was introduced by Brylawski \cite{b75}. Mayhew, Whittle, and Van Zwam \cite[Theorem 3.1]{mwvz11} showed the following.

\begin{theorem}
\label{thm:gen-par-con}
Suppose $A_1$ and $A_2$ are $\P$-matrices with the following structure:
\[A_1=\begin{blockarray}{ccc}
&Y_1&Y\\
\begin{block}{c[cc]}
X_1&T_1'&0\\
X&T_1&T_X\\
\end{block}
\end{blockarray}, \phantom{XXXX}A_2=\begin{blockarray}{ccc}
&Y&Y_2\\
\begin{block}{c[cc]}
X&T_X&T_2\\
X_2&0&T_2'\\
\end{block}
\end{blockarray},\] where $X$, $Y$, $X_1$, $Y_1$, $X_2$, and $Y_2$ are pairwise disjoint sets. If $X\cup Y$ is a modular flat of $M([I|A_1])$, then \[A=\begin{blockarray}{cccc}
&Y_1&Y&Y_2\\
\begin{block}{c[ccc]}
X_1&T_1'&0&0\\
X&T_1&T_X&T_2\\
X_2&0&0&T_2'\\
\end{block}
\end{blockarray}\] is a $\P$-matrix. Moreover, if $M_1=M([I|A_1])$ and $M_2=M([I|A_2])$, then $M([I|A])=P_N(M_1,M_2)$, where $N=M([I|T_X])$.
\end{theorem}

\begin{definition}
\label{def:homomorphism}
Let $\P_1$ and $\P_2$ be partial fields. A function $\varphi:\P_1\rightarrow\P_2$ is a \emph{partial-field homomorphism} if \begin{itemize}
\item $\varphi(1)=1$;
\item $\varphi(pq)=\varphi(p)\varphi(q)$ for all $p,q\in\P_1$; and
\item $\varphi(p+q)=\varphi(p)+\varphi(q)$ for all $p,q\in\P_1$ such that $p+q\in\P_1$.
\end{itemize}
\end{definition}

We will call a partial-field homomorphism $\P_1\rightarrow\P_2$ \emph{trivial} if $\P_2$ is the trivial partial field $(\{0\},\{0\})$. For a function $f:\P_1\rightarrow\P_2$ and a matrix $A$ over $\P_1$, we denote by $f(A)$ the matrix obtained by applying $f$ to each entry of $A$. The proof of the next theorem is found in \cite[Corollary 5.3]{sw96} as well as \cite[Corollary 2.9]{pvz10b}.

\begin{theorem}
\label{thm:homomorphism}
Let $\P_1$ and $\P_2$ be partial fields and let $\varphi:\P_1\rightarrow\P_2$ be a nontrivial homomorphism. If $A$ is a $\P_1$-matrix, then $\varphi(A)$ is a $\P_2$-matrix and $M(\varphi(A)) = M(A)$.
\end{theorem}

Partial-field homomorphisms have several basic properties that can be deduced easily. An important property for our purposes is that if $p\neq0$ and $\varphi$ is a nontrivial homomorphism, then $\varphi(p)\neq0$. (See also \cite[Section 5]{sw96}.)

If $G$ is a group and $g_1,g_2,\ldots,g_n\in G$, then we denote by $\langle g_1,g_2,\ldots,g_n\rangle$ the subgroup of $G$ generated by $\{g_1,g_2,\ldots,g_n\}$. We now give several examples of partial fields that will be used later in this paper.

\begin{example}
\label{exa:2-regular}
The \emph{$2$-regular} partial field is \[\mathbb{U}_2=(\mathbb{Q}(\alpha_1,\alpha_2),\langle-1,\alpha_1,\alpha_2,\alpha_1-1,\alpha_2-1,\alpha_1-\alpha_2\rangle),\] where $\alpha_1$ and $\alpha_2$ are indeterminates. This partial field has also been called the \emph{$2$-uniform} partial field \cite{vz09, pvz10b}.
\end{example}

\begin{example}
\label{exa:2-cyclotomic}
The \emph{$2$-cyclotomic} partial field is \[\mathbb{K}_2=(\mathbb{Q}(\alpha),\langle-1,\alpha,\alpha-1,\alpha+1\rangle),\] where $\alpha$ is an indeterminate.
\end{example}

\begin{example}
\label{exa:P4}
The partial field $\P_4$ is $(\mathbb{Q}(\alpha),\langle-1,\alpha,\alpha-1,\alpha+1,\alpha-2\rangle),$ where $\alpha$ is an indeterminate.
\end{example}

\begin{example}
\label{exa:gm}
Let $\tau$ be the positive root of $x^2-x-1$ over $\mathbb{R}$. The \emph{golden-mean} partial field is $\mathbb{G}=(\mathbb{Z}[\tau],\langle-1,\tau\rangle)$. Note that $\{\tau+1,\tau-1\}\subseteq\mathbb{G}$ because $\tau+1=\tau^2$ and $\tau-1=\tau^{-1}$. Note that a matroid $M$ is $\mathbb{G}$-representable if and only if $M\in\GM$.
\end{example}

The previous examples of partial fields have been studied before \cite{sw96,pvz10a,pvz10b}. However, the next definitions introduce partial fields that have not previously appeared in the literature, as far as the author can tell.

\begin{definition}
\label{def:Pappus}
The \emph{Pappus}\footnote{\label{note1}What we call the Pappus-template partial field was called the Pappus partial field in the author's PhD dissertation \cite{diss}.} partial field is \[\Ppap=(\mathbb{Q}(\alpha_1,\alpha_2),\langle-1,\alpha_1,\alpha_2,\alpha_1-1,\alpha_2-1,\alpha_1-\alpha_2,\alpha_1\alpha_2-\alpha_1+1,\alpha_1\alpha_2-\alpha_2+1\rangle),\] where $\alpha_1$ and $\alpha_2$ are indeterminates.
\end{definition}

\begin{definition}
\label{def:Pappus-template}
The \emph{Pappus-template}\textsuperscript{\ref{note1}} partial field is \[\PT=(\mathbb{Q}(\alpha),\langle-1,\alpha,\alpha+1,\alpha-1,\alpha+2,2\alpha+1\rangle),\] where $\alpha$ is an indeterminate.
\end{definition}

\begin{lemma}
\label{lem:pappus-template}
If $\P$ is either $\Ppap$ or $\PT$ and $\F$ is a field, then there is a partial-field homomorphism from $\P$ to $\F$ if and only if $\F=\GF(4)$ or $|\F|\geq7$.
\end{lemma}

\begin{proof}
First, we consider $\PT$. If $\P_1=(R_1,G_1)$ and $\P_2=(R_2,G_2)$ are partial fields and $\varphi:R_1\rightarrow R_2$ is a ring homomorphism such that $\varphi(G_1)\subseteq G_2$, then the restriction of $\varphi$ to $G_1$ is a partial-field homomorphism.

Let $\P_1=(R_1,G_1)$, where $R_1=\mathbb{Z}\left[\alpha,\frac{1}{\alpha},\frac{1}{\alpha+1},\frac{1}{\alpha-1},\frac{1}{\alpha+2},\frac{1}{2\alpha+1}\right]$ and where $G_1=\langle-1,\alpha,\alpha+1,\alpha-1,\alpha+2,2\alpha+1\rangle$. Let $\varphi_1:\PT\rightarrow\P_1$ be the partial-field homomorphism given by the identity map on the set $\PT$.

The composition of partial-field homomorphisms is again a partial-field homomorphism. Therefore, to construct a partial-field homomorphism from $\PT$ to a field $\F$, it suffices to construct a ring homomorphism $\varphi_2:R_1\rightarrow\F$ defined by $\varphi_2(\alpha)=x$, for some $x\in\F$ such that $x$, $x+1$, $x-1$, $x+2$, and $2x+1$ are all nonzero. Then $\varphi_2\circ\varphi_1:\PT\rightarrow\F$ is a partial-field homomorphism.

If $\F$ is a field other than a prime field (so $|\F|=4$ or $|\F|>7$) and $\bFp$ is its prime subfield, let $x\in\F-\bFp$. If $\F$ is a prime field of size $7$ or larger, let $x=2$. To show that there is no homomorphism $\varphi:\PT\rightarrow\F$ if $\F$ is $\GF(2)$, $\GF(3)$, or $\GF(5)$, it is easy to check that, for each $x\in\F$, if $\varphi(\alpha)=x$, then $\varphi$ maps either $\alpha,\alpha+1,\alpha-1,\alpha+2$, or $2\alpha+1$ to $0$.

The proof is similar for $\Ppap$; we need to find elements $x_1,x_2\in\F$ such that $x_1$, $x_2$, $x_1-1$, $x_2-1$, $x_1-x_2$, $x_1x_2-x_1+1$, and $x_1x_2-x_2+1$ are all nonzero. If $\F$ is a non-prime field, then let $x_1$ be a generator of the multiplicative group $\F^{\times}$ and let $x_2=x_1+1$. If $\F$ is a prime field such that $|\F|\geq7$, then let $x_1=2$ and $x_2=3$. If $\F$ is $\GF(2)$, $\GF(3)$, or $\GF(5)$, it is straightforward to check that, for each pair $x_1,x_2\in\F$, if $\varphi(\alpha_1)=x_1$ and $\varphi(\alpha_2)=x_2$, then $\varphi$ maps either $\alpha_1,\alpha_2,\alpha_1-1,\alpha_2-1,\alpha_1-\alpha_2,\alpha_1\alpha_2-\alpha_1+1$, or $\alpha_1\alpha_2-\alpha_2+1$ to $0$.
\end{proof}

The next lemma has a proof similar to that of Lemma \ref{lem:pappus-template}. (See also \cite[Proposition 3.1]{s96}, \cite[Lemma 4.14]{pvz10a}, and \cite[Lemma 2.5.42]{vz09}.)

\begin{lemma}
\label{lem:2-regular}
If $\P$ is either $\mathbb{U}_2$, $\mathbb{K}_2$, or $\P_4$ and $\F$ is a field, then there is a partial-field homomorphism from $\P$ to $\F$ if and only if $|\F|\geq4$.
\end{lemma}

\begin{lemma}
\label{lem:pf-homs}
There are partial-field homomorphisms $\mathbb{U}_2\rightarrow\mathbb{K}_2\rightarrow\P_4\rightarrow\mathbb{G}$.
\end{lemma}

\begin{proof}
The first homomorphism is obtained by restricting to $\mathbb{U}_2$ the ring homomorphism $\mathbb{Q}(\alpha_1,\alpha_2)\rightarrow\mathbb{Q}(\alpha)$ that maps $\alpha_1$ to $\alpha$ and $\alpha_2$ to $\alpha+1$. The second homomorphism is the inclusion map. The third homomorphism maps polynomials $f(\alpha)$ over $\mathbb{Z}$ to $f(\tau)$. Thus, $\alpha$, $\alpha+1$, $\alpha-1$, and $\alpha-2$ are mapped to $\tau$, $\tau+1=\tau^2$, $\tau-1=\tau^{-1}$, and $\tau-2=-\tau^{-2}$, respectively.
\end{proof}

Recall the definitions of $\GM$, $\AC$, $\AF$, and $\SL$, from Section \ref{sec:Introduction}.

\begin{corollary}
\label{cor:PtoF}
\leavevmode
\begin{enumerate}
\item[(i)] If a matroid $M$ is representable over $\P\in\{\mathbb{U}_2,\mathbb{K}_2,\P_4\}$, then $M\in\AF$.
\item[(ii)] If a matroid $M$ is representable over $\P\in\{\Ppap,\PT,\mathbb{U}_2,\mathbb{K}_2,\P_4\}$, then $M\in\SL$. In particular, $M$ is representable over $\GF(4)$ and all fields of size at least $7$.
\item[(iii)] If a matroid $M$ is representable over $\P\in\{\mathbb{U}_2,\mathbb{K}_2,\P_4,\mathbb{G}\}$, then $M\in\mathcal{GM}$.
\item[(iv)] If a matroid $M$ is representable over $\P\in\{\Ppap,\PT,\mathbb{U}_2,\mathbb{K}_2,\P_4,\mathbb{G}\}$, then $M\in\mathcal{AC}_4$.
\end{enumerate}
\end{corollary}

\begin{proof}
By Theorem \ref{thm:homomorphism}, to prove that $M$ is representable over some field $\F$, it suffices to prove that $M$ is representable over a partial field $\P$ such that there is a homomorphism from $\P$ to $\F$. By Lemma \ref{lem:2-regular}, if $\P\in\{\mathbb{U}_2,\mathbb{K}_2,\P_4\}$, there is a homomorphism from $\P$ to every field of size at least $4$. Thus, (i) holds. Lemma \ref{lem:pappus-template}, combined with (i), implies (ii). Lemma \ref{lem:pf-homs} implies (iii). Finally, (ii), (iii), and Lemma \ref{cor:GM-in-AC4} imply (iv).
\end{proof}

We will need the theory of universal partial fields developed by Pendavingh and Van Zwam in \cite{pvz10b}. The universal partial field $\P_M$ of a matroid $M$ can be thought of as the ``most general'' partial field over which a matroid is representable.

We construct a system of polynomial equations that has a solution over a partial field $\mathbb{P}$ if and only if $M$ is $\P$-representable. We start with four observations. First, for any $r\times |E(M)|$ representation matrix $A$ of the rank-$r$ matroid $M$, an $r\times r$ submatrix $D$ has $\det(D) \neq 0$ if and only if the set of column labels is a basis of $M$. Second, given a basis $B$ of $M$, we can pivot so that the submatrix of $A$ corresponding to $B$ is an identity matrix. Third, the fundamental circuits of $M$ with respect to $B$ determine exactly which entries of $A$ are zero (see Oxley \cite[Proposition 6.4.1]{o11}). Fourth, we can choose some entries of the remaining matrix arbitrarily. These entries correspond to a maximal forest of the \emph{fundamental graph of $M$ with respect to $B$} (see \cite[Theorem 6.4.7]{o11} and also \cite[Lemma 2.25]{pvz10a}). In this paper, we will choose these entries to be either $1$ or $-1$. This is possible since every partial field contains $\{1,-1\}$.

Introduce variables $x_1, \ldots, x_s$, one for each matrix entry not determined by the above observations, and variables $y_1, \ldots, y_t$, one for each basis of $M$. Fill a matrix $A'$ over $\mathbb{Z}[x_1, \ldots, x_s, y_1, \ldots, y_t]$ with zeros, ones, negative ones, and the $x_i$ as described. Let $S$ be a system of polynomials, one for each $r$-subset $X \subseteq E(M)$, given by
 \[   \begin{cases}
        \det(D) & \text{if } X \text{ is not a basis of } M\\
        \det(D) y_i - 1 & \text{if } X \text{ is the } i^{\text{th}} \text{ basis of } M,
    \end{cases}\]
where $D$ is the $r\times r$ submatrix of $A'$ corresponding to $X$. From the construction we have the following:

\begin{theorem}
\label{thm:P-representability}
The matroid $M$ has a representation over $\mathbb{P}$ if and only if we can assign elements of $\mathbb{P}$ to the variables $x_1, \ldots, x_t, y_1, \ldots, y_s$ such that all polynomials in $S$ evaluate to $0$.
\end{theorem}

Rather than repeat the definition of the universal partial field as given by Pendavingh and Van Zwam, it will be more useful for our purposes to use an equivalent definition. They proved this equivalence in \cite[Theorem 4.13]{pvz10b}.

\begin{definition}
\label{def:UPF}
Let $I$ be the ideal in $\mathbb{Z}[x_1, \ldots, x_s, y_1, \ldots, y_t]$ generated by the polynomials in $S$, and let $R=\mathbb{Z}[x_1, \ldots, x_s, y_1, \ldots, y_t]/I$. The universal partial field of $M$ is $\P_M=(R,\langle-1,y_1, \ldots, y_t\rangle)$.
\end{definition}

\begin{theorem}[{\cite[{Corollary 4.12}]{pvz10b}}]
\label{thm:UPF-homomorphism}
Let $\P$ be a partial field. A matroid $M$ is $\P$-representable if and only if there is a partial-field homomorphism $\P_M\rightarrow\P$.
\end{theorem}

\section{Characteristic Sets}
\label{sec:Characteristic Sets}

The \emph{characteristic set} of a matroid $M$ is the set $\mathcal{K}(M)=\{k : M$ is representable over some field of characteristic $k\}$. Let $\mathcal{P}$ denote the set of prime numbers. We will denote by $\mathcal{AC}_q$ the class of $\mathrm{GF}(q)$-representable matroids with characteristic set $\mathcal{P}\cup\{0\}$. (The notation $\mathcal{AC}$ stands for ``all characteristics.'') Combining results of Rado \cite[Theorem 6]{r57} and V\'amos \cite{v71} gives the following.

\begin{theorem}
\label{thm:char-sets}
If $M$ is a matroid, then either $0\in\mathcal{K}(M)$ and $\mathcal{P}-\mathcal{K}(M)$ is finite, or $0\notin\mathcal{K}(M)$ and $\mathcal{K}(M)$ is finite.
\end{theorem}

For a matroid $M$, recall the system $S$ of polynomials described at the end of the previous section. Consider the system of equations obtained by setting each polynomial in $S$ equal to $0$. By Theorem \ref{thm:P-representability}, $M$ is representable over a field $\F$ if and only if this system has a solution in $\F$. Baines and V\'amos developed an algorithm \cite[Theorem 3.5]{bv03} to determine the set of characteristics of the fields over which such a system has a solution. This algorithm involves the Gr\"obner basis of the ideal generated by $S$. The author implemented this algorithm in SageMath. The code used for this is essentially identical to code written by Dillon Mayhew \cite{m16}. 

It is well-known that the Fano matroid $F_7$ is $\mathbb{F}$-representable if and only if $\mathbb{F}$ has characteristic $2$. (See, for example, Oxley \cite[Proposition 6.4.8]{o11}.) It is also well-known that, for each $e\in E(F_7)$, the matroids $F_7/e$ and $F_7\backslash e$ are representable over every field. Therefore, since representability over a field is preserved by duality, $F_7$ and $F_7^*$ are both excluded minors for $\mathcal{AC}_4$.

We introduce here six additional excluded minors for $\mathcal{AC}_4$, none of which are binary. Three of them have characteristic set $\{2\}$, and three have characteristic set $(\mathcal{P}-\{3\})\cup\{0\}$.

\begin{definition}
\label{def:AC4-excluded-minors}
We define $V_1$, $V_2$, $V_3$, $P_1$, $P_2$, and $P_3$ to be the vector matroids of the following quaternary matrices.
\[\begin{blockarray}{cccccccccccc}
&&1 & 2 & 3 & 4 & 5 & 6 & 7 & 8 & 9& \\
\begin{block}{c(c[ccccccccc]c)}
&&1&0&0&1&0&1&1&1&1&\\
V_1=M&&0&1&0&1&1&\alpha^2&\alpha^2&\alpha&0&\\
&&0&0&1&0&1&1&\alpha&1&\alpha&\\
\end{block}
\end{blockarray}\]

\[\begin{blockarray}{ccccccccccc}
&&1 & 2 & 3 & 4 & 5 & 6 & 7 & 8 & \\
\begin{block}{c(c[cccccccc]c)}
&&1&0&0&0&1&1&1&0&\\
V_2=M&&0&1&0&0&\alpha^2&\alpha&\alpha&1&\\
&&0&0&1&0&\alpha&\alpha^2&\alpha&1&\\
&&0&0&0&1&\alpha^2&\alpha^2&\alpha&\alpha&\\
\end{block}
\end{blockarray}\]

\[\begin{blockarray}{cccccccccccc}
&&1 & 2 & 3 & 4 & 5 & 6 & 7 & 8 & 9& \\
\begin{block}{c(c[ccccccccc]c)}
      &&1   &   0  &   0   &0&  1      &  0 & 1           &   1        &   1        &\\
V_3=M &&0   &   1  &   0   &0&  \alpha &  1 & \alpha^2    &   \alpha^2 &   1        &\\
      &&0   &   0  &   1   &0&  \alpha &  1 & \alpha      &   1        &   \alpha^2 &\\
      &&0   &   0  &   0   &1&  0      &  1 & 1           &   0        &   0        &\\
\end{block}
\end{blockarray}\]

\[\begin{blockarray}{cccccccccccc}
&&1 & 2 & 3 & 4 & 5 & 6 & 7 & 8 & 9& \\
\begin{block}{c(c[ccccccccc]c)}
&&0&1&1&1&1&1&0&0&0&\\
P_1=M&&0&1&\alpha&0&0&0&1&1&1&\\
&&1&0&0&1&\alpha&\alpha^2&1&\alpha&\alpha^2&\\
\end{block}
\end{blockarray}\]

\[\begin{blockarray}{ccccccccccc}
&&1 & 2 & 3 & 4 & 5 & 6 & 7 & 8 & \\
\begin{block}{c(c[cccccccc]c)}
     &&    1  &   0 &  0 &  0 & 1        &  0  &  1        & 1&\\
P_2=M&&    0  &   1 &  0 &  0 & \alpha   &  1  &  1        & 1&\\
     &&    0  &   0 &  1 &  0 & 1        &  1  &  0        & 1&\\
     &&    0  &   0 &  0 &  1 & \alpha^2 &  1  &  \alpha^2 & 0&\\
\end{block}
\end{blockarray}\]

\[\begin{blockarray}{cccccccccccc}
&&1 & 2 & 3 & 4 & 5 & 6 & 7 & 8 & 9& \\
\begin{block}{c(c[ccccccccc]c)}
&&0&1&1&1&1&1&1&0&0&\\
P_3=M&&0&1&\alpha&\alpha^2&0&0&0&1&1&\\
&&1&0&0&0&1&\alpha&\alpha^2&1&\alpha&\\\end{block}
\end{blockarray}\]
\end{definition}

Geometric representations of $V_1$ and $V_2$, essentially the work of James Oxley \cite{o17}, are given in Figures \ref{fig:V_1} and \ref{fig:V_2}. Note that $P_1$ is the result of deleting from the rank-$3$ Dowling geometry $Q_3(\mathrm{GF}(4)^{\times})$ two joints and an additional point in the closure of the two joints, and that $P_3$ is the result of deleting from $Q_3(\mathrm{GF}(4)^{\times})$ two joints and a point (other than the third joint) not in the closure of the two joints.

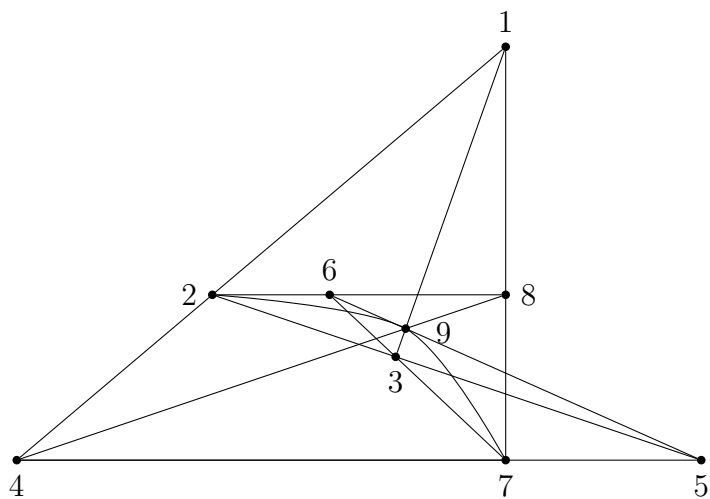
\begin{figure}[ht]
\[\begin{tikzpicture}[x=1.3cm, y=1.1cm]
	\vertex[fill,inner sep=1pt,minimum size=1pt] (1) at (5,5) [label=above:$1$] {};
 	\vertex[fill,inner sep=1pt,minimum size=1pt] (2) at (2,2) [label=left:$2$] {};
 	\vertex[fill,inner sep=1pt,minimum size=1pt] (3) at (31/8, 5/4) [label=below:$3$] {};
	\vertex[fill,inner sep=1pt,minimum size=1pt] (4) at (0,0) [label=below:$4$] {};
 	\vertex[fill,inner sep=1pt,minimum size=1pt] (5) at (7,0) [label=below:$5$] {};
 	\vertex[fill,inner sep=1pt,minimum size=1pt] (6) at (16/5,2) [label=above:$6$] {};
 	\vertex[fill,inner sep=1pt,minimum size=1pt] (7) at (5,0) [label=below:$7$] {};
 	\vertex[fill,inner sep=1pt,minimum size=1pt] (8) at (5,2) [label=right:$8$] {};
 	\vertex[fill,inner sep=1pt,minimum size=1pt] (9) at (175/44,35/22) [label=right:$$] {};
\node at (192/44,34/22) {9};
 	\path
 		(1) edge (4)
 		(1) edge (7)
 		(4) edge (5)
 		(4) edge (7)
 		(2) edge (8)
 		(2) edge (5)
 		(4) edge (8)
 		(6) edge (7)
 		(6) edge (5)
 		(1) edge (3);
\draw plot [smooth] coordinates {(2,2) (175/44,35/22) (5,0)};
\end{tikzpicture}\]
\caption{A Geometric Representation of $V_1$}
  \label{fig:V_1}
\end{figure}

\begin{figure}[ht]
\[\begin{tikzpicture}[x=1.3cm, y=.5cm]
 	\vertex[fill,inner sep=1pt,minimum size=1pt] (1) at (14/5,17/5) [label=above:$1$] {};
 	\vertex[fill,inner sep=1pt,minimum size=1pt] (2) at (5,3) [label=above:$2$] {};
 	\vertex[fill,inner sep=1pt,minimum size=1pt] (3) at (1,1) [label=below:$3$] {};
 	\vertex[fill,inner sep=1pt,minimum size=1pt] (4) at (6,3) [label=right:$4$] {};
 	\vertex[fill,inner sep=1pt,minimum size=1pt] (5) at (11/2,1) [label=below:$5$] {};
 	\vertex[fill,inner sep=1pt,minimum size=1pt] (6) at (2,3) [label=left:$6$] {};
 	\vertex[fill,inner sep=1pt,minimum size=1pt] (7) at (7,1) [label=below:$7$] {};
 	\vertex[fill,inner sep=1pt,minimum size=1pt] (8) at (16/7,13/7) [label=below:$8$] {};

  	\path
  		(0,8.5) edge (4,8)
  		(4,8) edge (8,8.5)
  		(0,8.5) edge (0,-5)
  		(8,8.5) edge (8,-5)
  		(4,8) edge (4,-5.5)
  		(0,-5) edge (4,-5.5)
  		(4,-5.5) edge (8,-5)
		(3) edge (4,7)
		(3) edge (4,5)
		(3) edge (4,3)
		(7) edge (4,7)
		(7) edge (4,4)
		(6) edge (4,4)
		(8) edge (4,7)
 		(6) edge (4,-5)
 		(4) edge (4,-5)
 		(4) edge (4,3)
		(5) edge (4,7)
	;
\draw plot [smooth] coordinates {(4,5) (5) (7)};
\end{tikzpicture}\]
\caption{A Geometric Representation of $V_2$}
  \label{fig:V_2}
\end{figure}
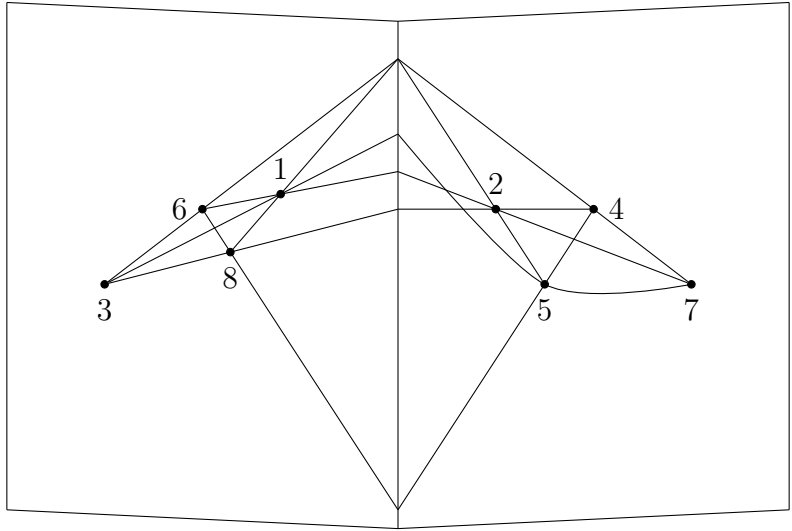

\begin{theorem}
\label{thm:AC4-excluded}
The matroids $V_1$, $V_2$, and $V_3$ each have $\{2\}$ as their characteristic set. The matroids $P_1$, $P_2$, and $P_3$ each have $(\mathcal{P}-\{3\})\cup\{0\}$ as their characteristic set. All six of these matroids are excluded minors for the class $\mathcal{AC}_4$. Moreover, $V_2$ and $P_2$ are self-dual.
\end{theorem}

\begin{proof}[Proof sketch]
For each $M\in{V_1,V_2,V_3,P_1,P_2,P_3}$ and each $e\in E(M)$, run the algorithm of Baines and V\'amos to determine the characteristic sets of $M$, $M\backslash e$, and $M/e$. This was done using SageMath. The fact that $V_2$ and $P_2$ are self-dual was also verified using SageMath.
\end{proof}

\section{Frame Templates}
\label{sec:Frame-Templates}

In this section, we will define the notion of a frame template, which is a key tool for obtaining our results. The definition given here is found in \cite{gvz18} and is a slight modification of the original definition given by Geelen, Gerards, and Whittle in \cite{ggw15}. Recall the definition and discussion of represented matroids given in Section \ref{sec:Preliminaries}. For the remainder of the paper (besides the \ref{sec:Appendix}), we will use the term \emph{matroid} to mean \emph{represented matroid}, unless an abstract matroid is specified.

{\color{red}A \emph{frame matrix} is a matrix each of whose columns has at most two nonzero entries.}  If $\F$ is a field, let $\Gamma$ be a subgroup of $\mathbb{F}^{\times}$. A $\Gamma$-frame matrix is a frame matrix $A$ such that:
\begin{itemize}
 \item Each column of $A$ with a nonzero entry contains a 1.
 \item If a column of $A$ has a second nonzero entry, then that entry is $-\gamma$ for some $\gamma\in\Gamma$.
\end{itemize}
If $\Gamma=\{1\}$, then the vector matroid of a $\Gamma$-frame matrix is a graphic matroid. For this reason, we will call the columns of a $\{1\}$-frame matrix \emph{graphic columns}.

A \textit{frame template} over a field $\mathbb{F}$ is a tuple $\Phi=(\Gamma,C,X,Y_0,Y_1,A_1,\Delta,\Lambda)$ such that the following hold.
\begin{itemize}
 \item [(i)] $\Gamma$ is a subgroup of $\mathbb{F}^{\times}$.
 \item [(ii)] $C$, $X$, $Y_0$ and $Y_1$ are disjoint finite sets.
 \item [(iii)] $A_1\in \mathbb{F}^{X\times (C\cup Y_0\cup Y_1)}$.
 \item [(iv)] $\Lambda$ is a subgroup of the additive group of $\mathbb{F}^X$ and is closed under scaling by elements of $\Gamma$.
 \item [(v)] $\Delta$ is a subgroup of the additive group of $\mathbb{F}^{C\cup Y_0 \cup Y_1}$ and is closed under scaling by elements of $\Gamma$.
\end{itemize}

Let $\Phi=(\Gamma,C,X,Y_0,Y_1,A_1,\Delta,\Lambda)$ be a frame template. Let $B$ and $E$ be finite sets, and let $A'\in\mathbb{F}^{B\times E}$. We say that $A'$ \textit{respects} $\Phi$ if the following hold:
\begin{itemize}
 \item [(i)] $X\subseteq B$ and $C, Y_0, Y_1\subseteq E$.
 \item [(ii)] $A'[X, C\cup Y_0\cup Y_1]=A_1$.
 \item [(iii)] There exists a set $Z\subseteq E-(C\cup Y_0\cup Y_1)$ such that $A'[X,Z]=0$, each column of $A'[B-X,Z]$ is a unit vector, and $A'[B-X, E-(C\cup Y_0\cup Y_1\cup Z)]$ is a $\Gamma$-frame matrix.
 \item [(iv)] Each column of $A'[X,E-(C\cup Y_0\cup Y_1\cup Z)]$ is contained in $\Lambda$.
 \item [(v)] Each row of $A'[B-X, C\cup Y_0\cup Y_1]$ is contained in $\Delta$.
\end{itemize}
The structure of $A'$ is shown below.

\begin{center}
\begin{tabular}{ r|c|c|ccc| }
\multicolumn{2}{c}{}&\multicolumn{1}{c}{$Z$}&\multicolumn{1}{c}{$Y_0$}&\multicolumn{1}{c}{$Y_1$}&\multicolumn{1}{c}{$C$}\\
\cline{2-6}
$X$&columns from $\Lambda$&$0$&&$A_1$&\\
\cline{2-6}
&&&&&\\
&$\Gamma$-frame matrix&unit columns&\multicolumn{3}{c|}{rows from  $\Delta$}\\
&&&&&\\
\cline{2-6}
\end{tabular}
\end{center}

Now, suppose that $A'$ respects $\Phi$ and that $A\in \mathbb{F}^{B\times E}$ satisfies the following conditions:
\begin{itemize}
\item [(i)] $A[B,E-Z]=A'[B,E-Z]$
\item [(ii)] For each $i\in Z$ there exists $j\in Y_1$ such that the $i$-th column of $A$ is the sum of the $i$-th and the $j$-th columns of $A'$.
\end{itemize}
We say that such a matrix $A$ \textit{conforms} to $\Phi$.

We say that an $\mathbb{F}$-represented matroid $M$ \textit{conforms} to $\Phi$ if there is a matrix $A$ conforming to $\Phi$ such that $M$ is isomorphic to $M(A)/C\backslash Y_1$. We denote by $\mathcal{M}(\Phi)$ the set of $\mathbb{F}$-represented matroids that conform to $\Phi$. We say that an $\mathbb{F}$-represented matroid $M$ \textit{coconforms} to a template $\Phi$ if its dual $M^*$ conforms to $\Phi$. We denote by $\mathcal{M}^*(\Phi)$ the set of $\F$-represented matroids that coconform to $\Phi$.

We now state the hypotheses on which the main results are based. As stated in Section \ref{sec:Introduction}, they are modified versions of hypotheses given by Geelen, Gerards, and Whittle in \cite{ggw15}, and their proofs are forthcoming. In their current forms, these hypotheses were stated by Grace and Van Zwam \cite{gvz18}.

\begin{hypothesis}[{\cite[{Hypothesis 4.3}]{gvz18}}]
\label{hyp:connected-template}
Let $\mathbb F$ be a finite field, let $m$ be a positive integer, and let 
$\mathcal M$ be a minor-closed class of $\mathbb F$-represented matroids.
Then there exist $k\in\mathbb{Z}_+$ and
frame templates $\Phi_1,\ldots,\Phi_s,\Psi_1,\ldots,\Psi_t$ such that
\begin{enumerate}
\item
$\mathcal{M}$ contains each of the classes
$\mathcal{M}(\Phi_1),\ldots,\mathcal{M}(\Phi_s)$,
\item
$\mathcal{M}$ contains the duals of the represented matroids in each of the classes
$\mathcal{M}(\Psi_1),\ldots,\mathcal{M}(\Psi_t)$, and
\item
if $M$ is a simple $k$-connected member of $\mathcal M$ with at least $2k$ elements
and $\widetilde{M}$ has no $\PG(m-1,\bFp)$-minor, 
then either
$M$ is a member of at least one of the classes
$\mathcal{M}(\Phi_1),\ldots,\mathcal{M}(\Phi_s)$, or
$M^*$ is a member of at least one of the classes
$\mathcal{M}(\Psi_1),\ldots,\mathcal{M}(\Psi_t)$.
\end{enumerate}
\end{hypothesis}

\begin{hypothesis}[{\cite[{Hypothesis 4.6}]{gvz18}}]
 \label{hyp:clique-template}
Let $\mathbb F$ be a finite field, let $m$ be a positive integer, and let $\mathcal M$ be a minor-closed class of $\mathbb F$-represented matroids. Then there exist $k,n\in\mathbb{Z}_+$ and frame templates $\Phi_1,\ldots,\Phi_s,\Psi_1,\ldots,\Psi_t$ such that
\begin{enumerate}
\item $\mathcal{M}$ contains each of the classes $\mathcal{M}(\Phi_1),\ldots,\mathcal{M}(\Phi_s)$,
\item $\mathcal{M}$ contains the duals of the represented matroids in each of the classes $\mathcal{M}(\Psi_1),\ldots,\mathcal{M}(\Psi_t)$,
\item if $M$ is a simple vertically $k$-connected member of $\mathcal M$ with an $M(K_n)$-minor but no $\PG(m-1,\bFp)$-minor, then $M$ is a member of at least one of the classes $\mathcal{M}(\Phi_1),\ldots,\mathcal{M}(\Phi_s)$, and
\item if $M$ is a cosimple cyclically $k$-connected member of $\mathcal M$ with an $M^*(K_n)$-minor but no $\PG(m-1,\bFp)$-minor, then $M^*$ is a member of at least one of the classes $\mathcal{M}(\Psi_1),\ldots,\mathcal{M}(\Psi_t)$.
\end{enumerate}
\end{hypothesis}

Since a template is a rich structure, it may happen that $\mathcal{M}(\Phi) = \mathcal{M}(\Phi')$ even though $\Phi$ and $\Phi'$ look very different. We now consider some tools to deal with such situations.

\begin{definition}
\label{def:equivalent} Let $\Phi$ and $\Phi'$ be frame templates over the fields $\mathbb{F}$ and $\mathbb{F}'$, respectively. The pair $\Phi,\Phi'$ are \emph{strongly equivalent} if $\mathbb{F}=\mathbb{F}'$ and if $\mathcal{M}(\Phi)=\mathcal{M}(\Phi')$. The pair $\Phi,\Phi'$ are \emph{minor equivalent} if $\mathbb{F}=\mathbb{F}'$ and if the closures of $\mathcal{M}(\Phi)$ and $\mathcal{M}(\Phi')$ under the taking of minors are equal. If there is a one-to-one correspondence between $\mathbb{F}$-represented matroids $M\in\mathcal{M}(\Phi)$ and $\mathbb{F}'$-represented matroids $N\in\mathcal{M}(\Phi')$ with $\widetilde{M}=\widetilde{N}$, then $\Phi$ and $\Phi'$ are \emph{algebraically equivalent}. 
\end{definition}

What we call strong equivalence was simply called equivalence in \cite{nw17} and \cite{gvz18}. We will reserve the term \emph{equivalent} for a different notion that we will introduce in Section \ref{sec:Reducing-a-Template}. Nelson and Walsh \cite{nw17} gave Definition \ref{def:reduced} and proved Lemma \ref{lem:all-reduced} below.

\begin{definition}
 \label{def:reduced}
A frame template $\Phi=(\Gamma,C,X,Y_0,Y_1,A_1,\Delta,\Lambda)$ is \emph{reduced} if there is a partition $(X_0,X_1)$ of $X$ such that
\begin{itemize}
 \item $\Delta=\Gamma(\mathbb{F}^C_p\times\Delta')$ for some additive subgroup $\Delta'$ of $\mathbb{F}^{Y_0\cup Y_1}$,
\item$\mathbb{F}_p^{X_0}\subseteq\Lambda|X_0$ while $\Lambda|X_1=\{0\}$ and $A_1[X_1,C]=0$, and
\item the rows of $A_1[X_1,C\cup Y_0\cup Y_1]$ form a basis for a subspace whose additive group is skew to $\Delta$.
\end{itemize}
We will refer to the partition $X=X_0\cup X_1$ given in Definition \ref{def:reduced} as the \emph{reduction partition} of $\Phi$.
\end{definition}

\begin{lemma}[{\cite[{Lemma 5.6}]{nw17}}]
 \label{lem:all-reduced}
Every frame template is strongly equivalent to a reduced frame template.
\end{lemma}

The following definition and theorem are found in \cite{gvz18}.

\begin{definition}[{\cite[{Definition 5.3}]{gvz18}}]
 \label{def:refined}
A frame template $\Phi=(\Gamma,C,X,Y_0,Y_1,A_1,\Delta,\Lambda)$ is \emph{refined} if it is reduced, with reduction partition $X=X_0\cup X_1$, and if $Y_1$ spans the matroid $M(A_1[X_1,Y_0\cup Y_1])$.
\end{definition}

\begin{theorem}[{\cite[{Theorem 5.6}]{gvz18}}]
\label{thm:Y1spanning}
If Hypothesis \ref{hyp:connected-template} holds for a class $\mathcal{M}$, then the constant $k$ and the templates $\Phi_1,\ldots,\Phi_s,\Psi_1,\ldots,\Psi_t$ can be chosen so that the templates are refined. Moreover, if Hypothesis \ref{hyp:clique-template} holds for a class $\mathcal{M}$, then the constants $k, n$, and the templates $\Phi_1,\ldots,\Phi_s,\Psi_1,\ldots,\Psi_t$ can be chosen so that the templates are refined.
\end{theorem}

To simplify the proofs in this paper, it will be helpful to expand the concept of conforming slightly.

\begin{definition}
 \label{virtual}
Let $A'$ be a matrix that respects $\Phi$, as defined above, except that we allow columns of $A'[B-X,Z]$ to be either unit columns or zero columns. Let $A$ be a matrix that is constructed from $A'$ as described above. Thus, $A[B,E-Z]=A'[B,E-Z]$, and for each $i\in Z$ there exists $j\in Y_1$ such that the $i$-th column of $A$ is the sum of the $i$-th and the $j$-th columns of $A'$. Let $M$ be isomorphic to $M(A)/C\backslash Y_1$. We say that $A$ and $M$ \textit{virtually conform} to $\Phi$ and that $A'$ \textit{virtually respects} $\Phi$. If $M^*$ virtually conforms to $\Phi$, we say that $M$ \textit{virtually coconforms} to $\Phi$.
\end{definition}

We will denote the set of $\mathbb{F}$-represented matroids that virtually conform to $\Phi$ by $\mathcal{M}_v(\Phi)$ and the set of $\mathbb{F}$-represented matroids that virtually coconform to $\Phi$ by $\mathcal{M}^*_v(\Phi)$.

Many applications of frame templates reduce to cases involving a particular type of template that we introduce now.

\begin{definition}
\label{def:Y-template}
A \emph{$Y$-template} over a field $\F$ is a refined template with all groups trivial (so $C=X_0=\emptyset$).
\end{definition}

Suppose $\Phi=(\{1\},\emptyset,X,Y_0,Y_1,A_1,\{0\},\{0\})$ is a $Y$-template. Since $\Phi$ is refined, $Y_1$ spans $M(A_1)$. Therefore, by elementary row operations, we may assume that $A_1$ is of the following form:

\begin{center}
\begin{tabular}{|c|c|c|}
\multicolumn{2}{c}{$Y_1$}&\multicolumn{1}{c}{$Y_0$}\\
\hline
$I_{|X|}$&$P_1$&$P_0$\\
\hline
\end{tabular}
\end{center}

\begin{definition}
\label{def:YT}
If $A_1$ has the form above, then $\YT(P_0,P_1)$ is defined to be the $Y$-template $(\{1\},\emptyset,X,Y_0,Y_1,A_1,\{0\},\{0\})$.
\end{definition}

Note that a {\color{red}simple} matroid conforming, or virtually conforming, to $\YT(P_0,P_1)$ is a restriction of the vector matroid of a matrix of the form shown in Figure \ref{fig:universal-conform} or Figure \ref{fig:universal}, respectively. Recall from Section \ref{sec:Preliminaries} that $D_n$ (or simply $D$) denotes the $n\times\binom{n}{2}$ matrix where all columns are distinct and where every column contains exactly two nonzero entries, the first a $1$ and the second a $-1$.

\begin{figure}[ht]
\begin{center}
\begin{tabular}{|c|c|c|c|c|c|c|c|c|}
\hline
$0$&$0$&$I_{|X|}$&$\cdots$&$I_{|X|}$&$P_1$&$\cdots$&$P_1$&$P_0$\\
\hline
\multirow{3}{*}{$I_{r-|X|}$}&\multirow{3}{*}{$D_{r-|X|}$}&$1\cdots1$&&&$1\cdots1$&&&\multirow{3}{*}{$0$}\\
&&&$\ddots$&&&$\ddots$&&\\
&&&&$1\cdots1$&&&$1\cdots1$&\\
\hline
\end{tabular}
\end{center}
\caption{Matrix conforming to $\Phi$}
\label{fig:universal-conform}
\end{figure}

\begin{figure}[ht]
\begin{center}
\begin{tabular}{|c|c|c|c|c|c|c|c|c|c|}
\hline
\multirow{4}{*}{$I_r$}&$0$&$I_{|X|}$&$\cdots$&$I_{|X|}$&$P_1$&$\cdots$&$P_1$&$P_1$&$P_0$\\
\cline{2-10}
&\multirow{3}{*}{$D_{r-|X|}$}&$1\cdots1$&&&$1\cdots1$&&&\multirow{3}{*}{$0$}&\multirow{3}{*}{$0$}\\
&&&$\ddots$&&&$\ddots$&&&\\
&&&&$1\cdots1$&&&$1\cdots1$&&\\
\hline
\end{tabular}
\end{center}
\caption{Matrix virtually conforming to $\Phi$}
\label{fig:universal}
\end{figure}

It will be helpful to expand the definition of $Y$-templates to partial fields.

\begin{definition}
\label{def:Y-Template-partial-field}
A \emph{$Y$-template} $\Phi=\YT(P_0,P_1)$ over a partial field $\P$ is a pair of matrices $P_0$ and $P_1$, with the same number of rows, that have entries from $\P$. A matrix $A$ with entries in $\P$ and the abstract matroid $\widetilde{M}(A)$ \emph{conform} to $\Phi$ if $A$ is a column submatrix of a matrix of the form given in Figure \ref{fig:universal-conform}; $A$ must contain the entire submatrix $P_0$ of the matrix shown in Figure \ref{fig:universal-conform}. Moreover, $A$ and $\widetilde{M}(A)$ \emph{virtually conform} to $\Phi$ if $A$ is a column submatrix of a matrix of the form given in Figure \ref{fig:universal}; $A$ must contain the entire submatrix $P_0$ of the matrix shown in Figure \ref{fig:universal}. A $Y$-template over a partial field is \emph{valid} if every matrix virtually conforming to the $Y$-template is a $\P$-matrix.
\end{definition}

We let $\mathcal{M}(\Phi)$ and $\mathcal{M}_v(\Phi)$ denote the sets of abstract matroids that conform and virtually conform, respectively, to a template $\Phi$ over a partial field $\P$. Note that we only define $Y$-templates over partial fields, not frame templates in general. In the case where $\Phi$ is a $Y$-template over a field, Definitions \ref{def:YT} and \ref{def:Y-Template-partial-field} are equivalent in the sense that a matrix conforms to $\Phi$ according to Definition \ref{def:YT} if and only if it conforms to $\Phi$ according to Definition \ref{def:Y-Template-partial-field}.

\begin{definition}
\label{def:universal}
The matrix shown in Figure \ref{fig:universal} is the \emph{rank-$r$ universal matrix} for $\YT(P_0,P_1)$, and the matroid it represents is the \emph{rank-$r$ universal matroid} for $\YT(P_0,P_1)$.
\end{definition}
Note that the rank-$r$ universal matroid of $\YT(P_0,P_1)$ need not be simple. For example, columns in $[I|P_1|P_0]$ can be scalar multiples of each other.

We make no attempt to define $\P$-represented matroids, where $\P$ is a partial field. This requires us to restate Definition \ref{def:equivalent} in terms of abstract $\P$-representable matroids.
\begin{definition}
\label{def:equivalent-PF} Let $\Phi$ and $\Phi'$ be frame templates over the partial fields $\mathbb{P}$ and $\mathbb{P}'$, respectively. The pair $\Phi,\Phi'$ are \emph{algebraically equivalent} if $\mathcal{M}(\Phi)=\mathcal{M}(\Phi')$. The pair $\Phi,\Phi'$ are \emph{strongly equivalent} if they are algebraically equivalent and if $\mathbb{P}=\mathbb{P}'$. The pair $\Phi,\Phi'$ are \emph{minor equivalent} if $\mathbb{P}=\mathbb{P}'$ and if the closures of $\mathcal{M}(\Phi)$ and $\mathcal{M}(\Phi')$ under the taking of minors are equal.
\end{definition}

Let $\Phi$ and $\Phi'$ be $Y$-templates over a field $\F$. If $\Phi$ and $\Phi'$ are strongly equivalent, minor equivalent, or algebraically equivalent in the sense of Definition \ref{def:equivalent}, then they are strongly equivalent, minor equivalent, or algebraically equivalent, respectively, in the sense of Definition \ref{def:equivalent-PF}.

We will need the next lemma to prove some of our results; it is a direct consequence of a result of Brylawski \cite[Theorem 3.11]{b75} (see also \cite[Proposition 6.9.11]{o11}).

\begin{lemma}
\label{lem:Kn-modular}
Let $G$ be a simple graph. The set of edges of a complete subgraph of $G$ is a modular flat of $M(G)$. 
\end{lemma}

\begin{lemma}
\label{lem:valid}
Let $\Phi=\YT(P_0,P_1)$ be a $Y$-template over a partial field $\P$, and let $P_0'=\left[
\begin{array}{c|c|c}
\multirow{2}{*}{$D$}&-P_1&P_0\\
\cline{2-3}
&I&0\\
\end{array}
\right]$. If $P_0'$ is a $\P$-matrix, then $\Phi$ is valid. Moreover, if $\Phi'=\YT(P_0',[\emptyset])$, then every matroid virtually conforming to $\Phi$ is a minor of a matroid virtually conforming to $\Phi'$.
\end{lemma}

\begin{proof}
It suffices to consider the rank-$r$ universal matrix $A$ virtually conforming to $\Phi$ because a matrix obtained from a $\P$-matrix by deleting columns is also a $\P$-matrix. First, consider the case where $P_1=[\emptyset]$; that is, $P_0'=[D|P_0]$. Let $X$ be the set of row indices of $P_0$. By row and column scaling and rearranging rows and columns (which are allowed by Proposition \ref{pro:preserve-P-matrix}), we may assume that the rank-$r$ universal matrix for $\Phi$ is of the following form:
\[\left[\begin{array}{c|c|c|c|c|c|c|c}
\multirow{3}{*}{$I_{r-|X|}$}&\multirow{3}{*}{0}&\multirow{3}{*}{$D_{r-|X|}$}&1\cdots1&&&\multirow{3}{*}{0}&\multirow{3}{*}{0}\\
&&&&\ddots&&&\\
&&&&&1\cdots1&&\\
\hline
0&I_{|X|}&0&-I_{|X|}&\cdots&-I_{|X|}&D_{|X|}&P_0\\
\end{array}\right]\]
Note that $[I_r|D_r]$, which represents the complete graphic matroid $M(K_{r+1})$, is a column submatrix of the matrix above. The matrix $[I_{|X|}|D_{|X|}]$ represents the complete graphic matroid $M(K_{|X|+1})$, which by Lemma \ref{lem:Kn-modular}, has an edge set that is a modular flat of $M(K_{r+1})$. Therefore, by Theorem \ref{thm:gen-par-con} (with $X_2=\emptyset$), and Lemma \ref{lem:identity}, the above matrix is a $\P$-matrix. Its vector matroid is the generalized parallel connection of $M(K_{r+1})$ with $\widetilde{M}([I_{|X|}|D_{|X|}|P_0])$ along $M(K_{|X|+1})$.

Now, we consider the general case. Let $m$ be the number of columns of $P_1$, and let $|X'|$ be the set of row indices of $P_0'$. Therefore, the column submatrix $D$ of $P_0'$ is $D_{|X'|}$. The rank-$r$ universal matrix $A$ of $\Phi$ can be obtained from the rank-$(r+m)$ universal matrix $A'$ of $\Phi'$ by deleting from $A'$ the columns necessary to transform $D_{|X'|}$ into 
$\left[\begin{array}{c}
D_{|X|}\\
\hline
0\\
\end{array}\right]$, by pivoting on every nonzero entry in the identity submatrix of $P_0'$, and by deleting from $A'$ the rows and columns in which those nonzero entries are contained. This is because this process transforms the submatrix $[I_{|X'|}|\cdots|I_{|X'|}]$ of $A'$ into the submatrix $[I_{|X|}|P_1|\cdots|I_{|X|}|P_1]$ of $A$. We have already seen that the result holds in the case where $P_1=[\emptyset]$; therefore, $A'$ is a $\P$-matrix. Thus, by Proposition \ref{pro:preserve-P-matrix}, $A$ is a $\P$-matrix. 

Every matroid virtually conforming to $\Phi$ is a restriction of a universal matroid $\widetilde{M}(A)$. The process described above shows that $\widetilde{M}(A)$ is a minor of $\widetilde{M}(A')$, obtained by deleting or contracting some elements represented by columns of $P_0'$.
\end{proof}

\section{Reducing a Template}
\label{sec:Reducing-a-Template}
In this section, we will introduce reductions and show that every template either reduces to one of several basic templates or is a $Y$-template. All templates in this section are over fields, rather than partial fields.

\begin{definition}
A \textit{reduction} is an operation on a frame template $\Phi$ that produces a frame template $\Phi'$ such that $\mathcal{M}(\Phi')\subseteq \mathcal{M}(\Phi)$.
\end{definition}

\begin{proposition} \label{reductions}
The following operations are reductions on a frame template $\Phi$:
\begin{itemize}
\item[(1)] Replace $\Gamma$ with a proper subgroup.
\item[(2)] Replace $\Lambda$ with a proper subgroup closed under multiplication by elements from $\Gamma$.
\item[(3)] Replace $\Delta$ with a proper subgroup closed under multiplication by elements from $\Gamma$.
\item[(4)] Remove an element $y$ from $Y_1$. (More precisely, replace $A_1$ with $A_1[X, Y_0\cup (Y_1-y)\cup C]$ and replace $\Delta$ with $\Delta|(Y_0\cup (Y_1-y)\cup C)$.)
\item[(5)] For all matrices $A'$ respecting $\Phi$, perform an elementary row operation on $A'[X, E]$. (Note that this alters the matrix $A_1$ and performs a change of basis on $\Lambda$.)
\item[(6)] If there is some element $x\in X$ such that, for every matrix $A'$ respecting $\Phi$, we have that $A'[\{x\},E]$ is a zero row vector, remove $x$ from $X$. (This simply has the effect of removing a zero row from every matrix conforming to $\Phi$.)
\item[(7)] Let $c\in C$ be such that $A_1[X,\{c\}]$ is a unit vector whose nonzero entry is in the row indexed by $x\in X$, and let either $\lambda_x=0$ for each $\lambda\in\Lambda$ or $\delta_c=0$ for each $\delta\in\Delta$. We contract $c$ from every matroid conforming to $\Phi$ as follows. Let $\Delta'$ be the result of adding $-\delta_cA_1[\{x\},Y_0\cup Y_1\cup C]$ to each element $\delta\in\Delta$. Replace $\Delta$ with $\Delta'$, and then remove $c$ from $C$ and $x$ from $X$. (More precisely, replace $A_1$ with $A_1[X-x, Y_0\cup Y_1\cup (C-c)]$, replace $\Lambda$ with $\Lambda|(X-x)$, and replace $\Delta$ with $\Delta'|(Y_0\cup Y_1\cup (C-c))$.)
\item[(8)] Let $c\in C$ be such that $A_1[X,\{c\}]$ is a zero column and $\delta_c=0$ for all $\delta\in\Delta$. Then remove $c$ from $C$. (More precisely, replace $A_1$ with $A_1[X, Y_0\cup Y_1\cup (C-c)]$, and replace $\Delta$ with $\Delta|(Y_0\cup Y_1\cup (C-c))$.)
\end{itemize}
\end{proposition}

\begin{proof}
Let $\Phi'$ be the template that results from performing one of operations (1)-(8) on $\Phi$.

For (1)-(3), every matrix $A'$ respecting $\Phi'$ also respects $\Phi$.

For (4), let $A'$ be a matrix respecting $\Phi'$, and let $M$ be the matroid $M(A)/C\backslash Y_1$, where $A$ is a matrix conforming to $\Phi'$ that has been constructed from $A'$ respecting $\Phi'$ as described in Section \ref{sec:Frame-Templates}. Since $Y_1$ is deleted to produce $M$, the only effect of $Y_1$ on $M$ is that  for each $i\in Z$ there exists $j\in Y_1$ such that the $i$-th column of $A$ is the sum of the $i$-th and the $j$-th columns of $A'$. But each $j\in Y_1$ in the template $\Phi'$ is also contained in $Y_1$ in the template $\Phi$. Therefore, $A$ conforms to $\Phi$, as does $M$.

For (5) and (6), let $A'$ be a representation matrix for a matroid $M$. Performing an elementary row operation on $A'$ or removing a zero row from $A'$ results in another representation matrix for $M$.

Operations (7)  and (8) have the effect of contracting $c$ from $M(A)\backslash Y_1$ for every matrix $A$ conforming to $\Phi$. Since all of $C$ is contracted to produce a matroid $M$ conforming to $\Phi$, the matroids we produce by performing either of these operations still conform to $\Phi$.
\end{proof}

Since $Y_0\subseteq E(M)$ for every matroid $M$ conforming to $\Phi$, operations (10)--(12) listed in the definition below are not reductions as defined above, but we continue the numbering from Proposition \ref{reductions} for ease of reference.

\begin{definition}
\label{def:weaklyconforming}
A template $\Phi'$ is a \textit{template minor} of $\Phi$ if $\Phi'$ is obtained from $\Phi$ by repeatedly performing the following operations:
\begin{itemize}
\item[(9)] Performing one of reductions (1)--(8) on $\Phi$.
\item[(10)] Removing an element $y$ from $Y_0$, replacing $A_1$ with $A_1[X,(Y_0-y)\cup Y_1\cup C]$, and replacing $\Delta$ with $\Delta|((Y_0-y)\cup Y_1\cup C)$. (This has the effect of deleting $y$ from every matroid conforming to $\Phi$.)
\item[(11)] Let $x\in X$ with $\lambda_x=0$ for every $\lambda\in\Lambda$, and let $y\in Y_0$ be such that $(A_1)_{x,y}\neq0$. Then contract $y$ from every matroid conforming to $\Phi$. (More precisely, perform row operations on $A_1$ so that $A_1[X, \{y\}]$ is a unit column with $(A_1)_{x,y}=1$. Then replace every element $\delta\in\Delta$ with the row vector $-\delta_y A_1[\{x\}, Y_0\cup Y_1\cup C]+\delta$. This induces a group homomorphism $\Delta\rightarrow\Delta'$, where $\Delta'$ is also a subgroup of the additive group of $\mathbb{F}^{C\cup Y_0 \cup Y_1}$ and is closed under scaling by elements of $\Gamma$. Finally, replace $A_1$ with $A_1[X-x,(Y_0-y)\cup Y_1\cup C]$, project $\Lambda$ into $\mathbb{F}^{X-x}$, and project $\Delta'$ into $\mathbb{F}^{(Y_0-y)\cup Y_1\cup C}$. The resulting groups play the roles of $\Lambda$ and $\Delta$, respectively, in $\Phi'$.)
\item[(12)] Let $y\in Y_0$ with $\delta_y=0$ for every $\delta\in\Delta$. Then contract $y$ from every matroid conforming to $\Phi$. (More precisely, if $A_1[X, \{y\}]$ is a zero vector, this is the same as simply removing $y$ from $Y_0$. Otherwise, choose some $x\in X$ such that $(A_1)_{x,y}\neq0$. Then for every matrix $A'$ that respects $\Phi$, perform row operations so that $A_1[X,\{y\}]$ is a unit column with $(A_1)_{x,y}=1$. This induces a group isomorphism $\Lambda\rightarrow\Lambda'$ where $\Lambda'$ is also a subgroup of the additive group of $\mathbb{F}^X$ and is closed under scaling by elements of $\Gamma$. Finally, replace $A_1$ with $A_1[X-x,(Y_0-y)\cup Y_1\cup C]$, project $\Lambda'$ into $\mathbb{F}^{X-x}$, and project $\Delta$ into $\mathbb{F}^{(Y_0-y)\cup Y_1\cup C}$. The resulting groups play the roles of $\Lambda$ and $\Delta$, respectively in $\Phi'$.)
\end{itemize}
\end{definition}

Recall the definitions of virtual respecting and conforming from Definition \ref{virtual}. Let $\Phi'$ be a template minor of $\Phi$, and let $A'$ be a matrix that virtually respects $\Phi'$. Let $A$ be a matrix that virtually conforms to $\Phi'$, and let $M$ be a matroid that virtually conforms to $\Phi'$. We say that $A'$ \textit{weakly respects} $\Phi$ and that $A$ and $M$ \textit{weakly conform} to $\Phi$. Let $\mathcal{M}_w(\Phi)$ denote the set of matroids that weakly conform to $\Phi$, and let $\mathcal{M}^*_w(\Phi)$ denote the set of matroids whose duals weakly conform to $\Phi$. If $M\in\mathcal{M}^*_w(\Phi)$, we say that $M$ \textit{weakly coconforms} to $\Phi$. The next lemma is \cite[Lemma 3.4]{gvz17}.

\begin{lemma} \label{minor}
 If a matroid $M$ weakly conforms to a template $\Phi$, then $M$ is a minor of a matroid that conforms to $\Phi$.
\end{lemma}

An easy consequence of Lemma \ref{minor} is that Hypotheses \ref{hyp:connected-template} and \ref{hyp:clique-template}, which deal with minor-closed classes, can be restated in terms of weak conforming. The proofs of the next two corollaries are essentially the same as that of \cite[Corollary 3.5]{gvz17} and are therefore omitted.

\begin{corollary}
\label{cor:weak-connected-template}
Suppose Hypothesis \ref{hyp:connected-template} holds. Let $\mathbb F$ be a finite field, let $m$ be a positive integer, and let 
$\mathcal M$ be a minor-closed class of $\mathbb F$-represented matroids.
Then there exist $k\in\mathbb{Z}_+$ and refined
frame templates $\Phi_1,\ldots,\Phi_s,\Psi_1,\ldots,\Psi_t$ such that
\begin{enumerate}
\item $\mathcal{M}$ contains each of the classes $\mathcal{M}_w(\Phi_1),\dots,\mathcal{M}_w(\Phi_s)$,
\item $\mathcal{M}$ contains the duals of the represented matroids in each of the classes $\mathcal{M}_w(\Psi_1)$,$\dots$,$\mathcal{M}_w(\Psi_t)$, and
\item if $M$ is a simple $k$-connected member of $\mathcal M$ with at least $2k$ elements and $\widetilde{M}$ has no $\PG(m-1,\bFp)$-minor, then either $M$ is a member of at least one of the classes $\mathcal{M}_w(\Phi_1),\ldots,\mathcal{M}_w(\Phi_s)$, or $M^*$ is a member of at least one of the classes $\mathcal{M}_w(\Psi_1),\ldots,\mathcal{M}_w(\Psi_t)$.
\end{enumerate}
\end{corollary}

\begin{corollary}
\label{cor:weak-cliquetemplate}
Suppose Hypothesis \ref{hyp:clique-template} holds. Let $\mathbb F$ be a finite field, let $m$ be a positive integer, and let $\mathcal M$ be a minor-closed class of $\mathbb F$-represented matroids. Then there exist $k,n\in\mathbb{Z}_+$ and refined frame templates $\Phi_1,\ldots,\Phi_s,\Psi_1,\ldots,\Psi_t$ such that
\begin{enumerate}
\item $\mathcal{M}$ contains each of the classes $\mathcal{M}_w(\Phi_1),\ldots,\mathcal{M}_w(\Phi_s)$,
\item $\mathcal{M}$ contains the duals of the represented matroids in each of the classes $\mathcal{M}_w(\Psi_1),\ldots,\mathcal{M}_w(\Psi_t)$,
\item if $M$ is a simple vertically $k$-connected member of $\mathcal M$ with an $M(K_n)$-minor but no $\mathrm{PG}(m-1,\bFp)$-minor, then $M$ is a member of at least one of the classes $\mathcal{M}_w(\Phi_1),\ldots,\mathcal{M}_w(\Phi_s)$, and
\item if $M$ is a cosimple cyclically $k$-connected member of $\mathcal M$ with an $M^*(K_n)$-minor but no $\mathrm{PG}(m-1,\bFp)$-minor, then $M^*$ is a member of at least one of the classes $\mathcal{M}_w(\Psi_1),\ldots,\mathcal{M}_w(\Psi_t)$.
\end{enumerate}
\end{corollary}

If $\mathcal{M}_{w}(\Phi)=\mathcal{M}_{w}(\Phi')$, we say that $\Phi$ is \textit{equivalent} to $\Phi'$ and write $\Phi\sim\Phi'$. It is clear that $\sim$ is indeed an equivalence relation. We define a preorder $\preceq$ on the set of frame templates over a field $\mathbb{F}$ as follows.

\begin{definition}
\label{def:preorder}
We say $\Phi\preceq\Phi'$ if $\mathcal{M}_w(\Phi)\subseteq\mathcal{M}_w(\Phi')$. This is indeed a preorder since reflexivity and transitivity follow from the subset relation.
\end{definition}


In the remainder of this section, we find a collection of templates that are minimal with respect to $\preceq$ among those templates that are not $Y$-templates.

\begin{definition}
\label{def:minimal}
We define the following templates for a field $\mathbb{F}=\GF(p^m)$.
\begin{itemize}
\item Let $\Phi_C$  be the template over $\mathbb{F}$ with all groups trivial and all sets empty except that $|C|=1$ and $\Delta\cong\mathbb{Z}/p\mathbb{Z}$.
\item Let $\Phi_X$  be the template over $\mathbb{F}$ with all groups trivial and all sets empty except that $|X|=1$ and $\Lambda\cong\mathbb{Z}/p\mathbb{Z}$.
\item Let $\Phi_{Y_0}$  be the template over $\mathbb{F}$ with all groups trivial and all sets empty except that $|Y_0|=1$ and $\Delta\cong\mathbb{Z}/p\mathbb{Z}$. 
\item For each $k\in(\mathbb{F}-\{0\})$, let $\Phi_{CXk}$ be the template with $Y_0=Y_1=\emptyset$, with $|C|=|X|=1$, with $\Delta\cong\Lambda\cong\mathbb{Z}/p\mathbb{Z}$, with $\Gamma$ trivial, and with $A_1=[k]$. We abbreviate $\Phi_{CX1}$ to $\Phi_{CX}$.
\item If $n$ is a prime dividing $p^m-1$, let $\Phi_n$ be the template with all sets empty and all groups trivial except that $\Gamma$ is the cyclic subgroup of $\mathbb{F}^{\times}$ of order $n$.
\end{itemize}
\end{definition}
Note that each of these templates is refined. 

Our goal in defining reductions and weak conforming is essentially to perform operations on matrices while leaving the $\Gamma$-frame submatrix intact. The following lemma does not contribute to that goal, so we will only make occasional use of it.

\begin{lemma}
\label{YCD}
For every field $\F$ and $k\in(\F-\{0\})$, we have $\Phi_{Y_0}\preceq\Phi_C$ and $\Phi_X\preceq\Phi_{CXk}$.
\end{lemma}

\begin{proof}
First we prove $\Phi_{Y_0}\preceq\Phi_{C}$. A matroid $M$ conforming to $\Phi_{Y_0}$ is the vector matroid of a matrix of the following form, where $\bar{v}$ is a column vector all of whose entries are contained in $\bFp$:
\begin{center}
\begin{tabular}{ |c|c|}
\hline
&\\
$\{1\}$-frame matrix&$\bar{v}$\\
&\\
\hline
\end{tabular}
\end{center}
Let $A$ be the matrix below. Label its sets of rows and columns as $B$ and $E$ respectively, and let $c$ be the last column, with $C=\{c\}$.
\begin{center}
\begin{tabular}{|c|c|c|}
\hline
0&1&1\\
\hline
&&\\
$\{1\}$-frame matrix&0&$-\bar{v}$\\
&&\\
\hline
\end{tabular}
\end{center}
Note that $M$ is isomorphic to $M(A)/C$. Since $A[B,E-C]$ is a $\{1\}$-frame matrix, we see that $M$ conforms to $\Phi_C$.

Now we prove $\Phi_X\preceq\Phi_{CXk}$. Every matroid $M$ conforming to $\Phi_X$ is the vector matroid of a matrix of the following form, where $\bar{v}$ is a row vector all of whose entries are contained in $\bFp$:
\begin{center}
\begin{tabular}{|c|}
\hline
$\bar{v}$\\
\hline
\\
$\Gamma$-frame matrix\\
\\
\hline
\end{tabular}
\end{center}
Consider the following matrix $A$, whose last column is indexed by $\{c\}=C$:
\begin{center}
\begin{tabular}{|c|c|}
\hline
$\bar{v}$&$k$\\
\hline
$0$&$1$\\
\hline
&\\
$\Gamma$-frame matrix&0\\
&\\
\hline
\end{tabular}
\end{center}
The matroid $M$ is isomorphic to $M(A)/c$, which conforms to $\Phi_{CXk}$.
\end{proof}

\begin{lemma} \label{yshift}
Let $\Phi$ be a template with $y\in Y_1$. Let $\Phi'$ be the template obtained from $\Phi$ by removing $y$ from $Y_1$ and placing it in $Y_0$. Then $\Phi'\preceq \Phi$.
\end{lemma}

\begin{proof}
Every matrix respecting $\Phi'$ virtually respects $\Phi$ since a matrix virtually respecting $\Phi$ can have zero columns in $Z$. Thus, every matroid conforming to $\Phi'$ virtually conforms to $\Phi$.
\end{proof}

We call the operation described in Lemma \ref{yshift} a \textit{$y$-shift}.

\begin{definition}
\label{def:standard-form}
 Let $\Phi=(\Gamma,C,X,Y_0,Y_1,A_1,\Delta,\Lambda)$ be a frame template over a finite field $\mathbb{F}$. We say that $\Phi$ is in \textit{standard form} if there are disjoint sets $C_0,C_1,X_0,$ and $X_1$ such that $C=C_0\cup C_1$, such that $X=X_0\cup X_1$, such that $A_1[X_0,C_0]$ is an identity matrix, and such that $A_1[X_1,C]$ is a zero matrix.
\end{definition}

{\color{red}Note that this partition $X=X_0\cup X_1$ is not the same as the reduction partition used to define a reduced template (see Definition \ref{def:reduced}). Roughly speaking, a reduced template is obtained by performing row operations to obtain an identity matrix within the columns contained in $\Lambda$, while a template in standard form is obtained by performing row operations to obtain an identity matrix whose columns are indexed by some subset of $C$.}

Figure \ref{fig:A' standard}, with the stars representing arbitrary matrices, shows a matrix that virtually respects a template in standard form. Note that if $\Phi$ is in standard form, $|C_0|=|X_0|$. Also note that any of $C_0,C_1,X_0,$ or $X_1$ may be empty.

\begin{figure}[ht]
\begin{center}
\begin{tabular}{ r|c|c|cccc| }
\multicolumn{2}{c}{}&\multicolumn{1}{c}{$Z$}&\multicolumn{1}{c}{$Y_0$}&\multicolumn{1}{c}{$Y_1$}&\multicolumn{1}{c}{$C_0$}&\multicolumn{1}{c}{$C_1$}\\
\cline{2-7}
$X_0$&columns from $\Lambda|X_0$&0&\multicolumn{2}{c|}{\multirow{2}{*}{$*$}}&\multicolumn{1}{c|}{$I$}&$*$\\
\cline{2-3} \cline{6-7}
$X_1$&columns from $\Lambda|X_1$&0&&&\multicolumn{2}{|c|}{0}\\
\cline{2-7}
&\multirow{5}{*}{$\Gamma$-frame matrix}&\multirow{5}{*}{unit or zero columns}&\multicolumn{4}{c|}{\multirow{5}{4em}{rows from  $\Delta$}}\\
&&&&&&\\
&&&&&&\\
&&&&&&\\
&&&&&&\\
\cline{2-7}
\end{tabular}
\end{center}
\caption{Standard Form}
  \label{fig:A' standard}
\end{figure}

\begin{lemma}
\label{standard}
 Every frame template $\Phi=(\Gamma,C,X,Y_0,Y_1,A_1,\Delta,\Lambda)$ is {\color{red}strongly} equivalent to a frame template in standard form.
\end{lemma}

\begin{proof}
 Choose a basis $C_0$ for $M(A_1[X,C])$, and let $C_1=C-C_0$. Repeatedly perform operation (5) to obtain a template $\Phi'$ where $A_1[X,C_0]$ consists of an identity matrix on top of a zero matrix. Each use of operation (5) results in {\color{red}a strongly} equivalent template; therefore, $\Phi$ {\color{red}is strongly equivalent to} $\Phi'$. Let $X_0\subseteq X$ index the rows of the identity matrix, and let $X_1\subseteq X$ index the rows of the zero matrix. Since $C_0$ is a basis for $M(A_1[X,C])$, the matrix $A_1[X,C_1]$ must be a zero matrix as well. Thus, $\Phi'$ is in standard form.
\end{proof}

In the remainder of this section, we will implicitly use Lemma \ref{standard} to assume that all templates are in standard form. Also, the operations (1)--(12) to which we will refer throughout the rest of the paper are the operations (1)--(8) from Proposition \ref{reductions} and (9)--(12) from Definition \ref{def:weaklyconforming}.

\begin{lemma}
\label{lem:Phin}
Let $\Phi=(\Gamma,C,X,Y_0,Y_1,A_1,\Delta,\Lambda)$ be a frame template over $\mathbb{F}\cong\GF(p^m)$ for some prime $p$. If $\Gamma$ is nontrivial, then $\Phi_n\preceq\Phi$ for some prime $n$ dividing $p^m-1$.
\end{lemma}

\begin{proof}
Perform operations (2) and (3) on $\Phi$ to obtain the following template:
\[(\Gamma,C,X,Y_0,Y_1,A_1,\{0\},\{0\}).\]  On this template, repeatedly perform operation (7), then (8), then (4), and then (11) to obtain the following template:
\[(\Gamma,\emptyset,X_1,\emptyset,\emptyset,[\emptyset],\{0\},\{0\}).\] Now, on this template, repeatedly perform operation (6) to obtain the following template:
\[\Phi'=(\Gamma,\emptyset,\emptyset,\emptyset,\emptyset,[\emptyset],\{0\},\{0\}).\] Since $\Gamma$ is a subgroup of the multiplicative group of a field, it is a cyclic group. If $|\Gamma|$ is a prime $n$, then $\Phi'=\Phi_n$ and we are finished. Otherwise, let $n$ be a prime dividing $|\Gamma|$ and let $\alpha$ be a generator for $\Gamma$. Then the subgroup $\Gamma'$ generated by $\alpha^{|\Gamma|/n}$ has order $n$. Perform operation (1) on $\Phi'$ to obtain the template
\[(\Gamma',\emptyset,\emptyset,\emptyset,\emptyset,[\emptyset],\{0\},\{0\}),\] which is $\Phi_n$.
\end{proof}

\begin{lemma}
 \label{PhiD}
If $\Phi=(\Gamma,C,X,Y_0,Y_1,A_1,\Delta,\Lambda)$ is a frame template over $\mathbb{F}\cong\GF(p^m)$ for some prime $p$ with $\Lambda|X_1$ nontrivial, then $\Phi_X\preceq\Phi$.
\end{lemma}

\begin{proof}
Perform operations (1), (2), and (3) on $\Phi$ to obtain the following template, where $\lambda$ is an element of $\Lambda$ with $\lambda_x\neq0$ for some $x\in X_1$:
\[(\{1\},C,X,Y_0,Y_1,A_1,\{0\},\langle\lambda\rangle).\] On this template, repeatedly perform operation (7), then (8), then (4), and then (10) until the following template is obtained: \[(\{1\},\emptyset,X_1,\emptyset,\emptyset,[\emptyset],\{0\},\langle\lambda|X_1\rangle).\] On this template, repeatedly perform operation (5) to obtain a template that is identical to the previous one except that the support of $\lambda|X_1$ contains only one element of $X_1$. On this template, repeatedly perform operation (6) to obtain the following template, where $x\in X_1$: \[(\{1\},\emptyset,\{x\},\emptyset,\emptyset,[\emptyset],\{0\},\mathbb{Z}/p\mathbb{Z}).\] This template is $\Phi_X$.
\end{proof}

\begin{lemma}
\label{PhiC}
 If $\Phi=(\Gamma,C,X,Y_0,Y_1,A_1,\Delta,\Lambda)$ is a frame template over $\mathbb{F}\cong\GF(p^m)$ for some prime $p$, then either $\Phi_C\preceq\Phi$ or $\Phi$ is {\color{red}strongly} equivalent to a template with $C_1=\emptyset$.
\end{lemma}

\begin{proof}
Suppose there is an element $\delta\in\Delta|C$ that is not in the row space of $A_1[X,C]$. Repeatedly perform operations (4) and (10) on $\Phi$ until the following template is obtained:
\[(\Gamma,C,X,\emptyset,\emptyset,A_1[X,C],\Delta|C,\Lambda).\]
On this template, perform operations (1), (2), and (3) to obtain the following template:
\[(\{1\},C,X,\emptyset,\emptyset,A_1[X,C],\langle\delta\rangle,\{0\}).\]
Every matrix virtually respecting this template is row equivalent to a matrix virtually respecting a template that is identical to the previous template except that there is the additional condition that $\delta|C_0$ is a zero vector. Note that $\delta|C_1$ is nonzero since, in the previous template, $\delta$ was not in the row space of $A_1[X,C]$.
Now, on the current template, repeatedly perform operation (7) and then operation (6) to obtain the following template:
\[\Phi'=(\{1\},C_1,\emptyset,\emptyset,\emptyset,[\emptyset],\langle\delta|C_1\rangle,\{0\}).\]

Now, every matroid $M$ conforming to $\Phi'$ is obtained by contracting $C_1$ from $M(A)$, where $A$ is a matrix conforming to $\Phi'$. By contracting any single element $c\in C_1$, where $\delta_c\neq0$, we turn the rest of the elements of $C_1$ into loops. So $C_1-c$ is deleted to obtain $M$. Thus, $M$ conforms to the template
\[(\{1\},\{c\},\emptyset,\emptyset,\emptyset,[\emptyset],\mathbb{Z}/p\mathbb{Z},\{0\}),\]
which is $\Phi_C$. Similarly, the converse is true that any matroid conforming to $\Phi_C$ conforms to $\Phi'$. Thus, $\Phi_C\sim\Phi'\preceq\Phi$.

Now suppose that every element of $\Delta|C$ is in the row space of $A_1[X,C]$. Thus, contraction of $C_0$ turns the elements of $C_1$ into loops, and contraction of $C_1$ is the same as deletion of $C_1$. By deleting $C_1$ from every matrix virtually conforming to $\Phi$, we see that $\Phi$ is {\color{red}strongly} equivalent to a template with $C_1=\emptyset$.
\end{proof}

\begin{lemma}
\label{PhiCD}
 If $\Phi=(\Gamma,C,X,Y_0,Y_1,A_1,\Delta,\Lambda)$ is a frame template over $\F\cong\GF(p^m)$, then one of the following is true:
\begin{itemize}
\item $\Phi_C\preceq\Phi$
\item $\Phi$ is {\color{red}strongly} equivalent to a template with $\Lambda|X_1$ nontrivial and $\Phi_X\preceq\Phi$
\item $\Phi$ is {\color{red}strongly} equivalent to a template with $\Lambda|X_0$ nontrivial and $\Phi_{CXk}\preceq\Phi$ for some $k\in\mathbb{F}-\{0\}$
\item $\Phi$ is {\color{red}strongly} equivalent to a template with $\Lambda$ trivial and $C=\emptyset$.
\end{itemize}
\end{lemma}

\begin{proof}
By Lemmas \ref{PhiD} and \ref{PhiC}, we may assume that $\Lambda|X_1$ is trivial and that $C_1=\emptyset$. 

First, suppose there exist elements $\delta\in\Delta|C_0$ and $\lambda\in\Lambda|X_0$ such that $\sum \delta_i\lambda_i=k^{-1}\neq0$. Thus, $\Lambda|X_0$ is nontrivial. {\color{red}Perform operation (1) to obtain the following template:
\[(\{1\},C,X,Y_0,Y_1,A_1,\Delta,\Lambda).\]  Then,} repeatedly perform operations (4) and (10) on $\Phi$ until the following template is obtained:
\[(\{1\},C_0,X,\emptyset,\emptyset,A_1[X,C_0],\Delta|C_0,\Lambda).\]
On this template, repeatedly perform operation (6) to obtain the following template:
\[\Phi'=(\{1\},C_0,X_0,\emptyset,\emptyset,A_1[X_0,C_0],\Delta|C_0,\Lambda|X_0).\]
 Perform operations (2) and (3) on $\Phi'$ to obtain the following template:
\[(\{1\},C_0,X_0,\emptyset,\emptyset,A_1[X_0,C_0],\langle\delta\rangle,\langle\lambda\rangle).\]
A matroid conforming to this template is obtained by contracting $C_0$. Let $a,b\in\bFp$. If $a\delta$ is in the row labeled by $r$ and $b\lambda$ is in the column labeled by $c$, then when $C_0$ is contracted, $-abk$ is added to the entry of the $\Gamma$-frame matrix in row $r$ and column $c$. We see then that this template is {\color{red}strongly} equivalent to $\Phi_{CXk}$, where $b$ is used to replace $b\lambda$ and $a$ is used to replace $a\delta$ .

Thus, we may assume that for every element  $\delta\in\Delta|C_0$ and $\lambda\in\Lambda|X_0$, we have $\sum \delta_i\lambda_i=0$. This implies that contraction of $C$ has no effect on the $\Gamma$-frame matrix. So $\Phi$ is {\color{red}strongly} equivalent to a template with $\Lambda|X_0$ trivial. Therefore, since $\Lambda|X_1$ is trivial, we see that $\Lambda$ is trivial. Note that operation (7) is a reduction that produces {\color{red}a strongly} equivalent template, since $C$ must be contracted to produce a matroid that conforms to a template. By repeatedly performing operation (7), we obtain a template {\color{red}strongly} equivalent to $\Phi$ with $C=\emptyset$.
\end{proof}

\begin{lemma}
 \label{PhiY0}
If $\Phi=(\Gamma,C,X,Y_0,Y_1,A_1,\Delta,\Lambda)$ is a frame template over $\F\cong\GF(p^m)$ with $\Lambda$ trivial and with $C=\emptyset$, then either $\Phi_{Y_0}\preceq\Phi$ or $\Phi$ is {\color{red}strongly} equivalent to a template with $\Delta$ trivial.
\end{lemma}

\begin{proof}
We first consider the case where there is an element $\delta\in\Delta$ that is not in the row space of $A_1=A_1[X_1,(Y_0\cup Y_1)]$. {\color{red}Perform operation (1) to obtain the following template:
\[(\{1\},C,X,Y_0,Y_1,A_1,\Delta,\{0\}).\]} Recall that a $y$-shift is the operation described in Lemma \ref{yshift}. Repeatedly perform $y$-shifts to obtain the following template, where $Y'_0=Y_0\cup Y_1$:
\[(\{1\},\emptyset,X,Y'_0,\emptyset,A_1,\Delta,\{0\}).\]
On this template, perform operation (3) to obtain the following template:
\[(\{1\},\emptyset,X,Y'_0,\emptyset,A_1,\langle\delta\rangle,\{0\}).\]

Choose a basis $B'$ for $M(A_1)$. By performing elementary row operations on every matrix virtually respecting $\Phi$, we may assume that $A_1[X,B']$ consists of an identity matrix with zero rows below it and that $\delta|B'$ is the zero vector. By assumption, there is some element $y\in (Y'_0-B')$ such that $\delta_y$ is nonzero. Thus, we can repeatedly perform operation (10) to obtain the following template:
\[(\{1\},\emptyset,X,B'\cup\{y\},\emptyset,A_1[X,B'\cup\{y\}],\langle\delta|(B'\cup\{y\}\rangle,\{0\}).\]
Now, we can repeatedly perform operation (6) and then operation (12) to obtain the following template, which is $\Phi_{Y_0}$:
\[(\{1\},\emptyset,\emptyset,\{y\},\emptyset,[\emptyset],\bFp,\{0\}).\]

Now, suppose that every element $\delta\in\Delta$ is in the row space of $A_1=A_1[X,(Y_0\cup Y_1)]$. Since $\Lambda$ is trivial, by performing elementary row operations on every matrix virtually respecting $\Phi$, we obtain a template {\color{red}strongly} equivalent to $\Phi$ with $\Delta$ trivial.
\end{proof}

\begin{theorem}
\label{thm:reduce-to-Y-template}
Let $\Phi$ be a refined frame template over a finite field $\mathbb{F}=\GF(p^m)$. Then one of the following is true:
\begin{itemize}
\item[(i)]$\Phi'\preceq \Phi$ for some $\Phi'\in\{\Phi_X,\Phi_C,\Phi_{Y_0}\}$ $\cup$ $\{\Phi_{CXk}:k\in\F-\{0\}\}$ $\cup$ $\{\Phi_n:n$ is a prime dividing $p^m-1\}$
\item[(ii)]$\Phi$ is a $Y$-template.
\end{itemize}
\end{theorem}

\begin{proof}
Suppose (i) does not hold. By Lemma \ref{lem:Phin}, we may assume that $\Gamma$ is trivial. By Lemma \ref{PhiCD}, we may assume that $\Lambda$ is trivial. By Lemma \ref{PhiY0}, we may assume that $\Delta$ is trivial. By Definition \ref{def:Y-template}, these facts, combined with the assumption that $\Phi$ is refined, imply that $\Phi$ is a $Y$-template.
\end{proof}

\section{Excluded Minors}
\label{sec:Excluded Minors}

In this section, we prove results that will be used in the proofs of Theorems \ref{thm:AC4exc-minor}--\ref{thm:AF4exc-minor}. Let $\mathcal{M}$ be a minor-closed class of $\mathbb{F}$-represented matroids, where $\bFp$ is the prime subfield of $\F$, and let $\mathcal{E}_1$ and $\mathcal{E}_2$ be two sets of $\mathbb{F}$-represented matroids. The following conditions will be used as hypotheses for Lemmas \ref{lem:connected-excluded-minors}--\ref{lem:coclique-excluded-minors}.

\begin{itemize}
\item[(i)] No member of $\mathcal{E}_1\cup\mathcal{E}_2$ is contained in $\mathcal{M}$.
\item[(ii)] For every refined frame template $\Phi$ over $\F$ such that $\mathcal{M}(\Phi)\nsubseteq\mathcal{M}$, there is a member of $\mathcal{E}_1$ that is a minor of a matroid conforming to $\Phi$.
\item[(iii)] For every refined frame template $\Psi$ over $\F$ such that $\mathcal{M}^*(\Psi)\nsubseteq\mathcal{M}$, there is a member of $\mathcal{E}_2$ that is a minor of a matroid coconforming to $\Psi$.
\item[(iv)] For some member $M$ of $\mathcal{E}_1$, the abstract matroid $\widetilde{M}$ is $\bFp$-representable.
\item[(v)] For some member $M$ of $\mathcal{E}_2$, the abstract matroid $\widetilde{M}$ is $\bFp$-representable.
\end{itemize}

\begin{lemma}
\label{lem:connected-excluded-minors}
Suppose Hypothesis \ref{hyp:connected-template} holds. Also, let (i)--(iii) and either (iv) or (v) above hold. There exists $k\in\mathbb{Z}_+$ such that a $k$-connected $\F$-represented matroid with at least $2k$ elements is contained in $\mathcal{M}$ if and only if it contains no minor isomorphic to one of the matroids in the set $\mathcal{E}_1\cup\mathcal{E}_2$.
\end{lemma}

\begin{proof}
Let $\mathcal{M}'$ be the class of $\F$-represented matroids that contain no minor isomorphic to a member of $\mathcal{E}_1\cup\mathcal{E}_2$. Because of (i), we have $\mathcal{M}\subseteq\mathcal{M}'$. Because of (iv) or (v), there is some integer $m$ such that no member of $\mathcal{M}'$ has a $\PG(m-1,\bFp)$-minor. By Corollary \ref{cor:weak-connected-template}, there is an integer $k$ and set of refined frame templates $\{\Phi_1,\dots\Phi_s,\Psi_1,\dots,\Psi_t\}$ such that every $k$-connected $\F$-represented member $M$ of $\mathcal{M}'$ with at least $2k$ elements either weakly conforms to some member of $\{\Phi_1,\dots\Phi_s\}$ or weakly coconforms to some member of $\{\Psi_1,\dots,\Psi_t\}$. By (ii) and (iii), $M\in\mathcal{M}$.
\end{proof}

\begin{lemma}
\label{lem:clique-excluded-minors}
Suppose Hypothesis \ref{hyp:clique-template} holds. Also, let (i), (ii), and (iv) above hold. There exist $k,n\in\mathbb{Z}_+$ such that a vertically $k$-connected $\F$-represented matroid with an $M(K_n)$-minor is contained in $\mathcal{M}$ if and only if it contains no minor isomorphic to one of the matroids in $\mathcal{E}_1$.
\end{lemma}

\begin{proof}
Let $\mathcal{M}'$ be the class of $\F$-represented matroids that contain no minor isomorphic to a member of $\mathcal{E}_1$. Because of (i), we have $\mathcal{M}\subseteq\mathcal{M}'$. Because of (iv), there is some integer $m$ such that no member of $\mathcal{M}'$ has a $\PG(m-1,\bFp)$-minor. By Corollary \ref{cor:weak-cliquetemplate}, there are positive integers $k,n$ and refined frame templates $\{\Phi_1,\dots\Phi_s\}$ such that every vertically $k$-connected $\F$-represented member $M$ of $\mathcal{M}'$ with an $M(K_n)$-minor weakly conforms to some member of $\{\Phi_1,\dots\Phi_s\}$. By (ii), $M\in\mathcal{M}$.
\end{proof}

We omit the proof of the next lemma which is similar to that of Lemma \ref{lem:clique-excluded-minors}.

\begin{lemma}
\label{lem:coclique-excluded-minors}
Suppose Hypothesis \ref{hyp:clique-template} holds. Also, let (i), (iii), and (v) above hold. There exist $k,n\in\mathbb{Z}_+$ such that a cyclically $k$-connected $\F$-represented matroid with an $M^*(K_n)$-minor is contained in $\mathcal{M}$ if and only if it contains no minor isomorphic to one of the matroids in $\mathcal{E}_2$.
\end{lemma}

We remark that Lemmas \ref{lem:clique-excluded-minors} and \ref{lem:coclique-excluded-minors} provide an improvement to the results in Section 9 of \cite{gvz19}, although those results are not incorrect. Namely, the matroids $L_{19}$, $M^*(K_6)$, and $M(K_6)$ can be removed form the statements of Theorems 9.2--9.4, respectively, of \cite{gvz19}. Consequently, $M(K_6)$ can also be removed form the statement of Corollary 9.5 of \cite{gvz19}.

\section{\textit{Y}-Templates}
\label{sec:Y-Templates}

Recall the definition of a $Y$-template from Definition \ref{def:Y-Template-partial-field}. In this section, we study $Y$-templates in more detail. We use the language of templates in order to apply Hypotheses \ref{hyp:connected-template} and \ref{hyp:clique-template}. However, we will see that the study of $Y$-templates amounts to a study of representable matroids with spanning cliques, independent of Hypotheses \ref{hyp:connected-template} and \ref{hyp:clique-template}. We begin with a lemma, proved in an earlier paper, that applies to a larger class of templates.

\begin{lemma}[{\cite[{Lemma 3.17}]{gvz17}}]
 \label{simpleY1}
If $\Phi$ is a frame template with $\Delta$ trivial, then $\Phi$ is equivalent to a template $\Phi'$ where $A_1[X,Y_1]$ is a matrix where every column is nonzero and where no column is a copy of another.
\end{lemma}

If two $Y$-templates have isomorphic universal matroids (Definition \ref{def:universal}), one might expect for them to be strongly equivalent (Definition \ref{def:equivalent-PF}). However, if a matroid $M$ conforms to a template $(\Gamma,C,X,Y_0,Y_1,A_1,\Delta,\Lambda)$, then we always have $Y_0\subseteq E(M)$. Some elements of the common universal matroid might be contained in $Y_0$ when thought of as conforming to one template but not the other template. This technicality motivates the next definition.

\begin{definition}
\label{def:semi-strong}
Let $\Phi=\YT(P_0,P_1)$ and $\Phi'=\YT(P_0',P_1')$ be $Y$-templates over a partial field $\P$. Let $X$ and $X'$ be the sets of row indices for $[P_0|P_1]$ and $[P_0'|P_1']$, respectively. If, for every $r\geq\max\{|X|,|X'|\}$, the rank-$r$ universal matroids of $\Phi$ and $\Phi'$ are isomorphic, then $\Phi$ and $\Phi'$ are \emph{semi-strongly equivalent}.
\end{definition}
Since every matroid conforming to a $Y$-template is a restriction of a universal matroid for the template, semi-strong equivalence implies minor equivalence.

\begin{definition}
\label{def:lifted}
 A refined template  $\Phi=(\Gamma,C,X,Y_0,Y_1,A_1,\Delta,\Lambda)$ over a field with reduction partition $X=X_0\cup X_1$, with $A_1[X_0,Y_1]$ a zero matrix, and with $A_1[X_1,Y_1]$ an identity matrix is a \emph{lifted} template.
\end{definition}
\begin{definition}
\label{def:lifted-partial}
 A \emph{lifted} $Y$-template over a partial field is of the form $\YT(P_0,[\emptyset])$.
\end{definition}
Note that the previous two definitions are equivalent in the case of $Y$-templates over fields.

\begin{remark}
\label{rem:lifted-Y-template}
The next lemma will mainly be used in the situation where $\Phi$ is a $Y$-template. Keeping that special case in mind may perhaps help to give the reader some intuition of the result. However, we prove the more general result since it may be of interest for future work. In the special case of $Y$-templates, the next lemma says that the $Y$-template $\YT(P_0,P_1)$ is equivalent to $\YT\left(\left[
\begin{array}{c|c}
-P_1&P_0\\
\hline
I&0\\
\end{array}
\right],[\emptyset]\right)$. This is not surprising in light of Lemma \ref{lem:valid}.
\end{remark}

\begin{lemma}
 \label{lem:lift}
Let $\Phi=(\Gamma,C,X,Y_0,Y_1,A_1,\Delta,\Lambda)$ be a refined frame template over a field $\F$ with reduction partition $X=X_0\cup X_1$. Then $\Phi$ is equivalent to a lifted template $\Phi'$. Moreover, let $\Phi=\YT(P_0,P_1)$ be a $Y$-template over a partial field $\P$. Then $\Phi$ is minor equivalent to a lifted $Y$-template $\Phi'$.
\end{lemma}

\begin{proof}
First, we consider the case where $\Phi$ is a refined frame template over a field $\F$. Recall that every refined template is reduced. Also recall that, in a reduced template, $A_1[X_1,C]$ is a zero matrix and that, in a refined template, $Y_1$ spans $M(A_1[X_1,Y_0\cup Y_1])$. Therefore, by performing elementary row operations, we may assume that $A_1$ is of the following form, with $Y_1=R\cup V$ and with each $P_i$ and each $Q_i$ an arbitrary matrix.

\begin{center}
\begin{tabular}{ c|c|c|c|c|}
\multicolumn{1}{c}{}&\multicolumn{1}{c}{$C$}&\multicolumn{1}{c}{$R$}&\multicolumn{1}{c}{$V$}&\multicolumn{1}{c}{$Y_0$}\\
\cline{2-5}
$X_0$&$Q_2$&$0$&$Q_1$&$Q_0$\\
\cline{2-5}
$X_1$&$0$&$I$&$P_1$&$P_0$\\
\cline{2-5}
\end{tabular}
\end{center}

We now construct $\Phi'$. Let $A'_1$ be of the following form, where $S$, $T$, and $U$ are pairwise disjoint sets each also disjoint from $R$, where $Y'_1=R\cup S$, where $Y'_0=T\cup U$, and where $X'=X_0\cup X_1\cup X_2$.
\begin{center}
\begin{tabular}{ c|c|c|c|c|c| }
\multicolumn{1}{c}{}&\multicolumn{1}{c}{$C$}&\multicolumn{1}{c}{$R$}&\multicolumn{1}{c}{$S$}&\multicolumn{1}{c}{$T$}&\multicolumn{1}{c}{$U$}\\
\cline{2-6}
$X_0$&$Q_2$&$0$&$0$&$-Q_1$&$Q_0$\\
\cline{2-6}
$X_1$&$0$&$I$&$0$&$-P_1$&$P_0$\\
\cline{2-6}
$X_2$&$0$&$0$&$I$&$I$&$0$\\
\cline{2-6}
\end{tabular}
\end{center}
Let $\Delta'$ be a subgroup of the additive group of $\mathbb{F}^{C\cup Y'_1\cup Y'_0}$ that is isomorphic to $\Delta$ with the isomorphism that maps
$\begin{blockarray}{cccc}
C&R&V&Y_0\\
\begin{block}{[cccc]}
\delta_1&\delta_2&\delta_3&\delta_4\\
\end{block}
\end{blockarray}$ to
$\begin{blockarray}{ccccc}
C&R&S&T&U\\
\begin{block}{[ccccc]}
\delta_1&\delta_2&0&-\delta_3&\delta_4\\
\end{block}
\end{blockarray}$, and let $\Lambda'$ be a subgroup of the additive group of $\mathbb{F}^{X'}$ that is isomorphic to $\Lambda$ with the isomorphism that maps
$\begin{blockarray}{c[c]}
X_0&\lambda_1\\
X_1&\lambda_2\\
\end{blockarray}$ to
$\begin{blockarray}{c[c]}
X_0&\lambda_1\\
X_1&\lambda_2\\
X_2&0\\
\end{blockarray}$. Then $\Phi'=(\Gamma,C,X',Y'_0,Y'_1,A'_1,\Delta',\Lambda')$ is a reduced template with reduction partition $X'=X_0\cup(X_1\cup X_2)$. The structure of a matrix virtually respecting $\Phi'$ is shown below. 
\begin{center}
\begin{tabular}{ r|c|c|ccccc| }
\multicolumn{2}{c}{}&\multicolumn{1}{c}{$Z$}&\multicolumn{1}{c}{$C$}&\multicolumn{1}{c}{$R$}&\multicolumn{1}{c}{$S$}&\multicolumn{1}{c}{$T$}&\multicolumn{1}{c}{$U$}\\
\cline{2-8}
$X_0$&columns from $\Lambda|X_0$&\multicolumn{1}{c|}{$0$}&\multicolumn{1}{c|}{$Q_2$}&\multicolumn{1}{c|}{$0$}&\multicolumn{1}{c|}{$0$}&\multicolumn{1}{c|}{$-Q_1$}&\multicolumn{1}{c|}{$Q_0$}\\
\cline{2-8}
$X_1$&$0$&\multicolumn{1}{|c|}{$0$}&\multicolumn{1}{|c|}{$0$}&\multicolumn{1}{|c|}{$I$}&\multicolumn{1}{|c|}{$0$}&\multicolumn{1}{|c|}{$-P_1$}&\multicolumn{1}{|c|}{$P_0$}\\
\cline{2-8}
$X_2$&$0$&\multicolumn{1}{|c|}{$0$}&\multicolumn{1}{|c|}{$0$}&\multicolumn{1}{|c|}{$0$}&\multicolumn{1}{|c|}{$I$}&\multicolumn{1}{|c|}{$I$}&\multicolumn{1}{|c|}{$0$}\\
\cline{2-8}
&\multirow{2}{*}{$\Gamma$-frame matrix}&\multirow{2}{*}{unit or zero columns}&\multicolumn{5}{c|}{\multirow{2}{4em}{rows from  $\Delta'$}}\\
&&&&&&&\\
\cline{2-8}
\end{tabular}
\end{center}
Since the reduction partition for $\Phi'$ is $X'=X_0\cup(X_1\cup X_2)$, the set $X_0$ will still play the role of $X_0$ in $\Phi'$, and $X_1\cup X_2$ will play the role of $X_1$ in $\Phi'$. Since $A_1'[X_0,Y_1']$ is a zero matrix, and since $A_1'[X_1\cup X_2,Y_1']$ is an identity matrix, $\Phi'$ is a lifted template. Thus, to prove the lemma, it suffices to show that $\Phi\sim\Phi'$.

First, note that $\mathcal{M}_w(\Phi)\subseteq\mathcal{M}_w(\Phi')$ because if $M$ conforms to $\Phi'$, then by contracting $T$ (pivoting on the nonzero entries in $A_1[X_2,T]$) we obtain a matroid conforming to $\Phi$. In other words, we repeatedly perform operation (12) on $\Phi'$ to obtain $\Phi$. To see that $\mathcal{M}_w(\Phi')\subseteq\mathcal{M}_w(\Phi)$, we will show that $\mathcal{M}(\Phi')\subseteq\mathcal{M}(\Phi)$. Recall that, in a matrix conforming to a template, the column indexed by an element $z\in Z$ is constructed by adding a column indexed by an element of $Y_1$ to the column indexed by $z$ in a matrix respecting the template. If $M$ conforms to $\Phi'$, let $Z_R\cup Z_S=Z\subseteq E(M)$, where $Z_R$ consists of the elements of $Z$ indexing columns constructed by adding a column indexed by an element of $R$, and where $Z_S$ consists of the elements of $Z$ indexing columns constructed by adding a column indexed by an element of $S$. Note that the set of elements of $M$ represented by columns with nonzero entries in rows indexed by $X_2$ consists of $Z_S\cup T$. Scale the columns indexed by $T$ and the rows indexed by $X_2$ by $-1$. To see that $M\in\mathcal{M}(\Phi)$, note that the columns indexed by elements of $T$ can now be constructed as columns indexed by elements of $Z$ in a matrix conforming to $\Phi$, and the columns indexed by elements of $Z_S$ can now be constructed as graphic columns (defined in Section \ref{sec:Frame-Templates}) indexed by elements of $E-(C\cup Z\cup Y_0\cup Y_1)$ in a matrix conforming to $\Phi$. Moreover, columns indexed by elements of $E-(C\cup Z\cup Y_0\cup Y_1\cup)$, or $Z_R$, or $C$, or $U$ in $\Phi'$ can be constructed as columns indexed by elements of $E-(C\cup Z\cup Y_0\cup Y_1)$, or $Z$, or $C$, or $Y_0$, respectively in $\Phi$.

In the case where $\Phi$ is a $Y$-template over a partial field, a similar (but simpler) argument shows that $\YT(P_0,P_1)$ is minor equivalent to $\YT(P_0',[\emptyset])$, where $P_0'=\left[
\begin{array}{c|c}
-P_1&P_0\\
\hline
I&0\\
\end{array}
\right]$.
\end{proof}

Recall that in a $Y$-template $\YT(P_0,P_1)$, the rows of $P_0$ and $P_1$ are indexed by a set $X$.

\begin{definition}
 \label{def:complete}
A $Y$-template $\YT(P_0,P_1)$ is \emph{complete} if $P_0$ contains $D_{|X|}$ as a submatrix.
\end{definition}

Recall the definition of minor equivalence from Definitions \ref{def:equivalent} and \ref{def:equivalent-PF}.

\begin{lemma}
 \label{lem:complete}
Every $Y$-template is minor equivalent to a complete template.
\end{lemma}

\begin{proof}
This follows immediately from the following {\color{red}lemma}.
\end{proof}

{\color{red}
\begin{lemma}
\label{lem:complete-construction}
Let $P_0$, $P_1$, $Q_1$, and $Q_2$ be a matrices with $n$ rows such that no column of $P_0$ is graphic but every column of $Q_1$ and $Q_2$ is graphic. The templates $\Phi_1=\YT([P_0|Q_1],P_1)$ and $\Phi_2=\YT([P_0|Q_2],P_1)$ are minor equivalent.
\end{lemma}

\begin{proof}
It suffices to show that every matroid virtually conforming to $\Phi_1$ is a minor of some matroid virtually conforming to $\Phi_2$, and vice-versa. First, we consider the case where $Q_1$ is a column submatrix of $Q_2$ with exactly one fewer column. One can easily see that every matroid virtually conforming to $\Phi_1$ is a minor of some matroid virtually conforming to $\Phi_2$ obtained by deleting one element.

We now must show that every matroid virtually conforming to $\Phi_2$ is a minor of some matroid virtually conforming to $\Phi_1$. Let $M$ be a matroid virtually conforming to $\Phi_2$. Then $M$ has a representation matrix of the following form, where $G$ consists entirely of graphic columns, $K$ consists entirely of unit and zero columns, and $H$ consists entirely of columns of $[I|P_1]$.
\begin{center}
\begin{tabular}{|c|c|c|c|}
\hline
$0$&$H$&$P_0$&$Q_2$\\
\hline
$G$&$K$&$0$&$0$\\
\hline
\end{tabular}
\end{center}
Let $[x_1,x_2,\dots,x_n]^T$ be the column of $Q_2$ that is not a column of $Q_1$. Since this is a graphic column, the following matrix represents a matroid virtually conforming to $\Phi_1$.
\begin{center}
\begin{tabular}{|c|c|c|c|c|c|}
\multicolumn{1}{c}{}&\multicolumn{1}{c}{}&\multicolumn{1}{c}{$Z'$}&\multicolumn{1}{c}{}&\multicolumn{1}{c}{}&\multicolumn{1}{c}{}\\
\hline
\multirow{3}{*}{$0$}&$0$&\multirow{3}{*}{$I$}&\multirow{3}{*}{$H$}&\multirow{3}{*}{$P_0$}&\multirow{3}{*}{$Q_1$}\\
&$\vdots$&&&&\\
&$0$&&&&\\
\hline
\multirow{3}{*}{$0$}&$x_1$&\multirow{3}{*}{$I$}&\multirow{3}{*}{$0$}&\multirow{3}{*}{$0$}&\multirow{3}{*}{$0$}\\
&$\vdots$&&&&\\
&$x_n$&&&&\\
\hline
\multirow{3}{*}{$G$}&$0$&\multirow{3}{*}{$0$}&\multirow{3}{*}{$K$}&\multirow{3}{*}{$0$}&\multirow{3}{*}{$0$}\\
&$\vdots$&&&&\\
&$0$&&&&\\
\hline
\end{tabular}
\end{center}
By contracting $Z'$, we obtain $M$. Therefore, $\Phi_1$ and $\Phi_2$ are minor equivalent.

For the general case, induction shows that $\Phi_1$ and $\Phi_2$ are both minor equivalent to $\YT(P_0,P_1)$ and therefore to each other.
\end{proof}}


\begin{lemma}
\label{lem:ones-and-zeros}
Every $Y$-template $\Phi=\YT(P_0,P_1)$ is semi-strongly equivalent to a $Y$-template $\Phi'=\YT(P_0',P_1')$ where the sum of the rows of $P_1'$ is $[1,\ldots,1]$ and the sum of the rows of $P_0'$ is the zero vector. If $\Phi$ is complete, so is $\Phi'$.
\end{lemma}

\begin{proof}
This follows immediately from the following construction.
\end{proof}

\begin{lemma}
\label{lem:ones-and-zeros-construction}
Let $P_0$ and $P_1$ be matrices over a partial field. Let $x_1,x_2,\ldots$ be the rows of $P_0$, let $y_1,y_2,\ldots$ be the rows of $P_1$, and let $y_*=[1,\ldots,1]-\Sigma y_i$. Let $P_1'=\left[\begin{array}{c}P_1\\
\hline
y_*\\
\end{array}\right]$, and let \[P_0'=\left[\begin{array}{c|c|c}
I&P_1&P_0\\
\hline
-1\cdots-1&-\Sigma y_i&-\Sigma x_{\color{red}i}\\
\end{array}\right].\] Then $\Phi=\YT(P_0,P_1)$ is semi-strongly equivalent to $\Phi'=\YT(P_0',P_1')$.
\end{lemma}

\begin{proof}
In this proof, we denote a zero vector by $\mathbf{0}$ and a vector of all $1$s as $\mathbf{1}$.

Every matroid virtually conforming to $\Phi$ is a restriction of a rank-$r$ universal matroid for $\Phi$ with a representation matrix as follows, with rows and columns indexed by $B$ and $E$, respectively. Here $x\in B-X$. (If $r=|X|$, then append a zero row with index $x$.)
\begin{center}
\begin{tabular}{r|c|c|c|c|c|c|c|c|c|c|c|c|c|c|}
\cline{2-15}
$X$&$\mathbf{0}$&$0$&$0$&$0$&$I$&$P_1$&$I$&$P_1$&$\cdots$&$I$&$P_1$&$I$&$P_1$&$P_0$\\
\cline{2-15}
$x$&1&$\mathbf{0}$&$\mathbf{0}$&$\mathbf{1}$&$\mathbf{1}$&$\mathbf{1}$&$\mathbf{0}$&$\mathbf{0}$&$\ldots$&$\mathbf{0}$&$\mathbf{0}$&$\mathbf{0}$&$\mathbf{0}$&$\mathbf{0}$\\
\cline{2-15}
&\multirow{3}{*}{$\mathbf{0}$}&\multirow{3}{*}{$I$}&\multirow{3}{*}{$D_{r-|X|-1}$}&\multirow{3}{*}{$-I$}&\multirow{3}{*}{$0$}&\multirow{3}{*}{$0$}&$\mathbf{1}$&$\mathbf{1}$&&&&\multirow{3}{*}{$0$}&\multirow{3}{*}{$0$}&\multirow{3}{*}{$0$}\\
&&&&&&&&&$\ddots$&&&&&\\
&&&&&&&&&&$\mathbf{1}$&$\mathbf{1}$&&&\\
\cline{2-15}
\end{tabular}
\end{center}

From the row indexed by $x$, subtract all rows indexed by elements of $X$ and, to the row indexed by $x$, add all rows indexed by elements in $B-(X\cup x)$. The result is the following (after scaling the fourth ``block'' of columns from $-I$ to $I$).
\begin{center}
\begin{tabular}{r|c|c|c|c|c|c|c|c|c|c|c|c|c|c|}
\cline{2-15}
$X$&$\mathbf{0}$&$0$&$0$&$0$&$I$&$P_1$&$I$&$P_1$&$\cdots$&$I$&$P_1$&$I$&$P_1$&$P_0$\\
\cline{2-15}
$x$&1&$\mathbf{1}$&$\mathbf{0}$&$\mathbf{0}$&$\mathbf{0}$&$y_*$&$\mathbf{0}$&$y_*$&$\ldots$&$\mathbf{0}$&$y_*$&$-\mathbf{1}$&$-\Sigma y_i$&$-\Sigma x_i$\\
\cline{2-15}
&\multirow{3}{*}{$\mathbf{0}$}&\multirow{3}{*}{$I$}&\multirow{3}{*}{$D_{r-|X|-1}$}&\multirow{3}{*}{$I$}&\multirow{3}{*}{$0$}&\multirow{3}{*}{$0$}&$\mathbf{1}$&$\mathbf{1}$&&&&\multirow{3}{*}{$0$}&\multirow{3}{*}{$0$}&\multirow{3}{*}{$0$}\\
&&&&&&&&&$\ddots$&&&&&\\
&&&&&&&&&&$\mathbf{1}$&$\mathbf{1}$&&&\\
\cline{2-15}
\end{tabular}
\end{center}

Rearranging the columns, we obtain the following, which is a representation matrix for the rank-$r$ universal matroid for $\Phi'$.
\begin{center}
\begin{tabular}{r|c|c|c|c|c|c|c|c|c|c|}
\cline{2-11}
$X\cup x$&$0$&$0$&$I$&$P_1'$&$\cdots$&$I$&$P_1'$&$I$&$P_1'$&$P_0'$\\
\cline{2-11}
&\multirow{3}{*}{$I$}&\multirow{3}{*}{$D_{r-|X|-1}$}&$\mathbf{1}$&$\mathbf{1}$&&&&\multirow{3}{*}{$0$}&\multirow{3}{*}{$0$}&\multirow{3}{*}{$0$}\\
&&&&&$\ddots$&&&&&\\
&&&&&&$\mathbf{1}$&$\mathbf{1}$&&&\\
\cline{2-11}
\end{tabular}
\end{center}
\end{proof}

\begin{definition}
\label{def:determined-by-P0}
The $Y$-template $\YT([P_0|D_{|X|}],[\emptyset])$ is the complete, lifted $Y$-template \emph{determined} by $P_0$ and is denoted by $\Phi_{P_0}$.
\end{definition}

Note that a matroid virtually conforming to $\Phi_{P_0}$ is a restriction of a matroid with a representation matrix of the following form:

\begin{center}
\begin{tabular}{|cc|c|c|c|}
\hline
\multicolumn{2}{|c|}{\multirow{2}{*}{$0$}}&unit&\multirow{2}{*}{$P_0$}&\multirow{2}{*}{$D_{|X|}$}\\
&&columns&&\\
\hline
\multicolumn{1}{|c|}{\multirow{2}{*}{$I_n$}}&\multicolumn{1}{|c|}{\multirow{2}{*}{$D_n$}}&unit or zero&\multicolumn{2}{|c|}{\multirow{2}{*}{$0$}}\\
\multicolumn{1}{|c|}{}&&columns&\multicolumn{2}{|c|}{}\\
\hline
\end{tabular}
\end{center}
 By scaling appropriately, we may assume that the bottom submatrix of unit columns actually consists of the negatives of unit columns. Thus a rank-$r$ matroid conforming to $\Phi_{P_0}$ is a restriction of a matroid with a representation matrix of the following form.
\[\left[\begin{array}{c|c|c}
\multirow{2}{*}{$I_r$}&\multirow{2}{*}{$D_r$}&P_0\\
\cline{3-3}
&&0\\
\end{array}\right]\]
The matroid represented by this matrix is the rank-$r$ universal matroid of $\Phi_{P_0}$.

\begin{lemma}
\label{lem:univ-gen-par}
The rank-$r$ universal matroid of $\Phi_{P_0}$ is the generalized parallel connection of $M(K_{r+1})$ and $\widetilde{M}([I_m|D_m|P_0])$ along $M(K_{m+1})$, where $m$ is the number of rows of $P_0$.
\end{lemma}

\begin{proof}
This follows from Lemma \ref{lem:Kn-modular} and Theorem \ref{thm:gen-par-con}.
\end{proof}

The next lemma applies Lemma \ref{lem:ones-and-zeros-construction} specifically to complete, lifted $Y$-templates.

\begin{lemma}
\label{lem:sum-to-zero}
\leavevmode
\begin{itemize}
\item[(i)] Every complete, lifted $Y$-template is semi-strongly equivalent to a complete, lifted $Y$-template determined by a matrix the sum of whose rows is the zero vector.
\item[(ii)] Conversely, let $\Phi$ be the complete, lifted $Y$-template determined by a matrix $P_0$ the sum of whose rows is the zero vector. Choose a row of $P_0$. Then $\Phi$ is semi-strongly equivalent to the complete, lifted $Y$-template determined by the matrix obtained from $P_0$ by removing that row.
\end{itemize}
\end{lemma}

\begin{proof} This follows directly from Lemma \ref{lem:ones-and-zeros-construction}.
\end{proof}

Recall the definition of algebraic equivalence (Definitions \ref{def:equivalent} and \ref{def:equivalent-PF}).

\begin{lemma}
\label{lem:min-dist-set}
Let $\Phi=\YT(P_0,P_1)$ and $\widehat{\Phi}=\YT(\widehat{P}_0,\widehat{P}_1)$ be $Y$-templates over the partial fields $\mathbb{P}$ and $\mathbb{\widehat{P}}$, respectively, such that $P_1$ and $\widehat{P}_1$ have the same dimensions. Let $P_0'=\left[
\begin{array}{c|c|c}
\multirow{2}{*}{$D$}&-P_1&P_0\\
\cline{2-3}
&I&0\\
\end{array}
\right]$, and let $\widehat{P}_0'$ be the matrix obtained from $P_0'$ by replacing $P_0$ with $\widehat{P}_0$ and $P_1$ with $\widehat{P}_1$. If the abstract vector matroids of $A=[I|P_0']$ and $\widehat{A}=[I|\widehat{P}_0']$ are equal (not just isomorphic), then $\Phi$ and $\Phi'$ are algebraically equivalent.
\end{lemma}

\begin{proof}
First, we consider the case where $P_1$ and $\widehat{P}_1$ are empty; that is $P_0'=[D|P_0]$ and $\widehat{P}_0'=[D|\widehat{P}_0]$. By Lemma \ref{lem:univ-gen-par}, the rank-$r$ universal matroid for $\Phi$ is the generalized parallel connection of $\widetilde{M}(A)$ with $M(K_{r+1})$. Similarly, the rank-$r$ universal matroid for $\widehat{\Phi}$ is the generalized parallel connection of $\widetilde{M}(\widehat{A})$ with $M(K_{r+1})$. Since $\widetilde{M}(A)=\widetilde{M}(\widehat{A})$, the rank-$r$ universal matroids for $\Phi$ and $\widehat{\Phi}$ are equal for every $r$. Because every matroid virtually conforming to a $Y$-template is a restriction of a universal matroid for the template, $\Phi$ and $\Phi'$ are algebraically equivalent. (This also depends on the fact that $P_0$ and $\widehat{P}_0$, which are always present in every matrix virtually conforming to $\Phi$ and $\widehat{\Phi}$, respectively, have equal dimensions.)

Now we consider the general case. Since $\widetilde{M}(A)=\widetilde{M}(\widehat{A})$, the above argument shows that $\YT(P_0',[\emptyset])$ is algebraically equivalent to $\YT(\widehat{P}_0',[\emptyset])$. The universal matroids for $\Phi$ and $\widehat{\Phi}$ are obtained by taking the equal universal matroids of $\YT(P_0',[\emptyset])$ and $\YT(\widehat{P}_0',[\emptyset])$ and contracting and deleting the same elements. Thus, the universal matroids for $\YT(P_0,P_1)$ and $\YT(\widehat{P}_0,\widehat{P}_1)$ are equal, and the templates are algebraically equivalent.
\end{proof}

\begin{definition}
\label{def:compatible}
Let $P_0$, $\widehat{P}_0$, $P_1$, and $\widehat{P}_1$ be matrices with the same number of rows such that $P_0$ and $\widehat{P}_0$ have the same zero-nonzero pattern, such that $P_1$ and $\widehat{P}_1$ have the same zero-nonzero pattern, such that $P_0$ and $P_1$ have entries from a partial field $\P$, and such that $\widehat{P}_0$ and $\widehat{P}_1$ have entries from a partial field $\widehat{\P}$. Let $\Phi=\YT(P_0,P_1)$ and $\widehat{\Phi}=\YT(\widehat{P}_0,\widehat{P}_1)$. Then $\Phi$ and $\Phi'$ are \emph{pattern-compatible} templates.
\end{definition}

\begin{theorem}
\label{thm:PtoFtemplate}
Let $P_0$ be a matrix, with $m$ rows, over a partial field $\P$. If $M=\widetilde{M}([I_m|D_m|P_0])$ is representable over a partial field $\widehat{\P}$, then $\Phi_{P_0}$ is algebraically equivalent to a valid $Y$-template over $\widehat{\P}$ that is pattern-compatible with $\Phi_{P_0}$.
\end{theorem}

\begin{proof}
Since $\widetilde{M}([I_m|D_m|P_0])$ is $\widehat{\P}$-representable, there is a $\widehat{\P}$-matrix that represents $M$. By pivoting and by scaling rows and columns, we may assume that this matrix is the $\widehat{\mathbb{P}}$-matrix $\left[
\begin{array}{c|c|c|c}
\multirow{2}{*}{$I_m$}&1\cdots1&0\cdots0&\multirow{2}{*}{$*$}\\
\cline{2-3}
&-I_{m-1}&D'_{m-1}&\\
\end{array}
\right]$, where $D'_{m-1}$ has nonzero entries in the same locations as $D_{m-1}$ and is scaled so that the first nonzero entry in each column is a $1$ and where the star represents a matrix with the same zero-nonzero pattern as $P_0$. In order to have the same linearly independent sets of column vectors as $M$, we must have $D'_{m-1}=D_{m-1}$. Therefore, letting $\widehat{P}_0$ denote the matrix represented by the star, we have that $M$ is represented by the $\widehat{\P}$-matrix $[I_m|D_m|\widehat{P}_0]$. By Lemma \ref{lem:min-dist-set}, $\Phi_{P_0}$ and $\Phi_{\widehat{P}_0}$ are algebraically equivalent. By Lemma \ref{lem:valid}, $\Phi_{\widehat{P}_0}$ is valid.
\end{proof}

Results such as Lemma \ref{lem:valid} and Remark \ref{rem:lifted-Y-template} show that matrices consisting of an arbitrary matrix placed on top of unit columns are important in the study of $Y$-templates. In the remainder of this section, we study similar phenomena in some more generality.

\begin{lemma}
\label{lem:P0-to-P1}
Let $Q_1$ and $Q_2$ be matrices over a partial field $\P$, each with $m$ rows. Let $\textbf{v}_1,\textbf{v}_2,\ldots,\textbf{v}_n\in\P^m$ be column vectors. If $P_1=\left[\begin{array}{c}
Q_1\\
\hline
0\\
\end{array}\right]$ and $P_0$ is of the form
\begin{center}
\begin{tabular}{|c|c|c|c|c|c|}
\hline
$\textbf{v}_1\cdots\textbf{v}_1$&$\textbf{v}_2\cdots\textbf{v}_2$&$\cdots$&$\textbf{v}_n\cdots\textbf{v}_n$&$Q_2$&\multirow{2}{*}{$H$}\\
\cline{1-5}
\multicolumn{4}{|c|}{unit or zero columns}&$0$&\\
\hline
\end{tabular},
\end{center}
where $H$ is a column submatrix of $D_{|X|}$, then $\mathcal{M}_v(\YT(P_0,P_1))\subseteq\mathcal{M}_v(\YT(P_0',P_1'))$, where $P_0'=[Q_2|D_{|X|}]$ and $P_1'=[Q_1|-\textbf{v}_1|-\textbf{v}_2|\cdots|-\textbf{v}_n]$.
\end{lemma}

\begin{proof}
Let $\Phi=\YT(P_0,P_1)$, and let $\Phi'=\YT(P_0',P_1')$. Every matroid virtually conforming to $\Phi$ is represented by a column submatrix of a matrix of the following form, where $V=[\textbf{v}_1|\textbf{v}_2|\cdots|\textbf{v}_n]$ and where $X=X_1\cup X_2$.
\begin{center}
\begin{tabular}{r|c|c|c|c|c|c|c|}
\multicolumn{1}{c}{}&\multicolumn{1}{c}{$A$}&\multicolumn{1}{c}{$B$}&\multicolumn{1}{c}{$C$}&\multicolumn{1}{c}{$D$}&\multicolumn{1}{c}{$F$}&\multicolumn{1}{c}{$G$}&\multicolumn{1}{c}{$J$}\\
\cline{2-8}
\multirow{3}{*}{$X_1$}&\multirow{3}{*}{$0$}&unit &\multirow{3}{*}{$0$}&columns&columns&\multirow{3}{*}{$Q_2$}&\multirow{5}{*}{$D_{|X|}$}\\
&&or zero&&from&from&&\\
&&columns&&$Q_1$&$V$&&\\
\cline{2-7}
\multirow{2}{*}{$X_2$}&\multirow{2}{*}{$0$}&\multirow{2}{*}{$0$}&unit or zero&\multirow{2}{*}{$0$}&unit or zero&\multirow{2}{*}{$0$}&\\
&&&columns&&columns&&\\
\cline{2-8}
&$D_{r-|X|}$&\multicolumn{3}{|c|}{unit or zero columns}&\multicolumn{3}{|c|}{$0$}\\
\cline{2-8}
\end{tabular}
\end{center}Every matroid virtually conforming to $\Phi'$ is represented by a column submatrix of a matrix of the following form.
\begin{center}
\begin{tabular}{|c|c|c|c|c|}
\multicolumn{1}{c}{$A'$}&\multicolumn{1}{c}{$B'$}&\multicolumn{1}{c}{$C'$}&\multicolumn{1}{c}{$D'$}&\multicolumn{1}{c}{$F'$}\\
\hline
$0$&unit or zero columns&columns from $P_1'$&$Q_2$&$D_{|X_1|}$\\
\hline
$D_{r-|X_1|}$&\multicolumn{2}{|c|}{unit or zero columns}&\multicolumn{2}{|c|}{$0$}\\
\hline
\end{tabular}
\end{center}

Let $J_1$ be the subset of $J$ indexing columns both of whose nonzero entries occur in rows indexed by $X_1$; let $J_2$ be the subset of $J$ indexing columns both of whose nonzero entries occur in rows indexed by $X_2$; and let $J_3$ be the subset of $J$ indexing columns that have one nonzero entry in each of the sets of rows indexed by $X_1$ and $X_2$. To show that every matroid virtually conforming to $\Phi$ virtually conforms to $\Phi'$, scale each of the rows indexed by $X_2$ and each of the columns indexed by $C\cup F\cup J_2$ by $-1$. Then we can choose $A'=A\cup C\cup J_2$, and $B'=B\cup J_3$, and $C'=D\cup F$, and $D'=G$, and $F'=J_1$.
\end{proof}

In the following definition, $e_i$ is a unit column whose nonzero entry is in the $i^{\text{th}}$ row.

\begin{definition}
\label{def:semi-parallel}
Let $\textbf{v}_1$, $\textbf{v}_2$, $\ldots$, $\textbf{v}_n$ be column vectors with the same number $m$ of entries. Let $A$ be a matrix whose columns can be scaled so that the matrix is of the following form, where $*$ represents an arbitrary matrix and where each $\textbf{v}_i$ appears at least once.
\begin{center}
\begin{tabular}{|c|c|c|c|c|}
\hline
&&&&\\
$\textbf{v}_1\cdots\textbf{v}_1$&$\textbf{v}_2\cdots\textbf{v}_2$&$\cdots$&$\textbf{v}_n\cdots\textbf{v}_n$&$*$\\
&&&&\\
\hline
\multicolumn{4}{|c|}{unit columns}&$0$\\
\hline
\end{tabular}
\end{center}
We say that all of the columns of the matrix of the form $\textbf{v}_i+e_k$, where $i\in\{1,\ldots,n\}$ and $k>m$, are \emph{semi-parallel} to each other. If $\textbf{v}$ is an additional column not contained in $A$, and if $\textbf{v}$ can be scaled so that it is of the form $\textbf{v}_j+e_{\ell}$, where $j\in\{1,\ldots,n\}$ and $\ell>m$, then we say that $\textbf{v}$ is also semi-parallel to the existing columns of $A$ of the form $\textbf{v}_i+e_k$, where $i\in\{1,\ldots,n\}$ and $k>m$. Moreover, we call $\textbf{v}$ a \emph{semi-parallel extension} of $A$.
\end{definition}

\begin{definition}
\label{def:contractible}
Let $Q'=\left[\begin{array}{c}
Q\\
\hline
I\\
\end{array}\right]$, where no column of $Q$ is a zero column or the negative of a unit column (which, of course, is a unit column itself in characteristic $2$), where no row of $Q$ is a zero row, and where no column of $Q'$ is a semi-parallel extension of the matrix resulting from $Q'$ by removing that column. We call any matrix that can be obtained from $Q'$ by permuting rows, permuting columns, and scaling columns a \emph{contractible} matrix.
\end{definition}

\begin{lemma}
\label{lem:contracible}
Let $P_0$ be a matrix whose contractible submatrix $Q'$ with the most rows has $r$ rows. Without loss of generality, let these be the first $r$ rows of $P_0$. Let $Q$ and $Q'$ be as in Definition \ref{def:contractible} (so that $r=m+n$), and let $\textbf{v}_1$, $\textbf{v}_2$, $\ldots$, $\textbf{v}_n$ be the columns of $Q$. If $P_0$ is of the following form, where $S$ and $T$ are arbitrary matrices, then there is a complete, lifted $Y$-template $\Phi$, determined by a matrix with $r$ rows, such that every matroid conforming to $\Phi_{P_0}$ is a minor of some matroid conforming to $\Phi$.
\begin{center}
\begin{tabular}{|c|c|c|c|c|c|c|c|}
\multicolumn{4}{c}{}&\multicolumn{1}{c}{$Z_1$}&\multicolumn{1}{c}{$\cdots$}&\multicolumn{1}{c}{$Z_n$}&\multicolumn{1}{c}{$J$}\\
\hline
$Q$&$\textbf{v}_1\cdots\textbf{v}_1$&$\cdots$&$\textbf{v}_n\cdots\textbf{v}_n$&$\textbf{v}_1\cdots\textbf{v}_1$&$\cdots$&$\textbf{v}_n\cdots\textbf{v}_n$&$S$\\
\cline{1-8}
$I_n$&\multicolumn{3}{|c|}{unit or zero columns}&\multicolumn{3}{|c|}{$0$}&$T$\\
\cline{1-8}
\multicolumn{4}{|c|}{$0$}&\multicolumn{3}{|c|}{unit columns}&$0$\\
\hline
\end{tabular}
\end{center}
\end{lemma}

\begin{proof}
Let $Q''$ be the column submatrix of $P_0$ with the same columns as $Q'$ (so that $Q''$ is $Q'$ with zero rows appended to it). In a complete, lifted $Y$-template $\Phi_{P_0}$, all unit columns and all graphic columns can always be constructed, regardless of $P_0$. Moreover, we are interested in simple matroids. Therefore, we may assume that no column indexed by an element of $J$ is a unit column, a graphic column, or a semi-parallel extension of $Q''$.

Let $\textbf{u}_1$, $\textbf{u}_2$, $\ldots$, $\textbf{u}_n$ be the rows of $T$, from top to bottom. For $1\leq i\leq n$, suppose $Z_i\neq\emptyset$ and $\textbf{u}_i$ is nonzero. Then we can find a contractible submatrix with $r+1$ rows contained in the columns of $Q''$ (other than the column containing $\textbf{v}_i$), the column indexed by the element of $Z_i$, and the column containing the nonzero entry of $\textbf{u}_i$. This contradicts the assumption that $Q'$ is the contractible submatrix of $P_0$ with the most rows. Therefore, for each $i$ with $1\leq i\leq n$, either $Z_i=\emptyset$ or $\textbf{u}_i$ is the zero vector.

Without loss of generality, let $Z_i=\emptyset$ if and only if $i\leq k$. Then $\textbf{u}_{k+1}$, $\textbf{u}_{k+2}$, $\ldots$, $\textbf{u}_n$ must all be the zero vector. By Lemma \ref{lem:P0-to-P1}, we have $\mathcal{M}_v(\Phi_{P_0})\subseteq\mathcal{M}_v(\YT(P_0',-Q))$, where $P_0'$ is of the following form, with $T'$ obtained by restricting to the first $k$ rows of $T$.
\begin{center}
\begin{tabular}{|c|c|c|c|c|c|}
\hline
$\textbf{v}_1\cdots\textbf{v}_k$&$\textbf{v}_1\cdots\textbf{v}_1$&$\cdots$&$\textbf{v}_n\cdots\textbf{v}_n$&$S$&\multirow{2}{*}{$D_{m+k}$}\\
\cline{1-5}
$I_k$&\multicolumn{3}{|c|}{unit or zero columns}&$T'$&\\
\hline
\end{tabular}
\end{center}
By Remark \ref{rem:lifted-Y-template} and Lemma \ref{lem:complete-construction}, $\YT(P_0',-Q)$ is minor equivalent to the complete, lifted $Y$-template determined by the following matrix, which has $r$ rows.
\begin{center}
\begin{tabular}{|c|c|c|c|c|c|}
\hline
\multicolumn{2}{|c|}{$Q$}&$\textbf{v}_1\cdots\textbf{v}_1$&$\cdots$&$\textbf{v}_n\cdots\textbf{v}_n$&$S$\\
\hline
$I_k$&$0$&\multicolumn{3}{|c|}{unit or zero columns}&$T'$\\
\hline
$0$&$I_{n-k}$&\multicolumn{4}{|c|}{$0$}\\
\hline
\end{tabular}
\end{center}
\end{proof}

\section{Extremal Functions}
\label{sec:Extremal Functions}
In this section, we will compute the extremal function of $\AC$ and $\GM$. We will call a largest simple matroid that conforms to a template an \emph{extremal matroid} of the template. The next lemma can be obtained by combining \cite[Lemma 4.16]{gvz18} and \cite[Theorem 5.6]{gvz18}.

\begin{lemma}
 \label{lem:extremaltemplate}
Suppose Hypothesis \ref{hyp:clique-template} holds, and let $\mathbb{F}$ be a finite field. Let $\mathcal{M}$ be a quadratically dense minor-closed class of $\mathbb{F}$-represented matroids, and let $\{\Phi_1,\dots,\Phi_s,\Psi_1,\dots,\Psi_t\}$ be the set of templates given by Hypothesis \ref{hyp:clique-template}. For all sufficiently large $r$, the extremal matroids of $\mathcal{M}$ are the extremal matroids of the templates in some subset of $\{\Phi_1,\dots,\Phi_s\}$. Moreover, $\{\Phi_1,\dots,\Phi_s\}$ can be chosen so that it consists entirely of refined templates.
\end{lemma}

The next lemma deals with a technicality involving virtual conforming. Note that there are templates such that the largest simple matroid of rank $r$ conforming to the template is also the largest simple matroid of rank $r$ virtually conforming to the template. The most obvious such templates are those with $Y_1=\emptyset$. For another set of examples, let $\Phi$ be a template with $Y_1=\{y_1,y_2\ldots,y_n\}$ and let $Y'=\{y_1',y_2',\ldots,y_n'\}\subseteq Y_0$ such that, for each $i\leq n$, we have $A_1[X,y_i]=A_1[X,y_i']$ and for each $\delta\in\Delta$, we have $\delta_{y_i}=\delta_{y_i'}$.

\begin{lemma}
\label{lem:extremal-virtual}
Suppose Hypothesis \ref{hyp:clique-template} holds, and let $\mathbb{F}$ be a finite field. Let $\mathcal{M}$ be a quadratically dense minor-closed class of $\mathbb{F}$-represented matroids, and let $\{\Phi_1,\dots,\Phi_s,\Psi_1,\dots,\Psi_t\}$ be the set of templates given by Hypothesis \ref{hyp:clique-template}. For all sufficiently large $r$, the extremal matroids of $\mathcal{M}$ are the largest simple matroids that virtually conform to the refined templates in some subset of $\{\Phi_1,\dots,\Phi_s\}$.
\end{lemma}

\begin{proof}
By Lemma \ref{lem:extremaltemplate}, we know that, for all sufficiently large $r$, the extremal matroids of $\mathcal{M}$ are the largest simple matroids that conform to to the templates in some subset $T\subseteq\{\Phi_1,\dots,\Phi_s\}$, where each of the templates is refined. By Lemma \ref{minor}, a matroid that virtually conforms to a template in this set is a minor of some matroid that conforms to it. Since every matroid conforming to the template is in the minor-closed class $\mathcal{M}$, every matroid virtually conforming to the template is also in $\mathcal{M}$. The size of the largest simple matroid that virtually conforms to a template is at least the size of the largest simple matroid that conforms to the template. Thus, $T$ must consist of templates where the largest simple matroids conforming to the template are the same as the largest simple matroids virtually conforming to the template.
\end{proof}

Consider the example given immediately before the previous lemma. The largest simple matroid (virtually) conforming to $\Phi$ is the same as the largest simple matroid virtually conforming to the template obtained from $\Phi$ by deleting $Y'$ from $Y_0$. In practice, when using templates to determine the extremal function of a minor-closed class, we will consider this other template, rather than $\Phi$ itself.

Recall the definitions of $T_r^2$, $G_r$, and $HP_r$ given in Definition \ref{def:families}. We note that Welsh \cite{w14} intended to define $T_r^2$, $G_r$, and $HP_r$ as matroids over the golden-mean partial field, rather than over $\GF(4)$, but this was done incorrectly because the matrices given in \cite[Figure 2.1]{w14} for $G_r$ and $HP_r$ are not actually $\mathbb{G}$-matrices. It suffices for our purposes to define $T_r^2$, $G_r$, and $HP_r$ as matrices over $\GF(4)$, as we did in Definition \ref{def:families} and as Welsh did in \cite[Figure 2.5]{w14}.

In the representation of $HP_r$ given in Definition \ref{def:families}, scale the top row by $\alpha^2$, the first column by $\alpha$, and the last three columns by $\alpha$. Then rearrange the last three columns to obtain the following.
\begin{center}
$HP_r=M\left(\begin{tabular}{|c|c|c|c|c|c|ccc|}
\hline
\multirow{3}{*}{$I_r$}&$0\cdots0$&$1\cdots1$&$0\cdots0$&$1\cdots1$&$\alpha\cdots\alpha$&$1$&$1$&$1$\\
&$0\cdots0$&$0\cdots0$&$1\cdots1$&$\alpha\cdots\alpha$&$1\cdots1$&$1$&$\alpha$&$\alpha^2$\\
\cline{2-9}
&$D_{r-2}$&$I_{r-2}$&$I_{r-2}$&$I_{r-2}$&$I_{r-2}$&&$0$&\\
\hline
\end{tabular}\right)$
\end{center}
Thus, for $r\geq2$, the $\GF(4)$-represented matroids corresponding to $T_r^2$, $G_r$, and $HP_r$ are the largest simple matroids of rank $r$ virtually conforming to the $Y$-template $\YT(P_0,P_1)$ over $\mathrm{GF}(4)$, where $(P_0,P_1)$ is $([\emptyset],[\alpha,\alpha^2])$, $\left(\begin{bmatrix}1&1&1\\1&\alpha&\alpha^2\\\end{bmatrix},\begin{bmatrix}\alpha&0\\0&\alpha\\\end{bmatrix}\right)$, and $\left(\begin{bmatrix}1\\1\\\end{bmatrix},\begin{bmatrix}1&\alpha\\\alpha&1\\\end{bmatrix}\right)$, respectively. We will call these templates $\Phi(T_r^2)$, $\Phi(G_r)$, and $\Phi(HP_r)$, respectively.

\begin{proposition}
\label{pro:AW-golden-mean}
 The matroids $T_r^2$, $G_r$, and $HP_r$ are golden-mean matroids.
\end{proposition}

\begin{proof}
By Definition \ref{def:families}, $T_r^2$, $G_r$, and $HP_r$ are representable over $\mathrm{GF}(4)$. By Theorem \ref{thm:golden-mean-char}, it suffices to show that they are representable over $\mathrm{GF}(5)$. Since these matroids are respectively the largest simple matroids of rank $r$ that virtually conform to complete $Y$-templates, it suffices to find complete $Y$-templates over $\mathrm{GF}(5)$ that are algebraically equivalent to these templates.

We claim that $\Phi(T_r^2)$, $\Phi(G_r)$, and $\Phi(HP_r)$ are algebraically equivalent to the templates $\YT([\emptyset],[3,4])$,
$\YT\left(
\begin{bmatrix}
1&1&1\\
4&2&3\\
\end{bmatrix},
\begin{bmatrix}
3&0\\
0&3\\
\end{bmatrix}\right)$, and $\YT\left(
\begin{bmatrix}1\\
4\\
\end{bmatrix},
\begin{bmatrix}
 1&3\\
 3&1\\
\end{bmatrix}\right)$, respectively, over $\mathrm{GF}(5)$. By Lemma \ref{lem:min-dist-set}, it is straightforward to verify this claim by checking that certain matroids are equal. For example, since $-3=2$ and $-4=1$ in $\GF(5)$, to verify that $\Phi(T_r^2)$ is algebraically equivalent to $\YT([\emptyset],[3,4])$, we must show that the vector matroids of $\begin{bmatrix}
1&0&0&1&1&0&\alpha&\alpha^2\\
0&1&0&1&0&1&1     &0\\
0&0&1&0&1&1&0     &1\\
\end{bmatrix}$ and $\begin{bmatrix}
1&0&0&1&1&0&2&1\\
0&1&0&4&0&1&1&0\\
0&0&1&0&4&4&0&1\\
\end{bmatrix}$ are equal. This is easily verified with SageMath, as are the analogous computations for $\Phi(G_r)$ and $\Phi(HP_r)$. 
\end{proof}

Recall from Section \ref{sec:Characteristic Sets} that $F_7$, $F_7^*$, $V_1$, $V_2$, $V_3$, $P_1$, $P_2$, and $P_3$ are excluded minors for $\mathcal{AC}_4$. This fact will be used in the next several proofs.

\begin{lemma}
\label{lem:AC4-Y-template}
If $\Phi=(\Gamma,C,X,Y_0,Y_1,A_1,\Delta,\Lambda)$ is a refined quaternary template such that $\mathcal{M}(\Phi)\subseteq\mathcal{AC}_4$, then $\Phi$ is a $Y$-template.
\end{lemma}

\begin{proof}
We will prove this result using Theorem \ref{thm:reduce-to-Y-template}. Consider the following representation of $F_7$: $\begin{bmatrix}
1&0&0&1&1&0&1\\
0&1&0&1&0&1&1\\
0&0&1&0&1&1&1\\
\end{bmatrix}$. Choose one of these rows and let $X=\{x\}$ be the set indexing it. The remaining two rows form a $\{1\}$-frame matrix. Thus, $F_7$ conforms to $\Phi_X$. Taking the same representation, if we index the last column by $Y_0$, the rest of the columns form a $\{1\}$-frame matrix. Thus, $F_7$ conforms to $\Phi_{Y_0}$. Therefore, $\Phi_X\npreceq\Phi$, and $\Phi_{Y_0}\npreceq\Phi$. Moreover, by Lemma \ref{YCD}, $\Phi_C\npreceq\Phi$, and $\Phi_{CXk}\npreceq\Phi$ for each $k\in\GF(4)-\{0\}$.

The multiplicative group of $\GF(4)$ has size $3$. Therefore, the only possible prime divisor of $|\Gamma|$ is $3$. The class $\mathcal{M}(\Phi_3)$ is exactly the class of quaternary matroids with representation matrices such that every column has at most two nonzero entries. The matroids $P_1$ and $P_3$ are such matroids; therefore $\Phi_3\npreceq\Phi$. Therefore, by Theorem \ref{thm:reduce-to-Y-template}, $\Phi$ is a $Y$-template.
\end{proof}

Recall Archer's conjecture from Section \ref{sec:Golden-Mean Matroids}. Welsh \cite[Theorem 6.1.4]{w14} showed that Archer's conjecture holds for golden-mean matroids containing a spanning clique. By Lemmas \ref{lem:complete} and \ref{lem:AC4-Y-template}, one can see that the largest simple matroids conforming to some template $\Phi$ with $\mathcal{M}(\Phi)\subseteq\GM\subseteq\AC$ must have a spanning clique. Thus, by Lemma \ref{lem:extremaltemplate}, for sufficiently large rank, Archer's conjecture holds, subject to Hypothesis \ref{hyp:clique-template}. We will give a different proof as part of our proof of Theorem \ref{thm:extremal}.

If an extremal matroid for a minor-closed class virtually conforms to a template, then we call that template an \emph{extremal template} for the class. Let $\Phi$ be a template such that $\mathcal{M}(\Phi)\subseteq\mathcal{AC}_4$. By Lemma \ref{lem:AC4-Y-template}, $\Phi$ is a $Y$-template $\YT(P_0,P_1)$. By Lemma \ref{simpleY1}, we may assume that every column of $P_1$ is nonzero and that no column of $[I_{|X|}|P_1]$ is a copy of another column. If $P_1$ is a $c\times d$ matrix and if we define $|\widehat{Y}|=\varepsilon(\widetilde{M}([I_{|X|}|P_1|P_0]))$, then the largest simple matroid $M$ of rank $r\geq|X|=c$ virtually conforming to $\Phi$ has $r-c+\binom{r-c}{2}$ elements in $E(M)-(Y_0\cup Y_1\cup Z)$ and $(c+d)(r-c)+|\widehat{Y}|$ elements in $Z\cup Y_0$. Thus, \[\varepsilon(M)=r-c+\binom{r-c}{2}+(c+d)(r-c)+|\widehat{Y}|.\] Some arithmetic shows \[\varepsilon(M)=\frac{1}{2}r^2+\frac{2d+1}{2}r-\frac{c^2+c+2dc-2|\widehat{Y}|}{2}.\]

Thus, for all sufficiently large ranks, to find the extremal templates for $\mathcal{AC}_4$, we must find $P_1$ with as many (distinct) columns as possible. We now analyze the structure of $P_1$ when $\mathcal{M}(\Phi)\subseteq\mathcal{AC}_4$.

Recall that applying an automorphism of a field $\F$ to every entry in a matrix $A$ over $\F$ results in a matrix whose vector matroid is equal to that of $A$. The only nontrivial automorphism of $\GF(4)$ maps $\alpha$ to $\alpha^2$ and $\alpha^2$ to $\alpha$ (while leaving $0$ and $1$ fixed). In the rest of the paper, when we say ``up to a field automorphism'', this is the automorphism we refer to.

\begin{lemma}
\label{lem:P1forbidden-1or2columns}
Suppose $\Phi=\YT(P_0,P_1)$ is a template such that $\mathcal{M}(\Phi)\subseteq\mathcal{AC}_4$. Then none of the matrices listed in Table \ref{tab:Forbidden_Matrices} are submatrices of $P_1$. Moreover, none of the matrices obtained from those in Table \ref{tab:Forbidden_Matrices} by replacing $\alpha$ with $\alpha^2$ and vice-versa are submatrices of $P_1$.
\end{lemma}

\begin{table}[!htbp]
\caption{Forbidden Matrices}
\label{tab:Forbidden_Matrices}
\begin{center}
\begin{tabular}{|c|c|c||c|c|c|}
\hline
Matrix&Rank&Excluded&Matrix&Rank&Excluded\\
&&Minor&&&Minor\\
\hline
$A=\begin{bmatrix}
1\\
1\\
\end{bmatrix}$&$3$&$F_7$&$B=\begin{bmatrix}
1\\
\alpha\\
\alpha^2\\
\end{bmatrix}$&$4$&$V_2$\\
\hline
$C=\begin{bmatrix}
\alpha\\
\alpha\\
\alpha\\
\end{bmatrix}$&$4$&$F_7^*$&$D=\begin{bmatrix}
\alpha\\
\alpha\\
\alpha^2\\
\end{bmatrix}$&$4$&$V_2$\\
\hline
$E=\begin{bmatrix}
\alpha&0\\
\alpha&0\\
0&\alpha\\
\end{bmatrix}$&$5$&$F_7^*$&$F=\begin{bmatrix}
\alpha&0\\
\alpha&0\\
0&\alpha^2\\
\end{bmatrix}$&$5$&$F_7^*$\\
\hline
$G=\begin{bmatrix}
1&0\\
\alpha&\alpha\\
\end{bmatrix}$&$3$&$F_7$&$H=\begin{bmatrix}
1&0\\
\alpha&\alpha^2\\
\end{bmatrix}$&$4$&$F_7^*$\\
\hline
$I=\begin{bmatrix}
1&\alpha^2\\
\alpha&1\\
\end{bmatrix}$&$4$&$P_2$&$J=\begin{bmatrix}
\alpha&\alpha\\
\alpha&0\\
\end{bmatrix}$&$4$&$P_2$\\
\hline
$K=\begin{bmatrix}
1&1\\
\alpha&\alpha^2\\
\end{bmatrix}$&$3$&$F_7$&$L=\begin{bmatrix}
\alpha^2&0\\
\alpha&0\\
0&\alpha\\
\end{bmatrix}$&$5$&$P_2$\\
\hline
$M=\begin{bmatrix}
\alpha&\alpha^2&0\\
\alpha&0&\alpha^2\\
\end{bmatrix}$&$3$&$F_7$&$N=\begin{bmatrix}
\alpha^2&\alpha&0\\
0&0&\alpha^2\\
\end{bmatrix}$&$4$&$P_2$\\
\hline
{\color{red}$O=\begin{bmatrix}
1       &\alpha&\alpha\\
\alpha&1       &\alpha\\
\end{bmatrix}$}&$3$&$V_1$&{\color{red}$P=\begin{bmatrix}
\alpha    &\alpha&0\\
\alpha^2&0       &\alpha\\
\end{bmatrix}$}&$4$&$F_7^*$\\
\hline
{\color{red}$Q=\begin{bmatrix}
\alpha^2&\alpha&0\\
\alpha    &0       &\alpha^2\\
\end{bmatrix}$}&$3$&$F_7$&{\color{red}$R=\begin{bmatrix}
1       &\alpha\\
\alpha&1\\
\alpha&0\\
\end{bmatrix}$}&$4$&$F_7^*$\\
\hline
\end{tabular}
\end{center}
\end{table}

\begin{proof}
If one of the matrices listed in Table \ref{tab:Forbidden_Matrices} is a submatrix of $P_1$, then we can perform deletion (operation (4) of Proposition \ref{reductions}) and $y$-shifts (Lemma \ref{yshift}) followed by contraction (operation (11) of Definition \ref{def:weaklyconforming}) on elements of $Y_1$ to obtain a template where the matrix from Table \ref{tab:Forbidden_Matrices} is itself $P_1$.

Now we must show that, if $P_1$ is one of the matrices in Table \ref{tab:Forbidden_Matrices}, then $\mathcal{M}(\Phi)\nsubseteq\mathcal{AC}_4$. We verify this with SageMath. The rank listed in the table is the rank of the matroid virtually conforming to the template which contains the excluded minor given in the table. The last statement of the lemma follows from the fact that field isomorphisms preserve (abstract) representable matroids.
\end{proof}

The next lemma involves larger matrices where the technique used in the previous lemma is significantly more time-consuming.

\begin{lemma}
\label{lem:alpha-times-identity}
Suppose $\Phi=\YT(P_0,P_1)$ is a template such that $\mathcal{M}(\Phi)\subseteq\mathcal{AC}_4$. Then $P_1$ contains neither ${\color{red}S=}\begin{bmatrix}
\alpha&0&0\\
0&\alpha&0\\
0&0&\alpha\\
\end{bmatrix}$ nor ${\color{red}T=}\begin{bmatrix}
\alpha&0&0\\
0&\alpha&0\\
0&0&\alpha^2\\
\end{bmatrix}$ as a submatrix.
\end{lemma}

\begin{proof}
If {\color{red}the matrix} $P_1{\color{red}=S}$, SageMath can be used to show that {\color{red}the matroid} $P_2$ is a minor of a matroid conforming to $\Phi$. Now, {\color{red}scaling the last row of $[I_3|S]$} by $\alpha^2$ and rearranging the columns, {\color{red}we obtain $[I_3|T]$. This implies that $\YT([\emptyset],S)$ is strongly equivalent to $\YT([\emptyset],T)$. Therefore, $P_1$ can contain neither $S$ nor $T$ as a submatrix.}
\end{proof}

\begin{lemma}
\label{lem:allowed-matrices}
Suppose $\Phi=\YT(P_0,P_1)$ is a template such that $\mathcal{M}(\Phi)\subseteq\mathcal{AC}_4$. Then $\Phi$ is equivalent to a template $\YT(P_0',P_1')$, where $P_1'$ is, after its zero rows are removed, a submatrix of one of the following matrices, up to field automorphism.

\begin{center}
$\begin{bmatrix}
1&\alpha\\
\alpha&\alpha^2\\
\alpha&0\\
\end{bmatrix}$,
$\begin{bmatrix}
1&\alpha\\
\alpha&\alpha\\
\alpha&1\\
\end{bmatrix}$,
$\begin{bmatrix}
\alpha^2&\alpha^2\\
\alpha&0\\
0&\alpha\\
\end{bmatrix}$,
$\begin{bmatrix}
\alpha^2&\alpha\\
\alpha&0\\
0&\alpha^2\\
\end{bmatrix}$,
$\begin{bmatrix}
\alpha^2&\alpha\\
\alpha&\alpha^2\\
\end{bmatrix}$
\end{center}
\end{lemma}

\begin{proof}
The letters {\color{red}$A,B,\ldots,R$} will refer to the matrices in Table \ref{tab:Forbidden_Matrices}.

Combining Lemma \ref{lem:ones-and-zeros} with the fact that $A$, $B$, $C$, and $D$ cannot be submatrices of $P_1$, we may assume that every column of $P_1$ is, up to field automorphism and permuting of rows, either $[1,\alpha,\alpha,0,\ldots,0]^T$ or $[\alpha,\alpha^2,0,\ldots,0]^T$.

We now analyze the possible forms $P_1$ can take if it contains exactly two columns. Tables \ref{tab:Case1}--\ref{tab:Case4}
\begin{table}[!htbp]
\caption{Candidate Matrices---Case 1}
\label{tab:Case1}
\begin{center}
\begin{tabular}{|c|c||c|c|}
\hline
Candidate Matrix&Forbidden Matrix&Candidate Matrix&Forbidden Matrix\\
\hline
$\begin{bmatrix}
1&0\\
\alpha&0\\
\alpha&0\\
0&\alpha\\
0&\alpha^2\\
\end{bmatrix}$&$E$&$\begin{bmatrix}
1&0\\
\alpha&0\\
\alpha&\alpha\\
0&\alpha^2\\
\end{bmatrix}$&$G$\\
\hline
$\begin{bmatrix}
1&0\\
\alpha&0\\
\alpha&\alpha^2\\
0&\alpha\\
\end{bmatrix}$&$H$&$\begin{bmatrix}
1&\alpha\\
\alpha&0\\
\alpha&0\\
0&\alpha^2\\
\end{bmatrix}$&$F$\\
\hline
$\begin{bmatrix}
1&\alpha^2\\
\alpha&0\\
\alpha&0\\
0&\alpha\\
\end{bmatrix}$&$E$&$\begin{bmatrix}
1&\alpha\\
\alpha&\alpha^2\\
\alpha&0\\
\end{bmatrix}$&None\\
\hline
$\begin{bmatrix}
1&\alpha^2\\
\alpha&\alpha\\
\alpha&0\\
\end{bmatrix}$&$J$&$\begin{bmatrix}
1&0\\
\alpha&\alpha\\
\alpha&\alpha^2\\
\end{bmatrix}$&$H$\\
\hline
\end{tabular}
\end{center}
\end{table}
list matrices that are candidates for being submatrices of $P_1$. If
\begin{table}[!htbp]
\caption{Candidate Matrices---Case 2}
\label{tab:Case2}
\begin{center}
\begin{tabular}{|c|c||c|c|}
\hline
Candidate Matrix&Forbidden Matrix&Candidate Matrix&Forbidden Matrix\\
\hline
$\begin{bmatrix}
1&0\\
\alpha&0\\
\alpha&0\\
0&1\\
0&\alpha\\
0&\alpha\\
\end{bmatrix}$&$E$&$\begin{bmatrix}
1&1\\
\alpha&0\\
\alpha&0\\
0&\alpha\\
0&\alpha\\
\end{bmatrix}$&$E$\\
\hline
$\begin{bmatrix}
1&0\\
\alpha&0\\
\alpha&\alpha\\
0&\alpha\\
0&1\\
\end{bmatrix}$&$J$&$\begin{bmatrix}
1&0\\
\alpha&0\\
\alpha&1\\
0&\alpha\\
0&\alpha\\
\end{bmatrix}$&$E$\\
\hline
$\begin{bmatrix}
1&1\\
\alpha&\alpha\\
\alpha&0\\
0&\alpha\\
\end{bmatrix}$&$J$&$\begin{bmatrix}
1&\alpha\\
\alpha&\alpha\\
\alpha&0\\
0&1\\
\end{bmatrix}$&$J$\\
\hline
$\begin{bmatrix}
1&0\\
\alpha&\alpha\\
\alpha&\alpha\\
0&1\\
\end{bmatrix}$&$G$&{\color{red}$\begin{bmatrix}
1&\alpha\\
\alpha&1\\
\alpha&0\\
0&\alpha\\
\end{bmatrix}$}&$R$\\
\hline
$\begin{bmatrix}
1&\alpha\\
\alpha&\alpha\\
\alpha&1\\
\end{bmatrix}$&None&&\\
\hline
\end{tabular}
\end{center}
\end{table}
such a matrix cannot occur as a submatrix of $P_1$, the second column of the table indicates the presence of a submatrix forbidden by Lemma \ref{lem:P1forbidden-1or2columns}. Table \ref{tab:Case1} considers the case where one column is of the type $[1,\alpha,\alpha,0,\ldots,0]^T$ and the other column is of the type $[\alpha,\alpha^2,0,\ldots,0]^T$. Up to a field automorphism, this is the same as the case where one column is of the type $[1,\alpha^2,\alpha^2,0,\ldots,0]^T$ and the other column is of the type $[\alpha,\alpha^2,0,\ldots,0]^T$. Table \ref{tab:Case2} considers the case where both columns are of the type $[1,\alpha,\alpha,0,\ldots,0]^T$. Up to a field automorphism, this is the same as the case where both columns are of the type $[1,\alpha^2,\alpha^2,0,\ldots,0]^T$. Table \ref{tab:Case3} considers the case where one column is of the type $[1,\alpha,\alpha,0,\ldots,0]^T$ and the other column is of the type $[1,\alpha^2,\alpha^2,0,\ldots,0]^T$. Table \ref{tab:Case4} considers the case where both columns are of the type $[\alpha,\alpha^2,0,\ldots,0]^T$.

\begin{table}[!htbp]
\caption{Candidate Matrices---Case 3}
\label{tab:Case3}
\begin{center}
\begin{tabular}{|c|c||c|c|}
\hline
Candidate Matrix&Forbidden Matrix&Candidate Matrix&Forbidden Matrix\\
\hline
$\begin{bmatrix}
1&0\\
\alpha&0\\
\alpha&0\\
0&1\\
0&\alpha^2\\
0&\alpha^2\\
\end{bmatrix}$&$F$
&
$\begin{bmatrix}
1&1\\
\alpha&0\\
\alpha&0\\
0&\alpha^2\\
0&\alpha^2\\
\end{bmatrix}$&$F$\\
\hline
$\begin{bmatrix}
1&0\\
\alpha&0\\
\alpha&\alpha^2\\
0&\alpha^2\\
0&1\\
\end{bmatrix}$&$H$
&
$\begin{bmatrix}
1&\alpha^2\\
\alpha&0\\
\alpha&0\\
0&1\\
0&\alpha^2\\
\end{bmatrix}$&$F$\\
\hline
$\begin{bmatrix}
1&0\\
\alpha&0\\
\alpha&1\\
0&\alpha^2\\
0&\alpha^2\\
\end{bmatrix}$&$F$ (up to field automorphism)&
$\begin{bmatrix}
1&1\\
\alpha&\alpha^2\\
\alpha&0\\
0&\alpha^2\\
\end{bmatrix}$&$K$\\
\hline
$\begin{bmatrix}
1&\alpha^2\\
\alpha&\alpha^2\\
\alpha&0\\
0&1\\
\end{bmatrix}$&$H$ (up to field automorphism)&
$\begin{bmatrix}
1&\alpha^2\\
\alpha&1\\
\alpha&0\\
0&\alpha^2\\
\end{bmatrix}$&$I$\\
\hline
$\begin{bmatrix}
1&0\\
\alpha&1\\
\alpha&\alpha^2\\
0&\alpha^2\\
\end{bmatrix}$&$H$&
$\begin{bmatrix}
1&0\\
\alpha&\alpha^2\\
\alpha&\alpha^2\\
0&1\\
\end{bmatrix}$&$H$\\
\hline
$\begin{bmatrix}
1&\alpha^2\\
\alpha&\alpha^2\\
\alpha&1\\
\end{bmatrix}$&$I$&
$\begin{bmatrix}
1&1\\
\alpha&\alpha^2\\
\alpha&\alpha^2\\
\end{bmatrix}$&$K$\\
\hline
\end{tabular}
\end{center}
\end{table}

\begin{table}[!htbp]
\caption{Candidate Matrices---Case 4}
\label{tab:Case4}
\begin{center}
\begin{tabular}{|c|c||c|c|}
\hline
Candidate Matrix&Forbidden Matrix&Candidate Matrix&Forbidden Matrix\\
\hline
$\begin{bmatrix}
\alpha^2&0\\
\alpha&0\\
0&\alpha\\
0&\alpha^2\\
\end{bmatrix}$&$L$&$\begin{bmatrix}
\alpha^2&\alpha\\
\alpha&0\\
0&\alpha^2\\
\end{bmatrix}$&None\\
\hline
$\begin{bmatrix}
\alpha^2&\alpha^2\\
\alpha&0\\
0&\alpha\\
\end{bmatrix}$&None&$\begin{bmatrix}
\alpha^2&\alpha\\
\alpha&\alpha^2\\
\end{bmatrix}$&None\\
\hline
\end{tabular}
\end{center}
\end{table}

We see from Tables \ref{tab:Case1}--\ref{tab:Case4} that the only matrices of two columns that are allowed to be contained in $P_1$ are the ones claimed in the result. Up to field automorphism and permuting of rows and columns, the only matrices with three columns such that every pair of columns consists of one of the permissible matrices are {\color{red}listed in Table \ref{tab:3columns}.

\begin{table}[!htbp]
\caption{{\color{red}Candidate Matrices---Three Columns}}
\label{tab:3columns}
\begin{center}
{\color{red}\begin{tabular}{|c|c||c|c|}
\hline
Candidate Matrix&Forbidden Matrix&Candidate Matrix&Forbidden Matrix\\
\hline
$\begin{bmatrix}
1&\alpha&\alpha\\
\alpha&1&\alpha\\
\alpha&\alpha&1\\
\end{bmatrix}$&$O$&$\begin{bmatrix}
1&\alpha&\alpha\\
\alpha&\alpha^2&0\\
\alpha&0&\alpha^2\\
\end{bmatrix}$&$M$\\
\hline
$\begin{bmatrix}
\alpha^2&\alpha^2&\alpha^2\\
\alpha&0&0\\
0&\alpha&0\\
0&0&\alpha\\
\end{bmatrix}$&$S$&$\begin{bmatrix}
\alpha^2&\alpha^2&\alpha\\
\alpha&0&0\\
0&\alpha&0\\
0&0&\alpha^2\\
\end{bmatrix}$&$T$\\
\hline
$\begin{bmatrix}
\alpha^2&\alpha^2&0\\
\alpha&0&\alpha^2\\
0&\alpha&\alpha\\
\end{bmatrix}$&$P$&$\begin{bmatrix}
\alpha^2&\alpha&0\\
\alpha&0&\alpha^2\\
0&\alpha^2&\alpha\\
\end{bmatrix}$&$Q$\\
\hline
$\begin{bmatrix}
\alpha^2&\alpha&\alpha\\
\alpha&\alpha^2&0\\
0&0&\alpha^2\\
\end{bmatrix}$&$N$&&\\
\hline
\end{tabular}}
\end{center}
\end{table}
Each of the matrices in Table \ref{tab:3columns} contains some forbidden matrix from either Table \ref{tab:Forbidden_Matrices} or Lemma \ref{lem:alpha-times-identity}. This completes the proof of the lemma.}
\end{proof}

Recall from Section \ref{sec:Preliminaries} that we write $f(r)\approx g(r)$ to denote that $f(r)=g(r)$ for all $r$ sufficiently large.

\begin{theorem}
\label{thm:golden-mean-extremal}
Suppose Hypothesis \ref{hyp:clique-template} holds. For all sufficiently large $r$, the extremal matroids of $\mathcal{AC}_4$ and $\mathcal{GM}$ are $T_r^2$, $G_r$, and $HP_r$. Thus, we have $h_{\mathcal{AC}_4}(r)\approx h_{\mathcal{GM}}(r)\approx\binom{r+3}{2}-5$.
\end{theorem}

\begin{proof}
Since graphic matroids are regular, $\mathcal{AC}_4$ and $\mathcal{GM}$ each contain the class of graphic matroids. Also note that there is no finite field $\F=\GF(p^k)$ such that $\mathcal{AC}_4$ is contained in the class of $\F$-representable matroids. Therefore, by Theorem \ref{thm:growthrate}, both $\mathcal{AC}_4$ and $\mathcal{GM}$ are quadratically dense.

By Lemma \ref{lem:extremal-virtual}, the extremal matroids of $\mathcal{AC}_4$ are the largest simple matroids that virtually conform to some template in the set $\{\Phi_1,\ldots,\Phi_s\}$ whose existence is implied by Hypothesis \ref{hyp:clique-template}. By Lemma \ref{lem:AC4-Y-template}, these are $Y$-templates.

Let $\Phi=\YT(P_0,P_1)$ be an extremal template for $\mathcal{AC}_4$. We know from the discussion at the beginning of this section that the extremal templates for $\mathcal{AC}_4$ are those templates where $P_1$ has the most columns. By Lemma \ref{lem:allowed-matrices}, $P_1$ has two columns. Recall from Definition \ref{def:semi-strong} that semi-strongly equivalent templates have the same universal matroids. Since the extremal matroids of a template are obtained by simplifying the universal matroids, Lemma \ref{lem:ones-and-zeros} implies that we may assume that each column of $P_0$ has entries that sum to $0$ and that each column of $P_1$ has entries that sum to $1$.

\begin{claim}
\label{cla:no-zero-rows}
If $\Phi=\YT(P_0,P_1)$ is an extremal template for $\mathcal{AC}_4$, then $\Phi$ is semi-strongly equivalent to a $Y$-template $\Phi=\YT(P_0',P_1')$ where every row of $P_1'$ is nonzero. That is, $P_1'$ can be chosen so that it is exactly one of the matrices given in Lemma \ref{lem:allowed-matrices}. Moreover, the sum of the rows of $P_0'$ is the zero vector.
\end{claim}

\begin{subproof}
By Lemma \ref{lem:ones-and-zeros}, we may assume that the sum of the rows of $P_1$ is $[1,\ldots,1]$ and that the sum of the rows of $P_0$ is the zero vector.

Let $\textbf{v}$ and $\textbf{w}$ be the columns of the matrix that results when the zero rows of $P_1$ are removed. Suppose $P_0$ contains a column with a nonzero entry in one of the rows corresponding to a zero row of $P_1$. By scaling, we may assume that that entry is $1$. Then $P_0$ and $P_1$ (with columns indexed by $Y_0$ and $Y_1$) contain the following submatrix.

\begin{center}
\begin{tabular}{ |c|c|c| }
\multicolumn{2}{c}{$Y_1$}&\multicolumn{1}{c}{$Y_0$}\\
\hline
$\textbf{v}$&$\textbf{w}$&$\textbf{u}$\\
\hline
$0$&$0$&$1$\\
\hline
\end{tabular}
\end{center}

Since the matrix $A_1[X,Y_1]$ contains an identity matrix in addition to $P_1$, by contracting this element of $Y_0$, we obtain the following submatrix in $A_1[X,Y_1]$.
\begin{center}
\begin{tabular}{ |c|c|c| }
\multicolumn{3}{c}{$Y_1$}\\
\hline
$\textbf{v}$&$\textbf{w}$&$\textbf{u}$\\
\hline
\end{tabular}
\end{center}
Since $P_1$ can have at most two columns,  either $\textbf{u}$ is a unit column or $\textbf{u}$ is equal to $\textbf{v}$ or $\textbf{w}$.

Thus, $P_0$ must be of the following form, where $T$ is an arbitrary matrix the sum of whose rows is the zero vector.
\begin{center}
\begin{tabular}{ |c|c|c|c| }
\hline
$\textbf{v}\ldots\textbf{v}$&$\textbf{w}\ldots\textbf{w}$&unit columns&$T$\\
\hline
\multicolumn{3}{|c|}{unit columns}&$0$\\
\hline
\end{tabular}
\end{center}

In fact, since $\Phi$ is an extremal template, $P_0$ must be the following for some positive integer $n$.
\begin{center}
\begin{tabular}{ |c|c|c|c| }
\hline
$\textbf{v}\ldots\textbf{v}$&$\textbf{w}\ldots\textbf{w}$&\multirow{2}{*}{$D_n$}&$T$\\
\cline{1-2}
\cline{4-4}
$I$&$I$&&$0$\\
\hline
\end{tabular}
\end{center}

The rank-$r$ universal matroid of this template is isomorphic to the rank-$r$ universal matroid of $\YT([D_m|T],[\textbf{v}|\textbf{w}])$, where $m$ is the number of rows in $[\textbf{v}|\textbf{w}]$. Thus, these two templates are semi-strongly equivalent, and we may take $P_0'=[D_m|T]$ and $P_1'=[\textbf{v}|\textbf{w}]$.
\end{subproof}

Before we can analyze what happens when $P_1$ is one of the matrices listed in Lemma \ref{lem:allowed-matrices}, we need two more claims.

\begin{claim}
\label{cla:Gr-first-matrix}
The extremal matroids for $\Phi=\YT\left(\begin{bmatrix}1&1\\1&\alpha^2\\\end{bmatrix},\begin{bmatrix}1&\alpha\\\alpha&0\\\end{bmatrix}\right)$ are isomorphic to those of $\Phi(G_r)$.
\end{claim}

\begin{subproof}
The rank-$r$ extremal matroid of $\Phi(G_r)$ (obtained by simplifying the universal matroid) has the following representation matrix.
\begin{center}
\begin{tabular}{ |c|c|c|c|c|c|ccccc|}
\hline
$0\cdots0$&$0\cdots0$&$1\cdots1$&$0\cdots0$&$\alpha\cdots\alpha$&$0\cdots0$&$1$&$0$&$1$&$1$&$1$\\
$0\cdots0$&$0\cdots0$&$0\cdots0$&$1\cdots1$&$0\cdots0$&$\alpha\cdots\alpha$&$0$&$1$&$1$&$\alpha$&$\alpha^2$\\
\hline
$I_{r-|X|}$&$D_{r-|X|}$&$I_{r-|X|}$&$I_{r-|X|}$&$I_{r-|X|}$&$I_{r-|X|}$&\multicolumn{5}{c|}{$0$}\\
\hline
\end{tabular}
\end{center}
To the first row of this matrix, add all other rows. Then scale the first row by $\alpha$. The result is the following.
\begin{center}
\begin{tabular}{ |c|c|c|c|c|c|ccccc|}
\hline
$\alpha\cdots\alpha$&$0\cdots0$&$0\cdots0$&$0\cdots0$&$1\cdots1$&$1\cdots1$&$\alpha$&$\alpha$&$0$&$1$&$\alpha^2$\\
$0\cdots0$&$0\cdots0$&$0\cdots0$&$1\cdots1$&$0\cdots0$&$\alpha\cdots\alpha$&$0$&$1$&$1$&$\alpha$&$\alpha^2$\\
\hline
$I_{r-|X|}$&$D_{r-|X|}$&$I_{r-|X|}$&$I_{r-|X|}$&$I_{r-|X|}$&$I_{r-|X|}$&\multicolumn{5}{c|}{$0$}\\
\hline
\end{tabular}
\end{center}
Scale the last five columns so that their first nonzero entries are $1$ and reorder the columns of the entire matrix to obtain the following, which is a representation matrix of the rank-$r$ extremal matroid of $\Phi$.
\begin{center}
\begin{tabular}{ |c|c|c|c|c|c|ccccc|}
\hline
$0\cdots0$&$0\cdots0$&$1\cdots1$&$0\cdots0$&$1\cdots1$&$\alpha\cdots\alpha$&$1$&$0$&$1$&$1$&$1$\\
$0\cdots0$&$0\cdots0$&$0\cdots0$&$1\cdots1$&$\alpha\cdots\alpha$&$0\cdots0$&$0$&$1$&$1$&$\alpha$&$\alpha^2$\\
\hline
$I_{r-|X|}$&$D_{r-|X|}$&$I_{r-|X|}$&$I_{r-|X|}$&$I_{r-|X|}$&$I_{r-|X|}$&\multicolumn{5}{c|}{$0$}\\
\hline
\end{tabular}
\end{center}
\end{subproof}

\begin{claim}
\label{cla:Gr-third-matrix}
The extremal matroids for $\Phi=\YT\left(\begin{bmatrix}1&1&1\\1&\alpha&\alpha^2\\\end{bmatrix},\begin{bmatrix}\alpha&0\\0&\alpha^2\\\end{bmatrix}\right)$ are isomorphic to those of $\Phi(G_r)$.
\end{claim}

\begin{subproof}
Consider the representation matrix for the rank-$r$ extremal matroid of $\Phi(G_r)$ given at the beginning of the proof of Claim \ref{cla:Gr-first-matrix}. Scaling the second row by $\alpha^2$, we obtain the following.
\begin{center}
\begin{tabular}{ |c|c|c|c|c|c|ccccc|}
\hline
$0\cdots0$&$0\cdots0$&$1\cdots1$&$0\cdots0$&$\alpha\cdots\alpha$&$0\cdots0$&$1$&$0$&$1$&$1$&$1$\\
$0\cdots0$&$0\cdots0$&$0\cdots0$&$\alpha^2\cdots\alpha^2$&$0\cdots0$&$1\cdots1$&$0$&$\alpha^2$&$\alpha^2$&$1$&$\alpha$\\
\hline
$I_{r-|X|}$&$D_{r-|X|}$&$I_{r-|X|}$&$I_{r-|X|}$&$I_{r-|X|}$&$I_{r-|X|}$&\multicolumn{5}{c|}{$0$}\\
\hline
\end{tabular}
\end{center} Scale the fourth from last column to once again make its nonzero entry $1$, and reorder the columns to obtain the following, which is a representation matrix of the rank-$r$ extremal matroid for $\Phi$.
\begin{center}
\begin{tabular}{ |c|c|c|c|c|c|ccccc|}
\hline
$0\cdots0$&$0\cdots0$&$1\cdots1$&$0\cdots0$&$\alpha\cdots\alpha$&$0\cdots0$&$1$&$0$&$1$&$1$&$1$\\
$0\cdots0$&$0\cdots0$&$0\cdots0$&$1\cdots1$&$0\cdots0$&$\alpha^2\cdots\alpha^2$&$0$&$1$&$1$&$\alpha$&$\alpha^2$\\
\hline
$I_{r-|X|}$&$D_{r-|X|}$&$I_{r-|X|}$&$I_{r-|X|}$&$I_{r-|X|}$&$I_{r-|X|}$&\multicolumn{5}{c|}{$0$}\\
\hline
\end{tabular}
\end{center}
\end{subproof}

We are now ready to analyze what happens when $P_1$ is one of the matrices listed in Lemma \ref{lem:allowed-matrices}.

\begin{claim}
\label{cla:three-row-matrices}
Let $P_1$ be one of the matrices listed in Table \ref{tab:AC-matrices} and let $M$ be the corresponding matroid. Let $r\geq3$, and consider all templates of the form $\YT(P_0,P_1)$, where the sum of the rows of $P_0$ is the zero vector. The largest simple matroid of rank $r$ virtually conforming to any such template is $M$.
\end{claim}

\begin{table}[!htbp]
\caption{$\AC$ Matrices}
\label{tab:AC-matrices}
\begin{center}
\begin{tabular}{|c|c||c|c|}
\hline
Matrix&Extremal Matroid&Matrix&Extremal Matroid\\
\hline
$\begin{bmatrix}
1&\alpha\\
\alpha&\alpha^2\\
\alpha&0\\
\end{bmatrix}$&$G_r$&$\begin{bmatrix}
1&\alpha\\
\alpha&\alpha\\
\alpha&1\\
\end{bmatrix}$&$HP_r$\\
\hline
$\begin{bmatrix}
\alpha^2&\alpha^2\\
\alpha&0\\
0&\alpha\\
\end{bmatrix}$&$G_r$&$\begin{bmatrix}
\alpha^2&\alpha\\
\alpha&0\\
0&\alpha^2\\
\end{bmatrix}$&$G_r$\\
\hline
\end{tabular}
\end{center}
\end{table} 

\begin{subproof}
Let $\Phi=\YT(P_0,P_1)$. Because we wish to find the largest simple matroid of rank $r$ that virtually conforms to a template of the form described in the claim, we must have that $P_0$ has the largest number of columns possible that are not zero columns, scalar multiples of each other, or unit columns (because these can be built using $Y_1$). Thus, we must have $P_0=\begin{bmatrix}
1&1&1&1&0\\
\alpha&\alpha^2&1&0&1\\
\alpha^2&\alpha&0&1&1\\
\end{bmatrix}$.

Choose any row of $[P_1|P_0]$, and let $P_0'$ and $P_1'$ be the results of removing that row from $P_0$ and $P_1$, respectively. By Lemma \ref{lem:ones-and-zeros} and the fact that the extremal matroids of a template come from simplifying the universal matroids,  we see that the extremal matroids of $\Phi$ are isomorphic to the extremal matroids of $\Phi'=\YT(P_0',P_1')$. After scaling columns appropriately, we see that the columns of $P_0'$ are the same as the columns of $\begin{bmatrix}1&0&1&1&1\\0&1&1&\alpha&\alpha^2\\\end{bmatrix}$.

We will now do the procedure described in the previous paragraph for each matrix $P_1$ listed in Table \ref{tab:AC-matrices}. In each case, we will see that the extremal matroids of $\Phi'$ are isomorphic to those of a template we have already considered. In the case of the matrix in the top left of Table \ref{tab:AC-matrices}, remove the middle row. The result is a template whose extremal matroids are isomorphic to those of the template considered in Claim \ref{cla:Gr-first-matrix}, which are $G_r$. For the matrix in the top right of Table \ref{tab:AC-matrices}, remove the middle row. The result is a template whose extremal matroids are isomorphic to those of $\Phi(HP_r)$. For the matrix in the bottom left of Table \ref{tab:AC-matrices}, remove the top row. The result is a template whose extremal matroids are isomorphic to those of $\Phi(G_r)$. For the matrix in the bottom right of Table \ref{tab:AC-matrices}, remove the top row. The result is a template whose extremal matroids are isomorphic to those of the template considered in Claim \ref{cla:Gr-third-matrix}, which are $G_r$.
\end{subproof}

\begin{claim}
\label{cla:fifth-allowed-matrix}
If $P_0=D_2$ and $P_1=\begin{bmatrix}
\alpha^2&\alpha\\
\alpha&\alpha^2\\
\end{bmatrix}$, then the rank-$r$ extremal matroid for $\YT(P_0,P_1)$ is $T_r^2$ for all $r\geq2$.
\end{claim}

\begin{subproof}
We perform a procedure similar to that done in the proof of the previous claim, removing the top row of $[P_1|P_0]$.
\end{subproof}

Combining Claim \ref{cla:no-zero-rows} with Claims \ref{cla:three-row-matrices} and \ref{cla:fifth-allowed-matrix}, we see that the extremal matroids for $\mathcal{AC}_4$ are $T_r^2$, $G_r$, and $HP_r$. Since $\mathcal{GM}\subseteq\mathcal{AC}_4$, we have $h_{\mathcal{GM}}(r)\leq h_{\mathcal{AC}_4}(r)$. By Proposition \ref{pro:AW-golden-mean}, $T_r^2$, $G_r$, and $HP_r$ are all golden-mean matroids. Thus, we have $h_{\mathcal{GM}}(r)\approx h_{\mathcal{AC}_4}(r)$. It is easily verified (see also \cite{w14}) that $\varepsilon(T_r^2)=\varepsilon(G_r)=\varepsilon(HP_r)=\binom{r+3}{2}-5$. This completes the proof.
\end{proof}

\section{Maximal Templates}
\label{sec:Maximal Templates}

In this section and Section \ref{sec:The Highly Connected Matroids in AC4}, we will determine a collection $\T$ of $Y$-templates over $\GF(4)$ such that, for each template $\Phi\in\T$, we have $\mathcal{M}(\Phi)\subseteq\mathcal{AC}_4$ and such that, for every refined template $\Phi'$ with $\mathcal{M}(\Phi')\subseteq\mathcal{AC}_4$, there is a template $\Phi\in\T$ such that $\Phi'\preceq\Phi$. Our motivation is to use Hypotheses \ref{hyp:connected-template} and \ref{hyp:clique-template} and Corollaries \ref{cor:weak-connected-template} and \ref{cor:weak-cliquetemplate} to study $\mathcal{AC}_4$ and $\mathcal{GM}$; however we will not refer to these hypotheses and corollaries again until later in Section \ref{sec:The Highly Connected Golden-Mean Matroids}. This will allow the next several results to be free from some of the inherent technicalities in those hypotheses and corollaries, and it will also illustrate that the results in this section are independent of the hypotheses.

\begin{lemma}
\label{lem:contract}
Let $\Phi=\YT(P_0,P_1)$ be such that $\mathcal{M}(\Phi)\subseteq\mathcal{AC}_4$. Suppose $P_0$ contains a contractible submatrix $Q'$, as given in Definition \ref{def:contractible}. Then $Q$ must be a submatrix of one of the matrices listed in Lemma \ref{lem:allowed-matrices} (up to permutations and field isomorphism). In particular, $Q$ and $Q'$ have at most two columns, and $Q'$ has at most five rows.
\end{lemma}

\begin{proof}
From $\Phi$, elements of $Y_0$ can be deleted (operation (10) of Definition \ref{def:weaklyconforming}) and elements of $Y_1$ can be $y$-shifted (see Lemma \ref{yshift}) and contracted (operation (12) of Definition \ref{def:weaklyconforming}) to obtain $\YT(Q',[\emptyset])$. Then the remaining elements of $Y_0$ can be contracted (operation (12) of Definition \ref{def:weaklyconforming}) to obtain $\YT([\emptyset],Q)$. By Lemma \ref{lem:allowed-matrices}, the result holds.
\end{proof}

Note that in the previous lemma, the requirement in Definition \ref{def:contractible} about semi-parallel extensions is necessary because, otherwise, contracting the semi-parallel columns results in equal columns of $P_1$. Whenever columns of $P_1$ are equal, all but one of them can be discarded to produce a new template strongly equivalent to the original template.

\begin{lemma}
\label{lem:5-rows}
Let $\Phi_{P_0}$ be a complete, lifted $Y$-template such that $\mathcal{M}(\Phi_{P_0})\subseteq\AC$. There is a matrix $P_0'$, with at most five rows, such that every matroid conforming to $\Phi_{P_0}$ is a minor of a matroid conforming to $\Phi_{P_0'}$. Moreover, every column of $P_0'$ is of the form $[\alpha,\alpha^2,1,0,\ldots,0]^T$ or $[1,\alpha,\alpha,1,0,\ldots,0]^T$, up to column scaling, permuting rows, and field isomorphism. If $P_0$ contains no $5\times2$ contractible submatrix, then $P_0'$ has at most four rows.
\end{lemma}

\begin{proof}
Recall from Lemma \ref{lem:AC4-Y-template} that, if $\Phi$ is a refined template such that $\mathcal{M}(\Phi)\subseteq\AC$, then $\Phi$ is a $Y$-template. Combining Remark \ref{rem:lifted-Y-template} and Lemma \ref{lem:complete}, we see that every $Y$-template is minor equivalent to a complete, lifted $Y$-template $\Phi_{P_0}$ determined by some matrix $P_0$. Moreover, by Lemma \ref{lem:sum-to-zero}, we may assume that the sum of the rows of $P_0$ is the zero vector. By Lemma \ref{lem:contract}, $[1,1,1]^T$, and $[{\color{red}\alpha,\alpha,\alpha^2},1]^T$, and $[1,\alpha,\alpha^2,1]^T$ are all forbidden from $P_0$. Thus, up to column scaling, permuting rows, and field isomorphism, each column of $P_0$ must be of the form $[\alpha,\alpha^2,1,0,\ldots,0]^T$ or $[1,\alpha,\alpha,1,0,\ldots,0]^T$. (Graphic columns are already assumed in a complete template.)

We will show that every quaternary matrix $P_0$ such that $\mathcal{M}(\Phi_{P_0})\subseteq\AC$ must satisfy the hypotheses of Lemma \ref{lem:contracible}, with $r\leq5$. The result follows from combining Lemma \ref{lem:contracible} with Lemma \ref{lem:contract}.

Recall that the weight of a vector is its number of nonzero entries. There are five cases to consider.
\begin{enumerate}
\item The matrix $P_0$ contains a contractible submatrix with two columns each of which have weight $4$.
\item Case 1 does not hold, but $P_0$ contains a contractible submatrix with one column of weight $4$ and another column of weight $3$.
\item Case 1 and Case 2 do not hold, but $P_0$ contains a column of weight $4$.
\item Every column of $P_0$ has weight $3$; there are weight-$3$ columns of $P_0$ with supports whose intersection has size exactly $1$.
\item Every column of $P_0$ has weight $3$, and there are no pairs of columns with supports whose intersection has size $1$.
\end{enumerate}

In Cases 1--5, respectively, the contractible submatrix with the most possible rows is as follows.
\begin{equation}
\label{equ:Cases}
\begin{bmatrix}
1&\alpha\\
\alpha&\alpha\\
\alpha&1\\
1&0\\
0&1\\
\end{bmatrix}
\begin{bmatrix}
1&\alpha\\
\alpha&\alpha^2\\
\alpha&0\\
1&0\\
0&1\\
\end{bmatrix}
\begin{bmatrix}
1\\
\alpha\\
\alpha\\
1
\end{bmatrix}
\begin{bmatrix}
\alpha&\alpha\\
\alpha^2&0\\
0&\alpha^2\\
1&0\\
0&1\\
\end{bmatrix}
\begin{bmatrix}
\alpha&\alpha^2\\
\alpha^2&\alpha\\
1&0\\
0&1\\
\end{bmatrix}
\end{equation}

In each case, let $Q''$ be the column submatrix of $P_0$ containing the same columns as the given contractible submatrix. 

In Case 1, by Lemmas \ref{lem:contract} and \ref{lem:allowed-matrices}, the contractible submatrix given in Case 1 can be chosen to be the first of the Matrices \ref{equ:Cases}. If an additional weight-$4$ column of $P_0$ has a nonzero entry outside of the first five rows, then the first three entries of that column must be nonzero; otherwise a contractible submatrix with two columns and at least six rows is formed. Thus, this additional column must be a semi-parallel extension of $Q''$. If a weight-$3$ column has a nonzero entry outside of the first five rows, scale the column so that this entry is a $1$. If the other two nonzero entries are contained in the first three rows, then a $6\times3$ contractible submatrix is formed. Otherwise, a contractible submatrix with two columns and at least six rows is formed.

In Case 2, the contractible submatrix is the second of the Matrices \ref{equ:Cases}. Label the columns of this matrix as $1$ and $2$. Suppose $P_0$ contains a third column $v$ that is not a semi-parallel extension of $Q''$, with a nonzero entry in a row other than the first five rows. Since Case 1 does not hold, $v$ must have weight $3$. By Lemmas \ref{lem:contract} and \ref{lem:allowed-matrices}, the intersection of the supports of $1$ and $v$ must have size $2$. Therefore, $v$ has at most one nonzero entry in a row other than the first four rows. By column scaling, let $v=[a,b,c,d,0,1,0,\ldots,0]^T$, where two members of $\{a,b,c,d\}$ are $0$. If either $c=0$ or $d=0$, then since $v$ is not a semi-parallel extension of $Q''$, we obtain a $6\times3$ contractible submatrix, which is forbidden. Therefore, $c\neq0$ and $d\neq0$. But then columns $2$ and $v$ have disjoint supports, resulting in a forbidden $6\times2$ contractible submatrix.

In Case 3, since we do not have Case 1 or 2, any column with a nonzero entry outside the first four rows must be semi-parallel to the given column $[1,\alpha,\alpha,1,0,\ldots,0]^T$. Thus, by column scaling and permuting rows, we may assume that it is $[1,\alpha,\alpha,0,1,0,\ldots,0]^T$. Let $R$ be the column submatrix of $P_0$ consisting of these two columns, and label these columns as $1$ and $2$. Now, suppose there is a third column and that this column has a nonzero entry outside of the first five rows. Using an argument similar to the one in Case 1, we see that this column must be a semi-parallel extension of $R$. The column $[0,\alpha,\alpha,1,1,0,\ldots,0]^T$ is semi-parallel to both $1$ and $2$. But, if $P_0$ contains all three of these columns, then SageMath tells us that $M([I|D|P_0])$ contains $F_7^*$ as a minor. Therefore, $P_0$ satisfies the hypotheses of Lemma \ref{lem:contracible}, with $r=4$.

In Case 4, the contractible submatrix is the fourth of the Matrices \ref{equ:Cases}. Label the columns of this matrix as $1$ and $2$. Suppose $P_0$ contains a third column $v$ that is not a semi-parallel extension of $Q''$, with a nonzero entry in a row other than the first five rows. Without loss of generality, we may assume that this nonzero entry is in the sixth row, and that for some values $a,b,c,d,e$, the column is scaled to be $v=[a,b,c,d,e,1,0\ldots,0]^T$. In order to avoid a $6\times3$ contractible submatrix, at least three members of $\{b,c,d,e\}$ must be nonzero, but this is impossible since every column of $P_0$ has weight $3$.

In Case 5, the contractible submatrix with the most possible rows is the fifth of the Matrices \ref{equ:Cases}. Label the columns of this matrix as $1$ and $2$. Suppose $P_0$ contains a third column $v$ with a nonzero entry in a row other than the first four rows. Since all pairs of columns have supports whose intersections have size at least $2$, the column $v$ must be a semi-parallel extension of $Q''$. Therefore, $P_0$ satisfies the hypotheses of Lemma \ref{lem:contracible}, with $r=4$.
\end{proof}

\begin{definition}
\label{def:possible-matrices}
We define the following matrices.

\noindent
$I=\begin{bmatrix}
1&1&\alpha&\alpha&\alpha&\alpha^2\\
\alpha&\alpha&\alpha&\alpha&\alpha^2&\alpha\\
\alpha&\alpha&1&1&1&1\\
1&0&1&0&0&0\\
0&1&0&1&0&0\\
\end{bmatrix}$
$II=\begin{bmatrix}
1     &\alpha&1       &\alpha^2&0     \\
\alpha&\alpha&\alpha^2&0       &1     \\
\alpha&1     &0       &\alpha^2&\alpha\\
1     &0     &\alpha^2&1       &\alpha\\
0     &1     &1       &1       &1     \\
\end{bmatrix}$
$III=\begin{bmatrix}
1     &\alpha^2&1       &\alpha  &\alpha^2\\
\alpha&\alpha  &\alpha^2&1       &0       \\
\alpha&1       &\alpha^2&\alpha  &\alpha  \\
1     &0       &1       &1       &1       \\
\end{bmatrix}$
$IV=\begin{bmatrix}
1     &\alpha  &\alpha^2&\alpha  &0       \\
\alpha&\alpha^2&\alpha  &0       &\alpha^2\\
\alpha&1       &0       &\alpha^2&\alpha  \\
1     &0       &1       &1       &1       \\
\end{bmatrix}$
$V=\begin{bmatrix}
1&\alpha&0\\
\alpha&\alpha^2&0\\
\alpha&0&\alpha^2\\
1&0&\alpha\\
0&1&1\\
\end{bmatrix}$
$VI=\begin{bmatrix}
1     &\alpha  &\alpha^2&\alpha  &\alpha\\
\alpha&\alpha^2&1       &\alpha^2&\alpha^2\\
\alpha&0       &\alpha^2&1       &0\\
1     &0       &1       &0       &1\\
0     &1       &0       &0       &0\\
\end{bmatrix}$
$VII=\begin{bmatrix}
1     &1     &\alpha  &\alpha  &\alpha  &\alpha^2\\
\alpha&\alpha&\alpha^2&\alpha^2&\alpha^2&\alpha\\
\alpha&\alpha&1       &0       &0       &1\\
1     &0     &0       &1       &0       &0\\
0     &1     &0       &0       &1       &0\\
\end{bmatrix}$
$VIII=\begin{bmatrix}
1     &\alpha^2&\alpha  &\alpha^2\\
\alpha&\alpha  &0       &\alpha^2\\
\alpha&1       &\alpha^2&1       \\
1     &0       &1       &1       \\
\end{bmatrix}$
$IX=\begin{bmatrix}
1     &1       &\alpha  &\alpha^2\\
\alpha&\alpha^2&\alpha  &\alpha^2\\
\alpha&\alpha^2&1       &1       \\
1     &1       &1       &1       \\
\end{bmatrix}$
$X=\begin{bmatrix} 
\alpha     	&\alpha     	&\alpha		&\alpha^2	&\alpha\\
\alpha^2	&0		&\alpha^2	&0 		&0\\
0 		&\alpha^2	&1     		&\alpha 	&\alpha^2\\
1     		&0     		&0     		&0     		&1\\
0     		&1     		&0     		&1     		&0\\
\end{bmatrix}$
$XI=\begin{bmatrix} 
\alpha   	&\alpha 	&\alpha^2	&\alpha^2\\
\alpha^2	&0     		&0     		&\alpha\\
0 		&\alpha^2	&\alpha     	&0\\
1     		&0     		&0     		&1\\
0     		&1     		&1     		&0\\
\end{bmatrix}$
$XII=\begin{bmatrix} 
\alpha     	&\alpha    	&\alpha		&\alpha^2	&\alpha     	&\alpha		\\
\alpha^2	&0 		&\alpha^2	&\alpha		&\alpha^2	&0     		\\
0 		&\alpha^2	&1     		&1     		&0 		&\alpha^2	\\
1     		&0     		&0     		&0     		&0     		&1     		\\
0     		&1     		&0     		&0     		&1     		&0     		\\
\end{bmatrix}$
$XIII=\begin{bmatrix} 
\alpha     	&\alpha     	&\alpha     	&\alpha 	&\alpha^2\\
\alpha^2	&0 		&\alpha^2	&\alpha^2	&0     	\\
0 		&\alpha^2	&1     		&0     		&\alpha  \\
1     		&0     		&0     		&0     		&0     	\\
0     		&1     		&0     		&1     		&1     	\\
\end{bmatrix}$
$XIV=\begin{bmatrix} 
\alpha     	&\alpha     	&\alpha     	&\alpha 	&\alpha^2	&\alpha \\
\alpha^2	&0 		&\alpha^2	&\alpha^2	&\alpha     	&0 \\
0 		&\alpha^2	&1     		&0     		&0     		&0 \\
1     		&0     		&0     		&0     		&0     		&\alpha^2\\
0     		&1     		&0     		&1     		&1     		&1	 \\
\end{bmatrix}$

\noindent
$XV=\begin{bmatrix} 
\alpha     	&\alpha     	&\alpha^2     	&\alpha^2\\
\alpha^2	&\alpha^2 	&\alpha		&\alpha  \\
1 		&0     		&1     		&0       \\
0     		&1     		&0     		&1     	 \\
\end{bmatrix}$
$XVI=\begin{bmatrix}
1     &\alpha  &1       &0       \\
\alpha&\alpha^2&\alpha^2&\alpha^2\\
\alpha&1       &\alpha^2&\alpha   \\
1     &0       &1       &1       \\
\end{bmatrix}$
$XVII=\begin{bmatrix} 
\alpha     	&\alpha^2     	&0       	&\alpha  \\
\alpha^2	&\alpha 	&\alpha^2	&0       \\
1 		&0     		&\alpha    	&\alpha^2\\
0     		&1     		&1     		&1     	 \\
\end{bmatrix}$
$XVIII=\begin{bmatrix}
1&\alpha^2&\alpha^2\\
\alpha&\alpha&0\\
\alpha&1&\alpha\\
1&0&1\\
\end{bmatrix}$
\end{definition}

\begin{lemma}
\label{lem:possible-matrices}
Let $\Phi$ be a refined frame template over $\GF(4)$ such that $\mathcal{M}(\Phi)\subseteq\mathcal{AC}_4$. Then there is a complete, lifted $Y$-template $\Phi'$ that is determined by a column submatrix of one of matrices $I$--$XVI$ listed in Definition \ref{def:possible-matrices} (up to permutations and field isomorphism) such that every matroid conforming to $\Phi$ is a minor of a matroid conforming to $\Phi'$.
\end{lemma}

\begin{proof}
By Lemma \ref{lem:5-rows}, the proof of Lemma \ref{lem:possible-matrices} amounts to a finite case check which is aided by the use of SageMath. The code and some additional explanation is found in the ancillary files. This case check showed that, if $P_0$ is a quaternary matrix with at most five rows, where every column is of the form $[1,\alpha,\alpha,1,0,\ldots,0]^T$ or $[\alpha,\alpha^2,1,0,\ldots,0]^T$ (up to permuting rows and field isomorphism), and such that $P_0$ is not a submatrix of any of matrices $I$--$XVI$ (up to permuting rows, permuting columns, and field isomorphism), then the vector matroid of $[I|D|P_0]$ contains either $F_7$, $(F_7)^*$, $V_1$, $V_2$, $V_3$, $P_1$, $P_2$, or $P_3$ as a minor. Therefore, $\mathcal{M}(\Phi_{P_0})\nsubseteq\AC$. 
\end{proof}

Note that, in the statement of Lemma \ref{lem:possible-matrices}, we do not claim that $\mathcal{M}(\Phi')\subseteq\AC$. This is indeed true but will be proved in Section \ref{sec:The Highly Connected Matroids in AC4}. Also, note that eighteen matrices are listed in Definition \ref{def:possible-matrices}, but Lemma \ref{lem:possible-matrices} only deals with the first sixteen of them.

\section{The Highly Connected Matroids in $\mathcal{AC}_4$}
\label{sec:The Highly Connected Matroids in AC4}
We begin this section with eighteen lemmas---one for each of the matrices $I$--$XVIII$ listed in Definition \ref{def:possible-matrices}. Recall the construction of the universal partial field $\P_M$ of a matroid $M$ from Section \ref{sec:Partial-Fields}. If $P_0$ is one of these matrices, we will show that the universal partial field of $M=\widetilde{M}([I_r|D_r|P_0])$ is $\mathbb{U}_2$, $\mathbb{K}_2$, $\mathbb{G}$, or $\PT$. To do this, we use a series of functions implemented in SageMath. These functions are described conceptually in the \ref{sec:Appendix}. For Lemmas \ref{lem:I} and \ref{lem:V}, the \ref{sec:Appendix} gives the details of how the SageMath code was used to obtain the results. In this section, we will only give the sketches of the proofs of Lemmas \ref{lem:I}--\ref{lem:XVIII}, making statements that implicitly refer to SageMath. The proofs of all of the lemmas follow the basic pattern explained in Section \ref{sec:Characteristic Sets} and the \ref{sec:Appendix}.

 \begin{lemma}
\label{lem:I}
Let $P_0$ be matrix $I$ listed in Definition \ref{def:possible-matrices}. Then the universal partial field of $M=\widetilde{M}([I_5|D_5|P_0])$ is $\mathbb{G}$. Moreover, $M$ is only representable over a field if it contains a root of $x^2-x-1$.
 \end{lemma}

\begin{proof}
 The ideal is generated by $\{z_{11}^2 + z_{11} - 1$, $z_0 + z_{11}$, $z_1 - z_{11}$, $z_2 + z_{11}$, $z_3 - z_{11}$, $z_4 - z_{11}$, $z_5 + z_{11}$, $z_6 - z_{11}$, $z_7 + z_{11}$, $z_8 + z_{11}$, $z_9 - z_{11} + 1$, $z_{10} + z_{11} + 1\}$. We solve for the variables in terms of $z_2$ and obtain $z_0=z_2$, $z_1=-z_2$, $z_3=-z_2$, $z_4=-z_2$, $z_5=z_2$, $z_6=-z_2$, $z_7=z_2$, $z_8=z_2$, $z_9=-z_2-1$, $z_{10}=z_2-1$, and $z_{11}=-z_2$. Since $z_{11}^2 + z_{11} - 1$ and $z_2 + z_{11}$ are in the ideal, every field over which $M$ has a representation must contain a root of $x^2-x-1$. Moreover, we have $z_2+1=z_2^2$, $z_2-1=z_2^{-1}$, and $z_2^2 - 2z_2 = z_2^{-1}$. Therefore, we check that the matrix is a $\mathbb{G}$-matrix by checking the partial field generated by $\{-1$, $z_2$, $z_2+1$, $z_2-1$, $z_2^2 - 2z_2\}$.
\end{proof}

\begin{lemma}
\label{lem:II}
Let $P_0$ be matrix $II$ listed in Definition \ref{def:possible-matrices}. Then the universal partial field of $M=\widetilde{M}([I_5|D_5|P_0])$ is $\mathbb{G}$. Moreover, $M$ is only representable over a field if it contains a root of $x^2-x-1$.
 \end{lemma}

\begin{proof}
The ideal is generated by $\{z_9^2 + z_9 - 1$, $z_0 + z_9$, $z_1 - z_9$, $z_2 - z_9$, $z_3 + z_9$, $z_4 + z_9 + 1$, $z_5 - z_9 - 1$, $z_6 - z_9 - 1$, $z_7 + z_9 + 1$, $z_8 + z_9\}$. Note that the ideal contains $z_8^2 - z_8 - 1$. We solve for the variables in terms of $z_8$ and obtain $z_0=z_8$, $z_1=-z_8$, $z_2=-z_8$, $z_3=z_8$, $z_4=z_8-1$, $z_5=-z_8+1$, $z_6=-z_8+1$, $z_7=z_8-1$, and $z_9=-z_8$. Therefore, we check that the matrix is a $\mathbb{G}$-matrix by checking the partial field generated by $\{-1$, $z_8$, $z_8+1$, $z_8-1$, $z_8-2$, $z_8^2-2z_8=-1/z_8$, $-2z_8+3=-1/z_8^3$, $2z_8+1=z_8^3\}$.

The fact that $z_8^2 - z_8 - 1$ is in the ideal implies that $M$ is only representable over fields that contain a root of $x^2-x-1$.
\end{proof}

\begin{lemma}
\label{lem:III}
Let $P_0$ be matrix $III$ listed in Definition \ref{def:possible-matrices}. Then the universal partial field of $M=\widetilde{M}([I_4|D_4|P_0])$ is $\mathbb{G}$. Moreover, $M$ is only representable over a field if it contains a root of $x^2-x-1$.
 \end{lemma}

\begin{proof}
 The ideal is generated by $\{z_9^2 - z_9 - 1$, $z_0 + z_9$, $z_1 - z_9$, $z_2 + z_9 - 1$, $z_3 - z_9 + 2$, $z_4 + z_9 - 1$, $z_5 - z_9 + 1$, $z_6 + z_9$, $z_7 - z_9$, $z_8 + z_9 + 1\}$.  We solve for the variables in terms of $z_9$ and obtain $z_0=-z_9$, $z_1=z_9$, $z_2=-z_9+1$, $z_3=z_9-2$, $z_4=-z_9+1$, $z_5=z_9-1$, $z_6=-z_9$, $z_7=z_9$, and $z_8=-z_9-1$. Since $z_9^2 - z_9 - 1$ is in the ideal, we have $z_9+1=z_9^2$, $z_9-1=1/z_9$, $z_9-2=-z_9^{-2}$, $-z_9^2+2=-z_9^{-1}$, and $-z_9^2 + 2z_9 + 1 = z_9$. Therefore, we check that the matrix is a $\mathbb{G}$-matrix by checking the partial field generated by $\{-1$, $z_9$, $z_9+1$, $z_9-1$, $z_9-2$, $-z_9^2+2$, $-z_9^2 + 2z_9 + 1\}$.

The fact that $z_9^2 - z_9 - 1$ is in the ideal implies that $M$ is only representable over fields that contain a root of $x^2-x-1$.
\end{proof}

 \begin{lemma}
\label{lem:IV}
Let $P_0$ be matrix $IV$ listed in Definition \ref{def:possible-matrices}. Then the universal partial field of $M=\widetilde{M}([I_4|D_4|P_0])$ is $\mathbb{G}$. Moreover, $M$ is only representable over a field if it contains a root of $x^2-x-1$.
 \end{lemma}

\begin{proof}
The ideal is generated by $\{z_9^2 - z_9 - 1$, $z_0 + z_9$, $z_1 - z_9$, $z_2 - z_9 + 2$, $z_3 + z_9 - 1$, $z_4 - z_9 + 1$, $z_5 + z_9$, $z_6 - z_9$, $z_7 + z_9 + 1$, $z_8 + z_9 + 1\}$.  We solve for the variables in terms of $z_9$ and obtain $z_0=-z_9$, $z_1=z_9$, $z_2=z_9-2$, $z_3=-z_9+1$, $z_4=z_9-1$, $z_5=-z_9$, $z_6=z_9$, $z_7=-z_9-1$, $z_8=-z_9-1$. Since $z_9^2 - z_9 - 1$ is in the ideal, we have $z_9+1=z_9^2$, $z_9-1=1/z_9$, $z_9-2=-z_9^{-2}$, $-z_9^2+2=-z_9^{-1}$, $2z_9^2 - 2z_9 - 1 = 1$, $2z_9 + 1 = z_9^3$, and $-z_9^2 + 2z_9 + 1 = z_9$. Therefore, we check that the matrix is a $\mathbb{G}$-matrix by checking the partial field generated by $\{-1$, $z_9$, $z_9+1$, $z_9-1$, $z_9-2$, $-z_9^2+2$, $2z_9^2 - 2z_9 - 1$, $2z_9 + 1$, $-z_9^2 + 2z_9 + 1\}$.

The fact that $z_9^2 - z_9 - 1$ is in the ideal implies that $M$ is only representable over fields that contain a root of $x^2-x-1$.
\end{proof}

\begin{lemma}
\label{lem:V}
Let $P_0$ be matrix $V$ listed in Definition \ref{def:possible-matrices}. Then the universal partial field of $M=\widetilde{M}([I_5|D_5|P_0])$ is $\mathbb{K}_2$.
 \end{lemma}

\begin{proof}
 The ideal is generated by $\{z_1z_5 + z_5 + 1$, $z_0 + z_1$, $z_2 - z_5$, $z_3 + z_5 + 1$, $z_4 + z_5 + 1\}$. We solve for the variables in terms of $z_0$ and obtain $z_1=-z_0$, $z_2=1/(z_0-1)$, $z_3=-z_0/(z_0-1)$, $z_4=-z_0/(z_0-1)$, and $z_5=1/(z_0-1)$. We check that the matrix is a $\mathbb{K}_2$-matrix by checking the partial field generated by $\{-1$, $z_0$, $z_0+1$, $z_0-1$, $1/z_0$, $1/(z_0+1)$, $1/(z_0-1)\}$.
\end{proof}

\begin{lemma}
\label{lem:VI}
Let $P_0$ be matrix $VI$ listed in Definition \ref{def:possible-matrices}. Then the universal partial field of $M=\widetilde{M}([I_5|D_5|P_0])$ is $\mathbb{K}_2$.
 \end{lemma}

\begin{proof}
 The ideal is generated by $\{z_1z_5 - 1$, $z_1z_9 + z_1 + z_9$, $z_5z_9 + z_9 + 1$, $z_0 + z_1$, $z_2 + z_9 + 1$, $z_3 - z_9$, $z_4 + z_5$, $z_6 + z_9 + 1$, $z_7 - z_9$, $z_8 + z_9 + 1\}$. We solve for the variables in terms of $z_0$ and obtain $z_1=-z_0$, $z_2=1/(z_0-1)$, $z_3=-z_0/(z_0-1)$, $z_4=1/z_0$, $z_5=-1/z_0$, $z_6=1/(z_0-1)$, $z_7=-z_0/(z_0-1)$, $z_8=1/(z_0-1)$, and $z_9=-z_0/(z_0-1)$. We check that the matrix is a $\mathbb{K}_2$-matrix by checking the partial field generated by $\{-1$, $z_0$, $z_0+1$, $z_0-1$, $1/z_0$, $1/(z_0+1)$, $1/(z_0-1)\}$.
\end{proof}

\begin{lemma}
\label{lem:VII}
Let $P_0$ be matrix $VII$ listed in Definition \ref{def:possible-matrices}. Then the universal partial field of $M=\widetilde{M}([I_5|D_5|P_0])$ is $\mathbb{K}_2$.
\end{lemma}

\begin{proof}
 The ideal is generated by $\{z_3z_9 + z_3 + z_9$, $z_3z_{11} + z_3 - 1$, $z_9z_{11} + 2z_9 + 1$, $z_0 + z_3$, $z_1 - z_3$, $z_2 + z_3$, $z_4 + z_9 + 1$, $z_5 - z_9$, $z_6 + z_9 + 1$, $z_7 - z_9$, $z_8 + z_9 + 1$, $z_{10} + z_{11} + 1\}$.  We solve for the variables in terms of $z_{11}$ and obtain $z_0=-1/(z_{11}+1)$, $z_1=1/(z_{11}+1)$, $z_2=-1/(z_{11}+1)$, $z_3=1/(z_{11}+1)$, $z_4=-(z_{11}+1)/(z_{11}+2)$, $z_5=-1/(z_{11}+2)$, $z_6=-(z_{11}+1)/(z_{11}+2)$, $z_7=-1/(z_{11}+2)$, $z_8=-(z_{11}+1)/(z_{11}+2)$, $z_9=-1/(z_{11}+2)$, and $z_{10}=-z_{11}-1$. We check that the matrix is a $\mathbb{K}_2$-matrix by checking the partial field generated by $\{-1$, $z_{11}$, $z_{11}+1$, $z_{11}+2$, $1/z_{11}$, $1/(z_{11}+1)$, $1/(z_{11}+2)\}$. (Here, $z_{11}+1$ is plays the role of $\alpha$ in Example \ref{exa:2-cyclotomic}.)
\end{proof}

\begin{lemma}
\label{lem:VIII}
Let $P_0$ be matrix $VIII$ listed in Definition \ref{def:possible-matrices}. Then the universal partial field of $M=\widetilde{M}([I_4|D_4|P_0])$ is $\mathbb{K}_2$.
\end{lemma}

\begin{proof}
 The ideal is generated by $\{z_1z_5 + z_1 - z_5$, $z_1z_7 + 1$, $z_5z_7 + z_5 + 1$, $z_0 + z_1$, $z_2 - z_7$, $z_3 + z_7 + 1$, $z_4 + z_5 + 1$, $z_6 + z_7\}$.  We solve for the variables in terms of $z_3$ and obtain $z_0=-1/(z_3+1)$, $z_1=1/(z_3+1)$, $z_2=-(z_3+1)$, $z_4=-(z_3+1)/z_3$, $z_5=1/z_3$, $z_6=z_3+1$, and $z_7=-(z_3+1)$. We check that the matrix is a $\mathbb{K}_2$-matrix by checking the partial field generated by $\{-1,z_3$, $z_3+1$, $z_3+2$, $1/z_3$, $1/(z_3+1)$, $1/(z_3+2)\}$. (Here, $z_3+1$ is plays the role of $\alpha$ in Example \ref{exa:2-cyclotomic}.)
\end{proof}

\begin{lemma}
\label{lem:IX}
Let $P_0$ be matrix $IX$ listed in Definition \ref{def:possible-matrices}. Then the universal partial field of $M=\widetilde{M}([I_4|D_4|P_0])$ is $\mathbb{K}_2$.
\end{lemma}

\begin{proof}
The ideal is generated by $\{z_5z_7 - 1$, $z_0 - z_5$, $z_1 + z_5$, $z_2 - z_7$, $z_3 + z_7$, $z_4 + z_5$, $z_6 + z_7\}$.  We solve for the variables in terms of $z_7$ and obtain $z_0=1/z_7$, $z_1=-1/z_7$, $z_2=z_7$, $z_3=-z_7$, $z_4=-1/z_7$, $z_5=1/z_7$, $z_6=-z_7$. Note that $-z_7^3 + z_7^2 + z_7 - 1 = - (z_7 - 1)^2(z_7 + 1)$. Therefore, we check that the matrix is a $\mathbb{K}_2$-matrix by checking the partial field generated by $\{-1$, $z_7$, $z_7+1$, $z_7-1$, $1/z_7$, $1/(z_7+1)$, $1/(z_7-1)$, $-z_7^3 + z_7^2 + z_7 - 1\}$.
\end{proof}

 \begin{lemma}
\label{lem:X}
Let $P_0$ be matrix $X$ listed in Definition \ref{def:possible-matrices}. Then the universal partial field of $M=\widetilde{M}([I_5|D_5|P_0])$ is $\mathbb{K}_2$.
 \end{lemma}

\begin{proof}
 The ideal is generated by $\{z_7z_9 - 1$, $z_0 + z_9 + 1$, $z_1 - z_9$, $z_2 + z_9 + 1$, $z_3 - z_9$, $z_4 + z_9 + 1$, $z_5 - z_9$, $z_6 + z_7 + 1$, $z_8 + z_9 + 1\}$.  We solve for the variables in terms of $z_9$ and obtain $z_0=-z_9-1$, $z_1=z_9$, $z_2=-z_9-1$, $z_3=z_9$, $z_4=-z_9-1$, $z_5=z_9$, $z_6=-(z_9+1)/z_9$, $z_7=1/z_9$, $z_8=-z_9-1$. We check that the matrix is a $\mathbb{K}_2$-matrix by checking the partial field generated by $\{-1$, $z_9$, $z_9+1$, $z_9-1$, $1/z_9$, $1/(z_9+1)$, $1/(z_9-1)\}$.
\end{proof}

 \begin{lemma}
\label{lem:XI}
Let $P_0$ be matrix $XI$ listed in Definition \ref{def:possible-matrices}. Then the universal partial field of $M=\widetilde{M}([I_5|D_5|P_0])$ is $\mathbb{K}_2$.
 \end{lemma}

\begin{proof}
 The ideal is generated by $\{z_3z_7 - 1$, $z_0 + z_3 + 1$, $z_1 - z_3$, $z_2 + z_3 + 1$, $z_4 + z_7 + 1$, $z_5 - z_7$, $z_6 + z_7 + 1\}$.  We solve for the variables in terms of $z_7$ and obtain $z_0=-(z_7+1)/z_7$, $z_1=1/z_7$, $z_2=-(z_7+1)/z_7$, $z_3=1/z_7$, $z_4=-z_7-1$, $z_5=z_7$, and $z_6=-z_7-1$. There is an optional argument for the function \texttt{check\_partial\_field} called \texttt{extra\_determinants} that allows us to enter a list of polynomials that are known to be products of the generators of the partial field but that we do not wish to include as a generator in order to reduce the running time of the function. We know that $z_7^4-2z_7^2+1=(z_7+1)^2(z_7-1)^2$. Therefore, we check that the matrix is a $\mathbb{K}_2$-matrix by checking the partial field generated by $\{-1$, $z_7$, $z_7+1$, $z_7-1$, $1/z_7$, $1/(z_7+1)$, $1/(z_7-1)\}$, with the list of extra determinants consisting of $(z_7^4 - 2z_7^2 + 1)/z_7^2$ and $(-z_7^4 + 2z_7^2 - 1)/z_7^2$.
\end{proof}

 \begin{lemma}
\label{lem:XII}
Let $P_0$ be matrix $XII$ listed in Definition \ref{def:possible-matrices}. Then the universal partial field of $M=\widetilde{M}([I_5|D_5|P_0])$ is $\mathbb{K}_2$.
 \end{lemma}

\begin{proof}
 The ideal is generated by $\{z_7z_{11} - 1$, $z_0 + z_{11} + 1$, $z_1 - z_{11}$, $z_2 + z_{11} + 1$, $z_3 - z_{11}$, $z_4 + z_{11} + 1$, $z_5 - z_{11}$, $z_6 + z_7 + 1$, $z_8 + z_{11} + 1$, $z_9 - z_{11}$, $z_{10} + z_{11} + 1\}$.  We solve for the variables in terms of $z_{11}$ and obtain $z_0=-z_{11}-1$, $z_1=z_{11}$, $z_2=-z_{11}-1$, $z_3=z_{11},z_4=-z_{11}-1$, $z_5=z_{11}$, $z_6=-(z_{11}+1)/z_{11}$, $z_7=1/z_{11}$, $z_8=-z_{11}-1$, $z_9=z_{11}$, and $z_{10}=-z_{11}-1$. We check that the matrix is a $\mathbb{K}_2$-matrix by checking the partial field generated by $\{-1$, $z_{11}$, $z_{11}+1$, $z_{11}-1$, $1/z_{11}$, $1/(z_{11}+1)$, $1/(z_{11}-1)\}$.
\end{proof}

 \begin{lemma}
\label{lem:XIII}
Let $P_0$ be matrix $XIII$ listed in Definition \ref{def:possible-matrices}. Then the universal partial field of $M=\widetilde{M}([I_5|D_5|P_0])$ is $\mathbb{K}_2$.
 \end{lemma}

\begin{proof}
  The ideal is generated by $\{z_7z_9 - 1$, $z_0 + z_7 + 1$, $z_1 - z_7$, $z_2 + z_7 + 1$, $z_3 - z_7$, $z_4 + z_7 + 1$, $z_5 - z_7$, $z_6 + z_7 + 1$, $z_8 + z_9 + 1\}$.  We solve for the variables in terms of $z_7$ and obtain $z_0=-z_7-1$, $z_1=z_7$, $z_2=-z_7-1$, $z_3=z_7$, $z_4=-z_7-1$, $z_5=z_7$, $z_6=-z_7-1$, $z_8=-(z_7+1)/z_7$, and $z_9=1/z_7$. We check that the matrix is a $\mathbb{K}_2$-matrix by checking the partial field generated by $\{-1,z_7,z_7+1,z_7-1,1/z_7,1/(z_7+1),1/(z_7-1)\}$.
\end{proof}

 \begin{lemma}
\label{lem:XIV}
Let $P_0$ be matrix $XIV$ listed in Definition \ref{def:possible-matrices}. Then the universal partial field of $M=\widetilde{M}([I_5|D_5|P_0])$ is $\mathbb{K}_2$.
 \end{lemma}

\begin{proof}
 The ideal is generated by $\{z_{11}^2 + z_9$, $z_0 + z_{11} + 1$, $z_1 - z_{11}$, $z_2 + z_{11} + 1$, $z_3 - z_{11}$, $z_4 + z_{11} + 1$, $z_5 - z_{11}$, $z_6 + z_{11} + 1$, $z_7 - z_{11}$, $z_8 + z_9 + 1$, $z_{10} + z_{11} + 1\}$.  We solve for the variables in terms of $z_{11}$ and obtain $z_0=-z_{11}-1$, $z_1=z_{11}$, $z_2=-z_{11}-1$, $z_3=z_{11}$, $z_4=-z_{11}-1$, $z_5=z_{11}$, $z_6=-z_{11}-1$, $z_7=z_{11}$, $z_8=(z_{11}+1)(z_{11}-1)$, $z_9=-z_{11}^2$, and $z_{10}=-z_{11}-1$. We check that the matrix is a $\mathbb{K}_2$-matrix by checking the partial field generated by $\{-1$, $z_{11}$, $z_{11}+1$, $z_{11}-1\}$.
\end{proof}

 \begin{lemma}
\label{lem:XV}
Let $P_0$ be matrix $XV$ listed in Definition \ref{def:possible-matrices}. Then the universal partial field of $M=\widetilde{M}([I_4|D_4|P_0])$ is $\mathbb{U}_2$.
 \end{lemma}

\begin{proof}
  The ideal is generated by $\{z_0 + z_3 + 1$, $z_1 - z_3$, $z_2 + z_3 + 1$, $z_4 + z_7 + 1$, $z_5 - z_7$, $z_6 + z_7 + 1\}$.  We solve for the variables in terms of $z_3$ and $z_7$ and obtain $z_0=-z_3-1$, $z_1=z_3$, $z_2=-z_3-1$, $z_4=-z_7-1$, $z_5=z_7$, and $z_6=-z_7-1$. We check that the matrix is a $\mathbb{U}_2$-matrix by checking the partial field generated by $\{-1$, $z_3$, $z_7$, $z_3+1$, $z_7+1$, $z_7-z_3\}$. (Here, $z_3+1$ and $z_7+1$ play the roles of $\alpha_1$ and $\alpha_2$ in Example \ref{exa:2-regular}.)
\end{proof}

\begin{lemma}
\label{lem:XVI}
Let $P_0$ be matrix $XVI$ listed in Definition \ref{def:possible-matrices}. Then the universal partial field of $M=\widetilde{M}([I_4|D_4|P_0])$ is $\PT$.
 \end{lemma}

\begin{proof}
The ideal is generated by $\{z_3z_7 + z_3 + z_7$, $z_0 + z_7$, $z_1 - z_7$, $z_2 + z_3 + 1$, $z_4 + z_7 + 1$, $z_5 - z_7 - 1$, $z_6 + z_7 + 1\}$.  We solve for the variables in terms of $z_7$ and obtain $z_0=-z_7$, $z_1=z_7$, $z_2=-1/(z_7+1)$, $z_3=-z_7/(z_7+1)$, $z_4=-z_7-1$, $z_5=z_7+1$, and $z_6=-z_7-1$. We check that the matrix is a $\PT$ matrix by checking the partial field generated by $\{-1$, $z_7$, $z_7+1$, $z_7-1$, $z_7+2$, $2z_7+1$, $1/z_7$, $1/(z_7+1)$, $1/(z_7-1)$, $1/(z_7+2)\}$. 
\end{proof}

\begin{lemma}
\label{lem:XVII}
Let $P_0$ be matrix $XVII$ listed in Definition \ref{def:possible-matrices}. Then the universal partial field of $M=\widetilde{M}([I_4|D_4|P_0])$ is $\mathbb{U}_2$.
\end{lemma}

\begin{proof}
The ideal is generated by $\{z_1z_5 + z_3 + z_5 + 1$, $z_3z_5 - z_3z_7 - z_5z_7 - z_7$, $z_1z_7 + z_3$, $z_0 + z_1 + 1$, $z_2 + z_3 + 1$, $z_4 + z_5 + 1$, $z_6 + z_7 + 1\}$.  We solve for the variables in terms of $z_1$ and $z_3$ and obtain $z_0=-z_1-1$, $z_2=-z_3-1$, $z_4=(z_3-z_1)/(z_1+1)$, $z_5=-(z_3+1)/(z_1+1)$, $z_6=(z_3-z_1)/z_1$, and $z_7=-z_3/z_1$. We check that the matrix is a $\mathbb{U}_2$-matrix by checking the partial field generated by $\{-1$, $z_1$, $z_3$, $z_1+1$, $z_3+1$, $z_3-z_1$, $1/z_1$, $1/(z_1+1)\}$. (Here, $z_1+1$ and $z_3+1$ play the roles of $\alpha_1$ and $\alpha_2$ in Example \ref{exa:2-regular}.)
\end{proof}

\begin{lemma}
\label{lem:XVIII}
Let $P_0$ be matrix $XVIII$ listed in Definition \ref{def:possible-matrices}. Then the universal partial field of $M=\widetilde{M}([I_4|D_4|P_0])$ is $\mathbb{K}_2$.
\end{lemma}

\begin{proof}
The ideal is generated by $\{z_3z_5 + z_5 - 1$, $z_0 + z_5$, $z_1 - z_5$, $z_2 + z_3 + 1$, $z_4 + z_5 + 1\}$.  We solve for the variables in terms of $z_5$ and obtain $z_0=-z_5$, $z_1=z_5$, $z_2=-1/z_5$, $z_3=-(z_5-1)/z_5$, and $z_4=-z_5-1$. We check that the matrix is a $\mathbb{K}_2$-matrix by checking the partial field generated by $\{-1$, $z_5$, $z_5+1$, $z_5-1$, $1/z_5$, $1/(z_5+1)$, $1/(z_5-1)\}$.
\end{proof}

The proofs of the next two theorems are essentially identical to each other. We give the proof of Theorem \ref{thm:AC4connected} but omit the proof of Theorem \ref{thm:AC4clique}.

\begin{theorem}
\label{thm:AC4connected}
 Suppose Hypothesis \ref{hyp:connected-template} holds. Then there exists $k\in\mathbb{Z}_+$ such that, for every $k$-connected member $M$ of $\mathcal{AC}_4$ with at least $2k$ elements, either $M$ or $M^*$ is a minor of the vector matroid of a matrix of the form below, where $P_0$ is one of matrices $I$--$XVI$ listed in Definition \ref{def:possible-matrices}, up to a field isomorphism.
\end{theorem}

\begin{center}
\begin{tabular}{|c|c|c|}
\hline
\multirow{2}{*}{$I_r$}&\multirow{2}{*}{$D_r$}&$P_0$\\
\cline{3-3}
&&$0$\\
\hline
\end{tabular}
\end{center}

\begin{theorem}
\label{thm:AC4clique} 
Suppose Hypothesis \ref{hyp:clique-template} holds. Then there exist $k,n\in\mathbb{Z}_+$ such that every simple vertically $k$-connected member of $\mathcal{AC}_4$ with an $M(K_n)$-minor is a minor of the vector matroid of a matrix of the form above, and every cosimple cyclically $k$-connected member of $\mathcal{AC}_4$ with an $M^*(K_n)$-minor is a minor of the dual of the vector matroid of a matrix of the form above, where $P_0$ is one of matrices $I$--$XVI$ listed in Definition \ref{def:possible-matrices}, up to a field isomorphism.
\end{theorem}

\begin{proof}[Proof of Theorem \ref{thm:AC4connected}]
Recall the definition of weak conforming from Definition \ref{def:weaklyconforming}. Also recall that $\mathcal{M}_w(\Phi)$ is the set of matroids weakly conforming to a template $\Phi$. By Corollary \ref{cor:weak-connected-template} (with $m=3$, since the characteristic set of $\PG(2,2)=F_7$ is $\{2\}$), there is a set of refined templates $\Phi_1,\ldots,\Phi_s,\Psi_1,\ldots,\Psi_t$ such that
\begin{enumerate}
\item $\mathcal{AC}_4$ contains each of the classes $\mathcal{M}_w(\Phi_1),\dots,\mathcal{M}_w(\Phi_s)$,
\item $\mathcal{AC}_4$ contains the duals of the represented matroids in each of the classes $\mathcal{M}_w(\Psi_1)$,$\dots$,$\mathcal{M}_w(\Psi_t)$, and
\item if $M$ is a simple $k$-connected member of $\mathcal{AC}_4$ with at least $2k$ elements, then either $M$ is a member of at least one of the classes $\mathcal{M}_w(\Phi_1),\ldots,\mathcal{M}_w(\Phi_s)$, or $M^*$ is a member of at least one of the classes $\mathcal{M}_w(\Psi_1),\ldots,\mathcal{M}_w(\Psi_t)$.
\end{enumerate}

Lemma \ref{lem:possible-matrices} (and the fact that $\mathcal{AC}_4$ is closed under duality) implies that these templates can be chosen so that each of them is the complete, lifted $Y$-template determined by a submatrix of one of matrices $I$--$XVI$. Now, consider the complete, lifted $Y$-templates determined by a matrix $P_0$, where $P_0$ is one of matrices $I$--$XVI$ themselves (rather than a submatrix). Let $m$ be the number of rows of $P_0$. By Lemmas \ref{lem:I}--\ref{lem:XVI}, $M=\widetilde{M}([I_m|D_m|P_0])$ is representable over $\mathbb{G}$, $\mathbb{K}_2$, $\mathbb{U}_2$, or $\PT$. Therefore, by Corollary \ref{cor:PtoF}, $M\in\mathcal{AC}_4$. Theorem \ref{thm:PtoFtemplate} then implies that $\mathcal{M}(\Phi_{P_0})\subseteq\mathcal{AC}_4$.

Therefore, we may take $\{\Phi_1,\ldots,\Phi_s\}$ and $\{\Psi_1,\ldots,\Psi_t\}$ both to consist of the complete, lifted $Y$-templates determined by matrices $I$--$XVI$. Again, let $P_0$ be one of these matrices. The rank-$r$ universal matroids conforming to $\Phi_{P_0}$ (of which every matroid conforming to $\Phi_{P_0}$ is a restriction) are represented by matrices of the form given in the statement of the theorem.
\end{proof}

We can now prove Theorem \ref{thm:AC4exc-minor}.

\begin{proof}[Proof of Theorem \ref{thm:AC4exc-minor}]
Let $\mathcal{E}_1=\{F_7,F_7^*,V_1,V_2,V_3,P_1,P_2,P_3\}$, and let $\mathcal{E}_2=\{F_7,F_7^*,V_1^*,V_2,V_3^*,P_1^*,P_2,P_3^*\}$. By Lemmas \ref{lem:connected-excluded-minors} and \ref{lem:clique-excluded-minors}, we must show that the hypotheses (i)--(iv) given in Section \ref{sec:Excluded Minors} are satisfied. By Theorem \ref{thm:AC4-excluded}, (i) holds. Since $F_7$ is binary, (iv) holds. The proofs of Lemmas \ref{lem:AC4-Y-template} and \ref{lem:possible-matrices} show that (ii) holds. By Theorem \ref{thm:AC4-excluded}, $V_2$ and $P_2$ are self-dual; therefore, (iii) holds by (ii) and duality.
\end{proof}

\section{The Highly Connected Golden-Mean Matroids}
\label{sec:The Highly Connected Golden-Mean Matroids}
In this section, we characterize the highly connected golden-mean matroids, subject to Hypotheses \ref{hyp:connected-template} and \ref{hyp:clique-template}. To do this, we need some information about the \emph{Pappus matroid}, which has the geometric representation given in Figure \ref{fig:Pappus}. \begin{figure}[ht]
\[\begin{tikzpicture}[x=2cm, y=2cm]
	\vertex[fill,inner sep=1pt,minimum size=1pt] (1) at (0,2) [label=above:$1$] {};
 	\vertex[fill,inner sep=1pt,minimum size=1pt] (2) at (4,0) [label=below:$2$] {};
 	\vertex[fill,inner sep=1pt,minimum size=1pt] (3) at (2,2) [label=above:$3$] {};
	\vertex[fill,inner sep=1pt,minimum size=1pt] (4) at (4,2) [label=above:$4$] {};
 	\vertex[fill,inner sep=1pt,minimum size=1pt] (5) at (3,1) [label=right:$5$] {};
 	\vertex[fill,inner sep=1pt,minimum size=1pt] (6) at (1,1) [label=left:$6$] {};
 	\vertex[fill,inner sep=1pt,minimum size=1pt] (7) at (2,1) [label=above:$7$] {};
 	\vertex[fill,inner sep=1pt,minimum size=1pt] (8) at (0,0) [label=below:$8$] {};
 	\vertex[fill,inner sep=1pt,minimum size=1pt] (9) at (2,0) [label=below:$9$] {};
 	\path
 		(1) edge (9)
 		(1) edge (4)
 		(1) edge (2)
 		(3) edge (8)
 		(3) edge (2)
 		(4) edge (8)
 		(4) edge (9)
 		(6) edge (5)
 		(8) edge (2);
\end{tikzpicture}\]
\caption{A Geometric Representation of the Pappus Matroid}
  \label{fig:Pappus}
\end{figure}
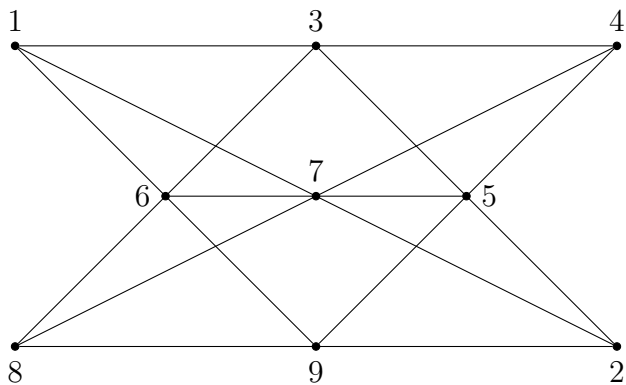

\begin{lemma}
 \label{lem:pappus-properties}
\leavevmode
 \begin{enumerate}
\item[(i)] The Pappus matroid is a minor of a matroid conforming to $\Phi_{P_0}$, where $P_0$ is matrix $XVI$.
\item[(ii)] The Pappus matroid is representable over a field $\F$ if and only if $\F=\GF(4)$ or $|\F|\geq7$.
\item[(iii)] If $P_0$ is any proper column submatrix of matrix $XVI$, then $P_0$ is also a submatrix of matrix $III$ or $IV$ (up to field isomorphism and permuting of rows and columns).
\end{enumerate}
\end{lemma}

\begin{proof}
To prove (i), let $P_0$ be matrix $XVI$. Note that the vector matroid of $[I_4|D_4|P_0]$ virtually conforms to $\Phi_{P_0}$ and is therefore a minor of a matroid conforming to $\Phi_{P_0}$. Contract the element represented by the first column and simplify. The result is \[\begin{bmatrix}
1&0&0&1&1&0&\alpha&\alpha^2&\alpha^2&\alpha^2\\
0&1&0&1&0&1&\alpha&1            &\alpha^2&\alpha   \\
0&0&1&0&1&1&1        &0            &1            &1       \\
\end{bmatrix}.\] Delete the element represented by the fourth column, and the result is
\[A=\begin{blockarray}{ccccccccc}
1 & 2 & 3 & 4 & 5 & 6 & 7 & 8 & 9 \\
\begin{block}{[ccccccccc]}
1&0&0&1&0&\alpha&\alpha^2&\alpha^2&\alpha^2\\
0&1&0&0&1&\alpha&1            &\alpha^2&\alpha   \\
0&0&1&1&1&1        &0            &1            &1       \\
\end{block}
\end{blockarray},\]
which represents the Pappus matroid with the column indices corresponding to the labels in Figure \ref{fig:Pappus}.

To prove (ii), we use SageMath to show that the universal partial field of the Pappus matroid is the Pappus partial field $\Ppap$. The computations are similar to those used in Section \ref{sec:The Highly Connected Matroids in AC4}, except that we must be more careful since the Pappus matroid does not have a spanning clique. Combining Lemma \ref{lem:pappus-template} with Theorem \ref{thm:UPF-homomorphism} proves (ii).

To prove (iii), label the columns of matrix $XVI$ from left to right as $a$, $b$, $c$, and $d$ and the rows as $1$, $2$, $3$, $4$. The matrix $[a,b,c]$, after a field isomorphism, is contained in matrix $III$. The matrix $[a,b,d]$ is contained in matrix $IV$. If we put the rows of $[a,c,d]$ in order $4,3,2,1$ and perform a field isomorphism, the resulting matrix is contained in matrix $III$. If we put the rows of $[b,c,d]$ in order $1,3,2,4$ and perform a field isomorphism, the resulting matrix is contained in matrix $IV$.
\end{proof}

The proofs of the next two theorems are essentially identical to each other. We give the proof of Theorem \ref{thm:GMconnected} but omit the proof of Theorem \ref{thm:GMclique}.

\begin{theorem}
\label{thm:GMconnected}
Suppose Hypothesis \ref{hyp:connected-template} holds. Then there exists $k\in\mathbb{Z}_+$ such that, for every $k$-connected golden-mean matroid $M$ with at least $2k$ elements, either $M$ or $M^*$ is a minor of the vector matroid of a matrix of the form below, where $P_0$ is one of matrices $I$--$XV$ listed in Definition \ref{def:possible-matrices}, up to a field isomorphism.
\end{theorem}

\begin{center}
\begin{tabular}{|c|c|c|}
\hline
\multirow{2}{*}{$I_r$}&\multirow{2}{*}{$D_r$}&$P_0$\\
\cline{3-3}
&&$0$\\
\hline
\end{tabular}
\end{center}

\begin{theorem}
\label{thm:GMclique} 
Suppose Hypothesis \ref{hyp:clique-template} holds. Then there exist $k,n\in\mathbb{Z}_+$ such that every simple, vertically $k$-connected golden-mean matroid with an $M(K_n)$-minor is a minor of the vector matroid of a matrix of the form above, and every cosimple cyclically $k$-connected golden-mean matroid with an $M^*(K_n)$-minor is a minor of the dual of the vector matroid of a matrix of the form above, where $P_0$ is one of matrices $I$--$XV$ listed in Definition \ref{def:possible-matrices}, up to a field isomorphism.
\end{theorem}

\begin{proof}[Proof of Theorem \ref{thm:GMconnected}]
Similarly to the proof of Theorem \ref{thm:AC4connected}, we find the templates $\Phi_1,\ldots,\Phi_s,\Psi_1,\ldots,\Psi_t$ for $\GM$ whose existence are implied by Corollary \ref{cor:weak-connected-template}. Since $\GM\subseteq\AC$, it follows from combining Lemma \ref{lem:possible-matrices} and Theorem \ref{thm:AC4connected} that we may take each of $\Phi_1,\ldots,\Phi_s,\Psi_1,\ldots,\Psi_t$ to be the complete, lifted $Y$-template determined by some column submatrix $P_0$ of one of matrices $I$--$XVI$.

Consider the complete, lifted $Y$-templates determined by a matrix $P_0$, where $P_0$ is one of matrices $I$--$XV$. Let $m$ be the number of rows of $P_0$. By Lemmas \ref{lem:I}--\ref{lem:XV}, $M=\widetilde{M}([I_m|D_m|P_0])$ is representable over $\mathbb{G}$, $\mathbb{K}_2$, or $\mathbb{U}_2$. Therefore, by Corollary \ref{cor:PtoF}, $M\in\GM$. Theorem \ref{thm:PtoFtemplate} then implies that $\mathcal{M}(\Phi_{P_0})\subseteq\GM$.

By Lemma \ref{lem:pappus-properties}(i--ii), $P_0$ cannot be matrix $XVI$ (because every golden-mean matroid is $\GF(5)$-representable). Lemma \ref{lem:pappus-properties} also states that if $P_0$ is a proper column submatrix of matrix $XVI$, then it must also be a submatrix of matrix $III$ or matrix $IV$; therefore, in that case, $P_0$ has already been analyzed above.

Therefore, we may take $\{\Phi_1,\ldots,\Phi_s\}$ and $\{\Psi_1,\ldots,\Psi_t\}$ both to consist of the complete, lifted $Y$-templates determined by matrices $I$--$XV$.
\end{proof}

We are also now ready to prove Theorem \ref{thm:GMexc-minor}.

\begin{proof}[Proof of Theorem \ref{thm:GMexc-minor}]
Let $\mathcal{E}_1=\{F_7,F_7^*,V_1,V_2,V_3,P_1,P_2,P_3$, $Pappus\}$, and let $\mathcal{E}_2=\{F_7,F_7^*,V_1^*,V_2,V_3^*,P_1^*,P_2,P_3^*$, $(Pappus)^*\}$. By Lemmas \ref{lem:connected-excluded-minors} and \ref{lem:clique-excluded-minors}, we must show that the hypotheses (i)--(iv) given in Section \ref{sec:Excluded Minors} are satisfied. By Theorem \ref{thm:AC4-excluded} and Lemma \ref{lem:pappus-properties}(ii) (recalling that every golden-mean matroid is representable over $\GF(5)$), (i) holds. Since $F_7$ is binary, (iv) holds. 

From the proof of Theorem \ref{thm:AC4exc-minor} (found in Section \ref{sec:The Highly Connected Matroids in AC4}), we know that to every template $\Phi$ such that $\mathcal{M}(\Phi)\nsubseteq\AC$ conforms a matroid that contains a member of $\mathcal{E}_1$ as a minor. Since $\GM\subseteq\AC$, we see from Theorems \ref{thm:AC4connected} and \ref{thm:GMconnected} that we only need to check the complete, lifted $Y$-template determined by matrix $XVI$.  By Lemma \ref{lem:pappus-properties}(i), the Pappus matroid is a minor of a matroid that conforms to $\Phi_{P_0}$. Therefore, (ii) holds.

Since $V_2$ and $P_2$ are self-dual, (iii) holds by (ii) and duality.
\end{proof}

\section{The Highly Connected Matroids in $\AF$ and $\SL$}
\label{sec:The Highly Connected Matroids in AF4 and SL4}

If $q$ is a prime power, let $\mathcal{AF}_q$ be the class of matroids representable over all fields of size at least $q$, and let $\mathcal{SL}_q$ denote the class of $\GF(q)$-representable matroids $M$ for which there exists a prime power $q'$ such that $M$ is representable over all fields of size at least $q'$. The abbreviations $\mathcal{AF}$ and $\mathcal{SL}$ stand for ``all fields'' and ``sufficiently large,'' respectively. Clearly, $\mathcal{AF}_q\subseteq\mathcal{SL}_q\subseteq\mathcal{AC}_q$. In this section, first we characterize the highly connected members of $\mathcal{AF}_4$ and $\mathcal{SL}_4$, subject to Hypotheses \ref{hyp:connected-template} and \ref{hyp:clique-template}. Then we determine the extremal functions and extremal matroids for these classes. We will need the next lemma to differentiate between $\mathcal{AC}_4$ and $\mathcal{SL}_4$.

\begin{lemma}
\label{lem:not-all-golden-mean}
There are infinitely many fields that do not contain a root of $x^2-x-1$.
\end{lemma}

\begin{proof}
Let $p$ be a prime other than $2$ or $5$. Solving $x^2-x-1=0$ in $\GF(p)$, we obtain $x=(1+\alpha)2^{-1}$, where $\alpha^2=5$ in $\GF(p)$. Thus, $x^2-x-1$ has a root in $\GF(p)$ if and only if there is a solution to $x^2\equiv5\pmod{p}$. A well-known result in number theory, known as quadratic reciprocity and first proved by Gauss \cite{gauss}, implies that $x^2\equiv5\pmod{p}$ has a solution if and only if $x^2\equiv p\pmod{5}$ has a solution. This is the case precisely when $p\equiv\pm1\pmod{5}$. Therefore, to prove the result, it suffices to show that there are infinitely many primes $p$ such that $p\equiv2\pmod{5}$ or $p\equiv3\pmod{5}$. This follows from another well-known number theoretic result known as Dirichlet's theorem on arithmetic progressions \cite{dirichlet}, which implies that if $a$ and $b$ are coprime integers, then there are infinitely many primes $p$ such that $p\equiv a\pmod{b}$.
\end{proof}

Recall that $I,II,\ldots,XVIII$ are the matrices given in Definition \ref{def:possible-matrices}. In order to more easily keep track of the structure of these matrices, it was helpful to use Lemma \ref{lem:sum-to-zero}(i) to ensure that the sum of the rows of each of these matrices was the zero vector. However, in this section, it will often be helpful to use Lemma \ref{lem:sum-to-zero}(ii) to give us smaller matrices to work with. 

\begin{definition}
\label{def:less-rows}
Let $I',II',\ldots,XVIII'$ denote the matrices obtained from matrices $I,II,\ldots,XVIII$, respectively, by removing the first row.
\end{definition}

The Betsy Ross matroid $B_{11}$ has the geometric representation given in Figure \ref{fig:Betsy-Ross}.
\begin{figure}[ht]
	\begin{center}
		\[\begin{tikzpicture}[x=1.4cm, y=1.4cm]
		\vertex[fill,inner sep=1pt,minimum size=1pt] (a) at (0,0) 		[label=right:$$] {};
		\vertex[fill,inner sep=0pt,minimum size=0pt] (a) at (-0.05,.35) 		[label=right:$5$] {};
		\vertex[fill,inner sep=1pt,minimum size=1pt] (b) at (0,-1.03) 		[label=below:$10$] {};
		\vertex[fill,inner sep=1pt,minimum size=1pt] (c) at (0.95,-0.31) 	[label=-45:$8$] {};
		\vertex[fill,inner sep=1pt,minimum size=1pt] (d) at (0.59,0.81) 	[label=45:$4$] {};
		\vertex[fill,inner sep=1pt,minimum size=1pt] (e) at (-0.59,0.81) 	[label=135:$3$] {};
		\vertex[fill,inner sep=1pt,minimum size=1pt] (f) at (-0.95,-0.31)	[label=225:$11$] {};
		\vertex[fill,inner sep=1pt,minimum size=1pt] (b') at (0,2.618) 	[label=right:$2$] {};
		\vertex[fill,inner sep=1pt,minimum size=1pt] (c') at (-2.49,0.809) 	[label=above:$7$] {};
		\vertex[fill,inner sep=1pt,minimum size=1pt] (d') at (-1.539,-2.188) 	[label=left:$6$] {};
		\vertex[fill,inner sep=1pt,minimum size=1pt] (e') at (1.539,-2.199) 	[label=right:$1$] {};
		\vertex[fill,inner sep=1pt,minimum size=1pt] (f') at (2.490,0.809) 	[label=above:$9$] {};
		\path
		(b) edge (b')
		(c) edge (c')
		(d) edge (d')
		(e) edge (e')
		(f) edge (f')
		(d') edge (b')
		(b') edge (e')
		(e') edge (c')
		(f') edge (d')
		(c') edge (f');
		\end{tikzpicture}\]
		\caption{A Geometric Representation of $B_{11}$}
		\label{fig:Betsy-Ross}
		\end{center}
\end{figure}
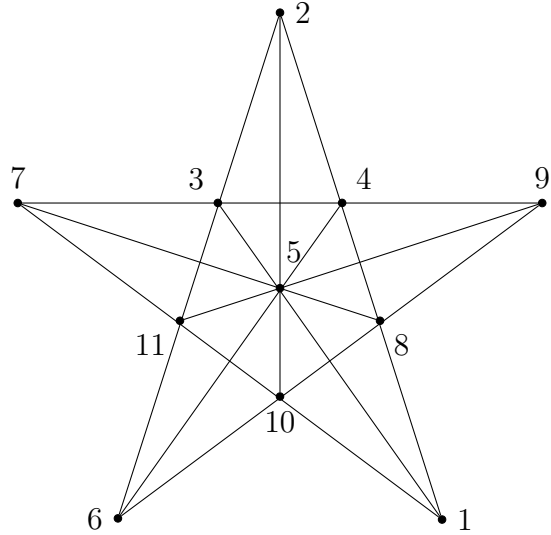 If $P_0$ is either matrix $III'$ or matrix $IV'$, then $M([I_3|D_3|P_0])\cong B_{11}$. The matrix representations are given below, with matrix $III'$ first and with both matrices having columns labeled to match the labels in Figure \ref{fig:Betsy-Ross}.
\[\begin{blockarray}{ccccccccccc}
1 & 2 & 3 & 4 & 5 & 6 & 7 & 8 & 9 & 10 & 11\\
\begin{block}{[ccccccccccc]}
1 & 0 & 0 & 1 & 1 & 0 & \alpha & \alpha & \alpha^2 & 1      & 0\\
0 & 1 & 0 & 1 & 0 & 1 & \alpha & 1      & \alpha^2 & \alpha & \alpha\\
0 & 0 & 1 & 0 & 1 & 1 & 1      & 0      & 1        & 1      & 1\\
\end{block}
\end{blockarray}\]
\[\begin{blockarray}{ccccccccccc}
1 & 8 & 10 & 2 & 11 & 6 & 5 & 4 & 7 & 9 & 3\\
\begin{block}{[ccccccccccc]}
1 & 0 & 0 & 1 & 1 & 0 & \alpha & \alpha^2 & \alpha & 0        & \alpha^2\\
0 & 1 & 0 & 1 & 0 & 1 & \alpha & 1        & 0      & \alpha^2 & \alpha  \\
0 & 0 & 1 & 0 & 1 & 1 & 1      & 0        & 1      & 1        & 1       \\
\end{block}
\end{blockarray}\]

Note that $B_{11}$ has a unique element contained on five lines with three points. We will call this element the \emph{hub} of $B_{11}$. Let $S_{10}$ denote the matroid obtained from $B_{11}$ by deleting the hub. Because $B_{11}$ is highly symmetric, there is only one other matroid (up to isomorphism) that can be obtained from $B_{11}$ by deleting a single element. (This symmetry is due in part to a matroid automorphism $\varphi$ such that $\varphi(1)=4$, $\varphi(2)=3$, $\varphi(3)=6$, $\varphi(4)=9$, $\varphi(5)=5$, $\varphi(6)=11$, $\varphi(7)=8$, $\varphi(8)=7$, $\varphi(9)=10$, $\varphi(10)=1$, and $\varphi(11)=2$.)
Call this matroid $B_{11}\backslash p$. We define $Y_9$ to be the matroid obtained from $B_{11}$ by deleting two points that form a line with the hub.

\begin{figure}[!htbp]
\[\begin{tikzpicture}[x=1cm, y=1cm, rotate=-36]
\vertex[fill,inner sep=1pt,minimum size=1pt] (a) at (0,0) 		[label=above:$$] {};
\vertex[fill,inner sep=1pt,minimum size=1pt] (b) at (0,-1) 		[label=left:$$] {};
\vertex[fill,inner sep=1pt,minimum size=1pt] (c) at (0.95,-0.31) 	[label=below:$$] {};
\vertex[fill,inner sep=1pt,minimum size=1pt] (d) at (0.59,0.81) 	[label=below:$$] {};
\vertex[fill,inner sep=1pt,minimum size=1pt] (f) at (-0.95,-0.31)	[label=above:$$] {};
\vertex[fill,inner sep=1pt,minimum size=1pt] (b') at (0,2.618) 	[label=below:$$] {};
\vertex[fill,inner sep=1pt,minimum size=1pt] (c') at (-2.49,0.809) 	[label=right:$$] {};
\vertex[fill,inner sep=1pt,minimum size=1pt] (d') at (-1.539,-2.188) 	[label=above:$$] {};
\vertex[fill,inner sep=1pt,minimum size=1pt] (f') at (2.490,0.809) 	[label=above:$$] {};
\path
(b) edge (b')
(c) edge (c')
(d) edge (d')
(f) edge (f')
(d') edge (b')
(b') edge (c)
(b) edge (c')
(f') edge (d')
(c') edge (f');
\end{tikzpicture}\]
\caption{A Geometric Representation of $Y_9$}
\label{fig:Y_9}
\end{figure}
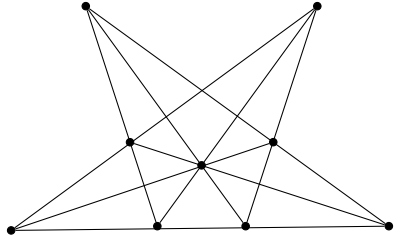

Recall from Lemma \ref{lem:pf-homs} that there are partial-field homomorphisms $\mathbb{U}_2\rightarrow\mathbb{K}_2\rightarrow\mathbb{P}_4$. Therefore, by Theorem \ref{thm:UPF-homomorphism}, every matroid that is either $\mathbb{U}_2$-representable or $\mathbb{K}_2$-representable is also $\P_4$-representable. In the proofs in the remainder of this section, whenever we permute the rows of a matrix, we automatically scale the columns so that the last nonzero entry in each column is $1$.

\begin{lemma}
\label{lem:I-need-gm}
Label the columns of matrix $I$, from left to right, as $[a,b,c,d,e,f]$, and let $P_0$ be a column submatrix of matrix $I$. If $P_0$ contains $[a,d]$, $[b,c]$, $[a,c,f]$, $[b,d,e]$, $[a,c,e]$, or $[b,d,f]$, then $M=\widetilde{M}([I_5|D_5|P_0])$ contains $Y_9$ as a minor. If $P_0$ is any column submatrix of matrix $II$ with at least two columns, then $M=\widetilde{M}([I_5|D_5|P_0])$ contains $Y_9$ as a minor. In either case, if $M$ does not contain $Y_9$ as a minor, then $M$ is a minor of a matroid virtually conforming to the complete lifted $Y$-template determined by either matrix $VII$ or matrix $IX$.
\end{lemma}

\begin{proof}
First, we will consider the case where $P_0$ is a submatrix of matrix $I$. Note that, if $B_{11}$ is labeled as in Figure \ref{fig:Betsy-Ross}, then $Y_9\cong B_{11}\backslash \{9,11\}$. Therefore, $Y_9\cong\widetilde{M}([I|D|T])$, where $T$ is the matrix obtained from matrix $III'$ by removing the third and fifth columns.

Consider $\widetilde{M}([I_5|D_5|P_0])$ where $P_0=[a,d]$; by symmetry, this is the same as $P_0=[b,c]$. If we contract the elements indexing the first and last columns, we find $Y_9$ as a restriction. Now consider $M=\widetilde{M}([I_5|D_5|P_0])$ where $P_0=[a,c,e]$. Swap the second and third rows, remove the first and fifth rows. (Recall that we scale so that the last nonzero entry of each column is $1$.) The result is the matrix $T$, up to ordering of columns. Thus, $Y_9$ is a minor of $M$. By symmetry, $[a,c,e]$ is the same as $[b,d,e]$. Moreover, if we reverse the order of the first three rows, then $[a,c,e]$ and $[b,d,e]$ become $[c,a,f]$ and $[d,b,f]$, respectively.

The maximal subsets of $\{a,b,c,d,e,f\}$ that contain none of the members of $\mathcal{F}=\{\{a,d\},\{b,c\},\{a,c,f\},\{b,d,e\},\{a,c,e\},\{b,d,f\}\}$ are the members of $\mathcal{S}=\{\{b,d\},\{c,d,e,f\}, \{a,c\}, \{a,b,e,f\}\}$. Lemma \ref{lem:sum-to-zero}(ii) implies that the complete, lifted $Y$-templates determined by $[b,d]$ and $[a,c]$ are, respectively, minor equivalent to the complete, lifted $Y$-templates determined by the matrices obtained from $[b,d]$ and $[a,c]$ by removing their zero rows. These matrices are both submatrices of matrix $IX$. The matrix $[a,b,e,f]$ is a submatrix of matrix $VII$, as is $[c,d,e,f]$, up to permuting of rows and column scaling. Lemmas \ref{lem:VII} and \ref{lem:IX}, with the fact that there is a partial-field homomorphism $\mathbb{K}_2\rightarrow\mathbb{P}_4$, imply that if $M=\widetilde{M}([I_5|D_5|P_0])$, where $P_0$ is the column submatrix of matrix $I$ corresponding to any set in $\mathcal{S}$, then $M$ is $\mathbb{P}_4$-representable.

Now, we consider the case where $P_0$ is a submatrix of matrix $II$. It is not difficult to see that every pair of columns of matrix $II$ is (up to column scaling and permuting rows) the submatrix $[a,d]$ of matrix $I$. Thus, we have already seen that $M=\widetilde{M}([I_5|D_5|P_0])$ contains $Y_9$ as a minor. If $M$ does not contain $Y_9$ as a minor, then $P_0$ consists of a single column and is a submatrix of $VII$.
\end{proof}

\begin{lemma}
\label{lem:III-IV-submatrices}
\leavevmode
\begin{itemize}
\item[(i)]Label the columns of matrix $III$, from left to right, as $[a,b,c,d,e]$. Other than $[a,b,d]$ and $[c,d,e]$, every column submatrix of matrix $III$ with at most three columns is a submatrix of matrix $VI$, $IX$, or $XVI$ (up to field isomorphism and permuting of rows and columns).
\item[(ii)]Label the columns of matrix $IV$, from left to right, as $[a,b,c,d,e]$. The column submatrix $[b,c,d,e]$ and every column submatrix containing at least three columns including column $a$ is a submatrix of matrix $VII$, $XVI$, $XVII$, or $XVIII$ (up to field isomorphism and permuting of rows and columns).
\end{itemize}
\end{lemma}

\begin{proof}
First we consider matrix $III$. Label the columns of matrix $III$, from left to right, as $a,b,c,d,e$ and the rows, from top to bottom, as $1,2,3,4$. If we reorder the rows as $3,4,1,2$, then the columns $a,b,c,d,e$ become $c,e,a,d,b$. Therefore, we need not consider sets that contain column $c$ but not column $a$. The matrix $[a,b,c]$ is a submatrix of matrix $XVI$ (after a field isomorphism). Matrix $XVIII$ is $[a,b,e]$. After swapping rows $2$ and $3$, $[a,c,d]$ is a submatrix of matrix $IX$. After putting the rows of $[a,c,e]$ in order $2,1,4,3$ and performing a field isomorphism, the resulting matrix is a submatrix of matrix $XVI$. For $[a,d,e]$, put the row in order $4,3,2,1$. The resulting matrix is a submatrix of matrix $III$. Put the rows of $[b,d,e]$ in order $2,1,3,4$. The result is a submatrix of matrix $XVI$.

Now we consider matrix $IV$. Label the columns of matrix $IV$, from left to right, as $a,b,c,d,e$ and the rows,from top to bottom, as $1,2,3,4$. The matrix $[b,c,d,e]$ is matrix $XVII$, up to reordering the columns. For $[a,b,c]$, swap rows $2$ and $3$. The result is matrix $XVIII$, up to reordering columns. For $[a,b,d]$, swap rows $2$ and $3$. The result is a submatrix of matrix $VII$. The matrix $[a,b,e]$ is a submatrix of matrix $XVI$. By putting the rows of $[a,c,d]$ in order $2,1,4,3$, we obtain a submatrix of matrix $XVI$. By putting the rows of $[a,c,e]$ in order $2,1,4,3$ and performing a field isomorphism, we obtain the matrix $XVIII$. By putting the rows of $[a,d,e]$ in order $4,2,3,1$, we obtain matrix $XVIII$.
\end{proof}

\begin{lemma}
\label{lem:S10}
The matroid $S_{10}$ is representable over $\mathbb{U}_2$ and $\P_4$.
\end{lemma}

\begin{proof}
Note that $S_{10}$ is isomorphic to $\widetilde{M}([I_3|D_3|P_0])$, where $P_0$ is obtained from matrix $IV'$ by removing the first column. Thus $P_0$ is a submatrix of matrix $XVII'$ (up to ordering of columns). By Lemma \ref{lem:XVII}, if $P_0$ is any column submatrix of matrix $IV'$ that does not contain the first column, then $\widetilde{M}([I_3|D_3|P_0])$ is representable over $\mathbb{U}_2$ and therefore $\mathbb{P}_4$.
\end{proof}

\begin{lemma}
\label{lem:B11-2}
Let $M$ be a matroid obtained from $B_{11}$ by deleting two elements. Either $M\cong Y_9$, or $M$ is $\mathbb{P}_4$-representable.
\end{lemma}

\begin{proof}
Note that $M$ is a proper restriction of $B_{11}\backslash p$. Without loss of generality, suppose $B_{11}\backslash p$ is obtained from $B_{11}$ by deleting the element $2$. Because $B_{11}$ is highly symmetric, every single-element deletion of $B_{11}\backslash p$ is isomorphic to either $B_{11}\backslash \{2,5\}$, $B_{11}\backslash \{2,3\}$, $B_{11}\backslash \{2,6\}$, $B_{11}\backslash \{2,7\}$, or $B_{11}\backslash \{2,10\}$. In fact, because of the automorphism $\phi$ discussed before Lemma \ref{lem:I-need-gm}, we have $B_{11}\backslash \{2,3\}\cong B_{11}\backslash \{2,11\}$.

The matroid $B_{11}\backslash \{2,5\}$ is a restriction of $S_{10}$, which is $\mathbb{P}_4$-representable by Lemma \ref{lem:S10}.

Since $B_{11}\backslash \{2,3\}\cong B_{11}\backslash \{8,9\}$, we use SageMath to compute the universal partial field of $\widetilde{M}([I_3|D_3|P_0])$, where $P_0$ is obtained from matrix $III'$ by removing the second and third columns. The ideal is generated by $\{z_0 + z_3$, $z_1 - z_3$, $z_2 - z_3\}$. We solve for the variables in terms of $z_3$ and obtain $z_0=-z_3$, $z_1=z_3$, and $z_2=z_3$. We check that the matrix is a $\mathbb{K}_2$-matrix by checking the partial field generated by $\{-1$, $z_3$, $z_3+1$, $z_3-1\}$.

Since $B_{11}\backslash \{2,6\}\cong B_{11}\backslash \{7,9\}$, we compute the universal partial field of $\widetilde{M}([I_3|D_3|P_0])$, where $P_0$ is obtained from matrix $III'$ by removing the first and third columns. (This matroid was denoted $M^{Y\Delta}_{8591}$ in \cite{vz09}). The ideal is generated by $\{z_0z_2 + z_0 + 1$, $z_1 - z_2\}$. We solve for the variables in terms of $z_2$ and obtain $z_0=-1/(z_2+1)$ and $z_1=z_2$. We check that the matrix is a $\mathbb{P}_4$-matrix by checking the partial field generated by $\{-1$, $z_2$, $z_2+1$, $z_2-1$, $2z_2+1\}$. If we let $x=1/z_2$, then this partial field is generated by $\{-1,1/x,(x+1)/x,-(x+1)/x,(x+2)/x\}$. This partial field is $\P_4$.

Since $B_{11}\backslash \{2,7\}\cong B_{11}\backslash \{8,11\}$, we compute the universal partial field of $\widetilde{M}([I_3|D_3|P_0])$, where $P_0$ is obtained from matrix $III'$ by removing the second and fifth columns. The ideal is generated by $\{z_3z_4 - 1$, $z_0 + z_4$, $z_1 - z_4$, $z_2 + z_3\}$. We solve for the variables in terms of $z_4$ and obtain $z_0=-z_4$, $z_1=z_4$, $z_2=-1/z_4$, and $z_3=1/z_4$. We check that the matrix is a $\mathbb{K}_2$-matrix by checking the partial field generated by $\{1$, $z_4$, $z_4+1$, $z_4-1\}$.

Finally, $B_{11}\backslash \{2,10\}\cong B_{11}\backslash \{9,11\}$, which is $Y_9$.
\end{proof}

\begin{lemma}
\label{lem:Y9}
The matroid $Y_9$ is an excluded minor for $\AF$, for $\SL$, and for the class of $\P_4$-representable matroids.
\end{lemma}

\begin{proof}
From the proof of Lemma \ref{lem:B11-2}, we see that every single-element deletion $M$ of $Y_9$ is a restriction of a matroid that is $\mathbb{P}_4$-representable. Since there are partial-field homomorphisms from $\mathbb{P}_4$ to every field of size at least $4$, we have $M\in\AF\subseteq\SL$.

Since $Y_9$ is quaternary, every single-element contraction $M$ of $Y_9$ has a simplification isomorphic to $U_{2,n}$ for $n\leq 5$. Thus, $M\in\AF\subseteq\SL$ and is $\P_4$-representable since the $\P_4$-matrix $\begin{bmatrix}1&0&1&1&1\\0&1&-1&\alpha&\alpha-1\\\end{bmatrix}$ represents $U_{2,5}$.

It remains to show that $Y_9\notin\SL$. By Lemma \ref{lem:not-all-golden-mean}, we can do this by showing that $Y_9$ is only representable over fields containing a root of $x^2-x-1$. This follows from the fact that SageMath shows its zero determinant ideal contains $z_3^2 - z_3 - 1$. Corollary \ref{cor:PtoF}(ii) then shows that $Y_9$ cannot be $\P_4$-representable.
\end{proof}

\begin{lemma}
\label{lem:III-IV-need-gm}
\leavevmode
\begin{itemize}
\item[(i)]Label the columns of matrix $III'$, from left to right, as $[a,b,c,d,e]$, and let $P_0$ be a column submatrix of matrix $III'$. If $P_0$ contains $[a,b,d]$, or $[c,d,e]$, or $[a,b,c,e]$, then $M=\widetilde{M}([I_3|D_3|P_0])$ contains $Y_9$ as a restriction. Otherwise, $M$ is $\mathbb{P}_4$-representable.
\item[(ii)]Label the columns of matrix $IV'$, from left to right, as $[a,b,c,d,e]$, and let $P_0$ be a column submatrix of matrix $IV'$. If $P_0$ contains at least four columns, including column $a$, then $M=\widetilde{M}([I_3|D_3|P_0])$ contains $Y_9$ as a restriction. Otherwise, $M$ is $\mathbb{P}_4$-representable.
\end{itemize}
\end{lemma}

\begin{proof}
Suppose $P_0$ is a column submatrix of either matrix $III'$ or matrix $IV'$. By Lemmas \ref{lem:B11-2} and \ref{lem:Y9}, every matroid obtained from $B_{11}$ by deleting three elements is $\mathbb{P}_4$-representable. Therefore, if $P_0$ has at most two columns, then $M$ is $\P_4$-representable.

Suppose $P_0$ has three columns. By Lemma \ref{lem:B11-2}, either $M\cong Y_9$, or $M$ is $\mathbb{P}_4$-representable. If $P_0$ is a submatrix of matrix $III'$, then $M\cong Y_9$ precisely when $P_0$ is $[a,b,d]$ or $[c,d,e]$. If $P_0$ is a submatrix of matrix $IV'$, then $M\ncong Y_9$.

Now suppose $P_0$ has four columns. Then either $M\cong S_{10}$, which is $\P_4$-representable by Lemma \ref{lem:S10}, or $M\cong B_{11}\backslash p$, which contains $Y_9$ as a restriction. If $P_0$ is a submatrix of matrix $III'$, then $P_0\cong B_{11}\backslash p$, and $P_0$ contains $[a,b,d]$, or $[c,d,e]$, or $[a,b,c,e]$. If $P_0$ is a submatrix of matrix $IV'$, then $P_0\cong B_{11}\backslash p$ unless $P_0$ does not contain column $a$.
\end{proof}

%

We now characterize the highly connected members of $\mathcal{SL}_4$, subject to Hypotheses \ref{hyp:connected-template} and \ref{hyp:clique-template}. The proofs of the next two theorems are essentially identical to each other. We give the proof of Theorem \ref{thm:SL4connected} but omit the proof of Theorem \ref{thm:SL4clique}.

\begin{theorem}
\label{thm:SL4connected}
Suppose Hypothesis \ref{hyp:connected-template} holds. Then there exists $k\in\mathbb{Z}_+$ such that, for every $k$-connected member $M$ of $\mathcal{SL}_4$ with at least $2k$ elements, either $M$ or $M^*$ is a minor of the vector matroid of a matrix of the form below, where $P_0$ is one of matrices $V$--$XVIII$ listed in Definition \ref{def:possible-matrices}, up to a field isomorphism.
\end{theorem}

\begin{center}
\begin{tabular}{|c|c|c|}
\hline
\multirow{2}{*}{$I_r$}&\multirow{2}{*}{$D_r$}&$P_0$\\
\cline{3-3}
&&$0$\\
\hline
\end{tabular}
\end{center}

\begin{theorem}
\label{thm:SL4clique} 
Suppose Hypothesis \ref{hyp:clique-template} holds. Then there exist $k,n\in\mathbb{Z}_+$ such that every simple vertically $k$-connected member of $\mathcal{SL}_4$ with an $M(K_n)$-minor is a minor of the vector matroid of a matrix of the form above, and every cosimple cyclically $k$-connected member of $\mathcal{SL}_4$ with an $M^*(K_n)$-minor is a minor of the dual of the vector matroid of a matrix of the form above, where $P_0$ is one of matrices $V$--$XVIII$ listed in Definition \ref{def:possible-matrices}, up to a field isomorphism.
\end{theorem}

\begin{proof}[Proof of Theorem \ref{thm:SL4connected}]
We wish to find the templates $\Phi_1,\ldots,\Phi_s,\Psi_1,\ldots,\Psi_t$ for $\mathcal{SL}_4$ whose existence is implied by Corollary \ref{cor:weak-connected-template}. Since $\mathcal{SL}_4\subseteq\AC$, it follows from combining Lemma \ref{lem:possible-matrices} and Theorem \ref{thm:AC4connected} that we may take each of $\Phi_1,\ldots,\Phi_s,\Psi_1,\ldots,\Psi_t$ to be the complete, lifted $Y$-template determined by some column submatrix $P_0$ of one of matrices $I$--$XVI$. 

Consider the complete, lifted $Y$-templates determined by a matrix $P_0$, where $P_0$ is one of matrices $V$--$XVIII$. Let $m$ be the number of rows of $P_0$. By Lemmas \ref{lem:V}--\ref{lem:XVIII}, $M=\widetilde{M}([I_m|D_m|P_0])$ is representable over $\mathbb{K}_2$, $\mathbb{U}_2$, or $\PT$. Therefore, by Corollary \ref{cor:PtoF}(ii), $M\in\mathcal{SL}_4$. Theorem \ref{thm:PtoFtemplate} then implies that $\mathcal{M}(\Phi_{P_0})\subseteq\mathcal{SL}_4$.

However, if $P_0$ is one of matrices $I$--$IV$, Lemmas \ref{lem:I}--\ref{lem:IV}, respectively, combined with Lemma \ref{lem:not-all-golden-mean}, imply that there are infinitely many fields $\F$ such that $\mathcal{M}(\Phi_{P_0})$ is not contained in the class of $\F$-represented matroids. Therefore, $\mathcal{M}(\Phi_{P_0})$ is not contained in $\mathcal{SL}_4$.

Now, suppose $P_0$ is any column submatrix of one of matrices $I$--$IV$ such that $\mathcal{M}(\Phi_{P_0})\subseteq\mathcal{SL}_4$. Then Lemma \ref{lem:Y9} implies that $\widetilde{M}([I|D|P_0])$ cannot contain $Y_9$ as a minor. But then Lemmas \ref{lem:I-need-gm}, \ref{lem:III-IV-need-gm}, and \ref{lem:III-IV-submatrices} imply that $P_0$ must be a submatrix of one of matrices $V$--$XVIII$.

Therefore, we may take $\{\Phi_1,\ldots,\Phi_s\}$ and $\{\Psi_1,\ldots,\Psi_t\}$ both to consist of the complete, lifted $Y$-templates determined by matrices $V$--$XVIII$. Similarly to the proof of Theorem \ref{thm:AC4connected}, we see that every simple $k$-connected member of $\mathcal{SL}_4$ with at least $2k$ elements is a minor of a matroid represented by a matrix of the form given in the statement of the theorem.
\end{proof}

Now we characterize the highly connected members of $\mathcal{AF}_4$, subject to Hypotheses \ref{hyp:connected-template} and \ref{hyp:clique-template}. The proofs of the next two theorems are essentially identical to each other. We give the proof of Theorem \ref{thm:AF4connected} but omit the proof of Theorem \ref{thm:AF4clique}.

\begin{theorem}
\label{thm:AF4connected}
 Suppose Hypothesis \ref{hyp:connected-template} holds. Then there exists $k\in\mathbb{Z}_+$ such that, for every $k$-connected member $M$ of $\mathcal{AF}_4$ with at least $2k$ elements, either $M$ or $M^*$ is a minor of the vector matroid of a matrix of the form below, where $P_0$ is one of the proper column submatrices of matrix $XVI$ that contains three columns, or $P_0$ is one of matrices $V$--$XV$, $XVII$, or $XVIII$ listed in Definition \ref{def:possible-matrices}, up to a field isomorphism.
\end{theorem}

\begin{center}
\begin{tabular}{|c|c|c|}
\hline
\multirow{2}{*}{$I_r$}&\multirow{2}{*}{$D_r$}&$P_0$\\
\cline{3-3}
&&$0$\\
\hline
\end{tabular}
\end{center}

\begin{theorem}
\label{thm:AF4clique} 
Suppose Hypothesis \ref{hyp:clique-template} holds. Then there exist $k,n\in\mathbb{Z}_+$ such that every simple vertically $k$-connected member of $\mathcal{AF}_4$ with an $M(K_n)$-minor is a minor of the vector matroid of a matrix of the form above, and every cosimple cyclically $k$-connected member of $\mathcal{AF}_4$ with an $M^*(K_n)$-minor is a minor of the dual of the vector matroid of a matrix of the form above, where $P_0$ is one of the proper column submatrices of matrix $XVI$ that contains three columns, or $P_0$ is one of matrices $V$--$XV$, $XVII$, or $XVIII$ listed in Definition \ref{def:possible-matrices}, up to a field isomorphism.
\end{theorem}

\begin{proof}[Proof of Theorem \ref{thm:AF4connected}]
We wish to find the templates $\Phi_1,\ldots,\Phi_s,\Psi_1,\ldots,\Psi_t$ for $\mathcal{AF}_4$ whose existence is implied by Corollary \ref{cor:weak-connected-template}. Since $\mathcal{AF}_4\subseteq\mathcal{SL}_4$, it follows from Theorem \ref{thm:SL4connected} that we may take each of $\Phi_1,\ldots,\Phi_s,\Psi_1,\ldots,\Psi_t$ to be the complete, lifted $Y$-template determined by some column submatrix $P_0$ of one of matrices $V$--$XVIII$.

Consider the complete, lifted $Y$-templates determined by a matrix $P_0$, where $P_0$ is one of matrices $V$--$XV$, $XVII$, or $XVIII$. Let $m$ be the number of rows of $P_0$. By Lemmas \ref{lem:V}--\ref{lem:XV}, \ref{lem:XVII}, and \ref{lem:XVIII}, $M=\widetilde{M}([I_m|D_m|P_0])$ is representable over $\mathbb{K}_2$ or $\mathbb{U}_2$. Therefore, by Corollary \ref{cor:PtoF}(i), $M\in\mathcal{AF}_4$. Theorem \ref{thm:PtoFtemplate} then implies that $\mathcal{M}(\Phi_{P_0})\subseteq\mathcal{AF}_4$.

However, if $P_0$ is matrix $XVI$, Lemma \ref{lem:pappus-properties}(i--ii) implies that $\mathcal{M}(\Phi_{P_0})$ is not contained in the class of $\GF(5)$-representable matroids. Therefore, $\mathcal{M}(\Phi_{P_0})$ is not contained in $\mathcal{AF}_4$. However, if $P_0$ is any column submatrix of matrix $XVI$, then $P_0$ must be a submatrix of either matrix $III$ or matrix $IV$, by Lemma \ref{lem:pappus-properties}(iii). Thus, Theorem \ref{thm:GMconnected} and the fact that golden-mean matroids are representable over $\GF(5)$ imply that $\mathcal{M}(\Phi_{P_0})$ is contained in the class of $\GF(5)$-representable matroids. The fact that $P_0$ is a submatrix of matrix $XVI$ implies that the members of $\mathcal{M}(\Phi_{P_0})$ are representable over all fields of size at least $7$. Thus, they are representable over all fields of size at least $4$, and $\mathcal{M}(\Phi_{P_0})\subseteq\mathcal{AF}_4$.

Therefore, we may take $\{\Phi_1,\ldots,\Phi_s\}$ and $\{\Psi_1,\ldots,\Psi_t\}$ both to consist of the complete, lifted $Y$-templates determined by the column submatrices of matrix $XVI$ with three columns and those determined by matrices $V$--$XV$, $XVII$, and $XVIII$. Similarly to the proof of Theorem \ref{thm:AC4connected}, we see that every simple $k$-connected member of $\mathcal{AF}_4$ with at least $2k$ elements is a minor of a matroid represented by a matrix of the form given in the statement of the theorem.
\end{proof}

\begin{lemma}
\label{lem:less-rows}
Theorems \ref{thm:AC4connected}, \ref{thm:AC4clique}, \ref{thm:GMconnected}, \ref{thm:GMclique}, \ref{thm:SL4connected}, \ref{thm:SL4clique}, \ref{thm:AF4connected}, and \ref{thm:AF4clique} still hold if any subset of matrices $I,II,\ldots,XVIII$ are replaced by the corresponding matrices $I',II',\ldots,XVIII'$, respectively.
\end{lemma}

\begin{proof}
Recall the definition of semi-strong equivalence found in Definition \ref{def:semi-strong}. It follows that semi-strongly equivalent $Y$-templates are minor equivalent. Therefore, in Theorems \ref{thm:AC4connected}, \ref{thm:AC4clique}, \ref{thm:GMconnected}, \ref{thm:GMclique}, \ref{thm:SL4connected}, \ref{thm:SL4clique}, \ref{thm:AF4connected}, and \ref{thm:AF4clique}, we may replace one of the templates $\Phi_1,\ldots,\Phi_s,\Psi_1,\ldots,\Psi_t$ with a template that is semi-strongly equivalent to it. Indeed, we may do this to any subset of the templates $\Phi_1,\ldots,\Phi_s,\Psi_1,\ldots,\Psi_t$ that we wish. Lemma \ref{lem:sum-to-zero} states that a complete, lifted $Y$-template determined by a matrix $P_0$, the sum of whose rows is the zero vector, is semi-strongly equivalent to the template determined by the matrix obtained by $P_0$ by removing a row.
\end{proof}

We now determine the extremal functions and extremal matroids of $\mathcal{SL}_4$ and $\mathcal{AF}_4$. Recall the definitions of the matroids $T_r^2$, $G_r$, and $HP_r$ from Definition \ref{def:families}. Also recall, from the discussion following Definition \ref{def:families}, that $T_r^2$, $G_r$, and $HP_r$ are the largest simple matroids of rank $r$ virtually conforming to the templates $\Phi(T_r^2)$, $\Phi(G_r)$, and $\Phi(HP_r)$, respectively. Note that Theorems \ref{thm:golden-mean-extremal} and \ref{thm:SL-AF-extremal} combine to give Theorem \ref{thm:extremal}.

\begin{theorem}
\label{thm:SL-AF-extremal}
Suppose Hypothesis \ref{hyp:clique-template} holds. For all sufficiently large $r$, the extremal matroids of $\mathcal{SL}_4$ and $\mathcal{AF}_4$ are $T_r^2$ and $G_r$. Thus, we have $h_{\mathcal{SL}_4}(r)\approx h_{\mathcal{AF}_4}(r)\approx\binom{r+3}{2}-5$.
\end{theorem}

\begin{proof}
Recall from Theorem \ref{thm:golden-mean-extremal} that the extremal matroids for $\mathcal{AC}_4$ are $T_r^2$, $G_r$, and $HP_r$, each of which have size $\binom{r+3}{2}-5$. Since $\mathcal{AF}_4\subseteq\mathcal{SL}_4\subseteq\mathcal{AC}_4$, it suffices to show that for all $r$, we have $T_r^2\in\mathcal{AF}_4$ and $G_r\in\mathcal{AF}_4$ and that for some $r$, we have $HP_r\notin\mathcal{SL}_4$.

Combining Lemma \ref{lem:ones-and-zeros-construction}, Remark \ref{rem:lifted-Y-template}, and Lemma \ref{lem:complete-construction}, and  we see that $\Phi(T_r^2)$, $\Phi(G_r)$, and $\Phi(HP_r)$ are, respectively, minor equivalent to the complete, lifted $Y$-templates determined by the following matrices. \[
\begin{bmatrix}
\alpha^2&\alpha  \\
\alpha  &\alpha^2\\
1       &0       \\
0       &1       \\
\end{bmatrix},
\begin{bmatrix} 
\alpha^2   	&\alpha^2 	&\alpha^2	&\alpha\\
\alpha 		&0     		&1     		&1\\
0 		&\alpha 	&\alpha     	&\alpha^2\\
1     		&0     		&0     		&0\\
0     		&1     		&0     		&0\\
\end{bmatrix},
\begin{bmatrix}
\alpha&\alpha  \\
1     &\alpha\\
\alpha&1       \\
1     &0       \\
0     &1\\
\end{bmatrix}
\]

For the first of these matrices, swap the first two rows and the result is a submatrix of matrix $XVII$. Combining Lemma \ref{lem:XVII} with Theorem \ref{thm:PtoFtemplate}, we see that $T_r^2\in \mathcal{AF}_4$ for every $r$. For the second of the above matrices, scale each column so that its last nonzero entry is $1$. Then perform a field isomorphism. The result is a submatrix of matrix $XII$. Combining Lemma \ref{lem:XII} with Theorem \ref{thm:PtoFtemplate}, we see that $G_r\in \mathcal{AF}_4$ for every $r$. For the third matrix, swap the first two rows, and call the resulting matrix $P_0$. By Lemma \ref{lem:I-need-gm}, $\widetilde{M}([I_5|D_5|P_0])$ contains $Y_9$ as a minor. Lemma \ref{lem:Y9} combined with the minor equivalence of $\Phi_{P_0}$ and $\Phi(HP_r)$ implies that $HP_r\notin\mathcal{SL}_4$ for some $r$. (In fact, $Y_9$ is a restriction of $HP_3$.)
\end{proof}

The family of matroids $T_r^2$ is a special case of a family of matroids denoted $T_r^k$; the $2$-regular partial field $\mathbb{U}_2$, defined in Section \ref{sec:Partial-Fields}, is a special case of a set of partial fields called the $k$-regular partial fields $\mathbb{U}_k$. See Semple \cite{s99} for the definitions of $T_r^k$ and the $k$-regular partial fields. A matroid is \emph{$k$-regular} if it is representable over $\mathbb{U}_k$. For a prime power $q$, recall the definitions of $\mathcal{AF}_q$ and $\mathcal{SL}_q$ from earlier in this section and the definition of $\mathcal{AC}_q$ from Section \ref{sec:Characteristic Sets}. It is clear that $\mathcal{AF}_q\subseteq\mathcal{SL}_q\subseteq\mathcal{AC}_q$. Semple showed that, for each rank $r$, the matroid $T_r^k$ is $k$-regular \cite[Section 2]{s99} and that every $k$-regular matroid is representable over all fields of size at least $k+2$ \cite[Proposition 3.1]{s96}. (The converse is false however.) Thus, for a prime power $q$, we have $T_r^{q-2}\in\mathcal{AF}_q$. Semple \cite[Theorem 2.3]{s99} showed that the extremal function for the class of $k$-regular matroids is $\varepsilon(T_r^k)$.

The class of regular matroids is $\mathcal{AF}_2$. A result of Heller \cite{h57} implies that the extremal function for the regular matroids is $\binom{r+1}{2}$ and is attained by the complete graphic matroid $M(K_{r+1})$, which is $T_r^0$. A binary matroid is regular if and only if it is representable over some field of characteristic other than $2$ (see \cite[Theorem 6.6.3]{o11}). Therefore, $\mathcal{AC}_2=\mathcal{AF}_2$.

The class of near-regular matroids is $\mathcal{AF}_3$. The $\sqrt[6]{1}$-matroids are representable over fields that contain a root of $x^2-x+1$. Therefore, $\sqrt[6]{1}$-matroids are representable over $\GF(3)$ and $\GF(p^2)$ for all primes $p$ (see also \cite[Theorem 2.30]{pvz10a}). Thus, the class of $\sqrt[6]{1}$-matroids is contained in $\mathcal{AC}_3$. The reverse containment follows from results of Whittle \cite[Theorem 5.1, Lemma 5.7.2, Proposition 3.4, Proposition 3.3]{w97}. Oxley, Vertigan, and Whittle \cite[Theorem 2.1, Corollary 2.2]{ovw98} showed that, for $r\neq3$, the extremal function for both the near-regular matroids and the $\sqrt[6]{1}$-matroids is $\binom{r+2}{2}-2$ and is attained by $T_r^1$. Thus, $h_{\mathcal{AC}_3}(r)=h_{\mathcal{AF}_3}(r)$ for $r\neq3$. Theorem \ref{thm:extremal} shows that, subject to Hypothesis \ref{hyp:clique-template}, $h_{\AC}(r)\approx h_{\SL}(r)\approx h_{\AF}(r)\approx \binom{r+3}{2}-5=\varepsilon(T_r^2)$. Perhaps this pattern continues.

\begin{conjecture}
\label{con:ACq}
If $q$ is a prime power, then $h_{\mathcal{AC}_q}(r)\approx h_{\mathcal{SL}_q}(r)\approx h_{\mathcal{AF}_q}(r)\approx\binom{r+1}{2}+(q-2)(r-1)=\varepsilon(T_r^{q-2})$.
\end{conjecture}

Pendavingh and Van Zwam \cite[Conjecture 6.2]{pvz10a} conjectured that a matroid $M$ is representable over all fields of size at least $4$ if and only if $M$ is representable over $\P_4$. Subject to Hypotheses \ref{hyp:connected-template} and \ref{hyp:clique-template}, we show that this conjecture holds for highly connected matroids.

\begin{theorem}
\label{thm:P4-connected}
Suppose Hypothesis \ref{hyp:connected-template} holds; then there exists $k\in\mathbb{Z}_+$ such that, if $M$ is a $k$-connected matroid with at least $2k$ elements, then $M$ is $\P_4$-representable if and only if $M\in\mathcal{AF}_4$. Moreover, suppose Hypothesis \ref{hyp:clique-template} holds; then there exist $k,n\in\mathbb{Z}_+$ such that, if $M$ is a vertically $k$-connected matroid with an $M(K_n)$-minor, then $M$ is $\P_4$-representable if and only if $M\in\mathcal{AF}_4$.
\end{theorem}

\begin{proof}
We prove the first statement; the proof of the second statement is essentially identical. If $M$ is $\P_4$-representable, then $M\in\AF$, by Corollary \ref{cor:PtoF}(i). To prove the converse, let $k$ be the integer given by Theorem \ref{thm:AF4connected}. That theorem implies that, to prove the result, we must show that $\mathcal{M}(\Phi_{P_0})$ is contained in the class of $\P_4$-representable matroids, where $P_0$ is one of the proper column submatrices of matrix $XVI$ that contains three columns, or $P_0$ is one of matrices $V$--$XV$, $XVII$, or $XVIII$ listed in Definition \ref{def:possible-matrices}. In fact, by Lemma \ref{lem:less-rows}, we replace matrix $XVI$ with matrix $XVI'$. By Theorem \ref{thm:PtoFtemplate} and the fact that there are partial-field homomorphisms $\mathbb{U}_2\rightarrow\mathbb{K}_2\rightarrow\P_4$, it suffices to show that $\widetilde{M}([I_4|D_4|P_0])$ is $\mathbb{U}_2$- or $\mathbb{K}_2$-representable. If $P_0$ is one of matrices $V$--$XV$, $XVII$, or $XVIII$, then this follows by Lemma \ref{lem:V}--\ref{lem:XV}, \ref{lem:XVII}, or \ref{lem:XVIII}.

Therefore, it remains to analyze the case where $P_0$ is a proper column submatrix of matrix $XVI'$. Note that $\widetilde{M}([I_3|D_3|P_0])$ cannot contain $Y_9$ as a minor since $Y_9$ is an excluded minor for $\AF$. By Lemma \ref{lem:pappus-properties}(iii), $\Phi_{P_0}$ is minor equivalent to the complete, lifted $Y$-template determined by a submatrix of matrix $III'$ or $IV'$. By Lemma \ref{lem:B11-2}, since $\widetilde{M}([I_3|D_3|P_0])$ is not $Y_9$, it must be $\P_4$-representable.
\end{proof}

We conclude this section by proving Theorems \ref{thm:SL4exc-minor} and \ref{thm:AF4exc-minor}.

\begin{proof}[Proof of Theorem \ref{thm:SL4exc-minor}]
Let $\mathcal{E}_1=\{F_7,F_7^*,V_1,V_2,V_3,P_1,P_2,P_3,Y_9\}$, and let $\mathcal{E}_2=\{F_7,F_7^*,V_1^*,V_2,V_3^*,P_1^*,P_2,P_3^*,Y_9^*\}$. By Lemmas \ref{lem:connected-excluded-minors} and \ref{lem:clique-excluded-minors}, it suffices to show that the hypotheses (i)--(iv) given in Section \ref{sec:Excluded Minors} are satisfied. By Theorem \ref{thm:AC4-excluded} and Lemma \ref{lem:Y9}, (i) holds. Since $F_7$ is binary, (iv) holds.

We now prove (ii). From the proof of Theorem \ref{thm:AC4exc-minor} (found in Section \ref{sec:The Highly Connected Matroids in AC4}), we know that to every template $\Phi$ such that $\mathcal{M}(\Phi)\nsubseteq\AC$ conforms a matroid that contains a member of $\mathcal{E}_1$ as a minor. Since $\SL\subseteq\AC$ and since matrix $XVII$ is a submatrix of matrix $IV$, we see from Theorems \ref{thm:AC4connected} and \ref{thm:SL4connected} that we only need to consider the complete, lifted $Y$-templates determined by submatrices $P_0$ of matrices $I$--$IV$. In fact, by applying Lemma \ref{lem:less-rows} to Theorem \ref{thm:AC4connected}, we may check submatrices $P_0$ of matrices $I$, $II$, $III'$ and $IV'$. Lemmas \ref{lem:I-need-gm} and \ref{lem:III-IV-need-gm} imply that, to each such template $\Phi$ such that $\mathcal{M}(\Phi)\nsubseteq\SL$, conforms a matroid containing $Y_9$ as a minor. Therefore, (ii) holds.

Since $V_2$ and $P_2$ are self-dual, (iii) holds by (ii) and duality.
\end{proof}

\begin{proof}[Proof of Theorem \ref{thm:AF4exc-minor}]
Let $\mathcal{E}_1=\{F_7,F_7^*,V_1,V_2,V_3,P_1,P_2,P_3,Y_9,Pappus\}$, and let $\mathcal{E}_2=\{F_7,F_7^*,V_1^*,V_2,V_3^*,P_1^*,P_2,P_3^*,Y_9^*,(Pappus)^*\}$. By Lemmas \ref{lem:connected-excluded-minors} and \ref{lem:clique-excluded-minors}, it suffices to show that the hypotheses (i)--(iv) given in Section \ref{sec:Excluded Minors} are satisfied. By Theorem \ref{thm:AC4-excluded} and Lemmas \ref{lem:pappus-properties} and \ref{lem:Y9}, (i) holds. Since $F_7$ is binary, (iv) holds.

We now prove (ii). From the proof of Theorem \ref{thm:SL4exc-minor}, we know that to every template $\Phi$ such that $\mathcal{M}(\Phi)\nsubseteq\SL$ conforms a matroid that contains a member of $\mathcal{E}_1$ as a minor. Since $\AF\subseteq\SL$, we see from Theorems \ref{thm:SL4connected} and \ref{thm:AF4connected} that we only need to consider the complete, lifted $Y$-template determined by matrix $XVI$. By Lemma \ref{lem:pappus-properties}(i), the Pappus matroid is a minor of a matroid that conforms to this template. Therefore, (ii) holds.

Since $V_2$ and $P_2$ are self-dual, (iii) holds by (ii) and duality.
\end{proof}

\section{Summary}
\label{sec:Summary}
We conclude with a proof of Theorem \ref{thm:3sets} and a few additional observations.

\begin{proof}[Proof of Theorem \ref{thm:3sets}]
We prove the statement based on Hypothesis \ref{hyp:connected-template}. The statement based on Hypothesis \ref{hyp:clique-template} is proved similarly.

Let $k_1$, $k_2$, $k_3$, and $k_4$ be the values for $k$ given by Theorems \ref{thm:AC4connected}, \ref{thm:GMconnected}, \ref{thm:SL4connected}, and \ref{thm:AF4connected}, and let $k=\max\{k_1,k_2,k_3,k_4\}$. If $\mathcal{M}$ is a minor-closed class of matroids, let $\mathcal{M}(k)$ denote the set of $k$-connected members of $\mathcal{M}$ with at least $2k$ elements.

Combining Theorems \ref{thm:AC4connected}, \ref{thm:GMconnected}, and \ref{thm:SL4connected}, we see that $\mathcal{AC}_4(k)=\mathcal{GM}(k)\cup\mathcal{SL}_4(k)$. Combining Theorem \ref{thm:SL4connected} with the lemmas in Section \ref{sec:The Highly Connected Matroids in AC4}, we see that the members of $\mathcal{SL}_4(k)$ are representable over all fields of size at least $7$. Therefore, a member of $\mathcal{SL}_4(k)$ is a member of $\mathcal{AF}_4(k)$ if and only if it is representable over $\GF(5)$, implying that it is a member of $\mathcal{GM}(k)$. Thus, $\mathcal{AC}_4(k)$ is the disjoint union $\mathcal{AF}_4(k)\cup(\mathcal{GM}(k)-\mathcal{AF}_4(k))\cup(\mathcal{SL}_4(k)-\mathcal{AF}_4(k))$.
\end{proof}

By Theorem \ref{thm:3sets}, if $M$ is a large, highly connected member of $\mathcal{AC}_4$, then the set of fields of size at least $4$ over which $M$ is representable is one of exactly three sets. This is in contrast to the general case where there are infinitely many such sets. Oxley, Vertigan, and Whittle \cite{ovw96} showed that, for all $r$, the rank-$r$ free spike is representable over all finite fields of non-prime order. Therefore, the rank-$r$ free spike is a member of $\AC$. However, Geelen, Oxley, Vertigan, and Whittle \cite{govw02} showed that, for all primes $p\leq r+1$, the rank-$r$ free spike is not representable over $\GF(p)$.

Despite the previous paragraph, it may still be the case that $\AF=\GM\cap\SL$ in general. This is an interesting question for future research.

As a final remark, we note that the words ``up to a field isomorphism'' in Theorems \ref{thm:AC4connected}, \ref{thm:AC4clique}, \ref{thm:GMconnected}, \ref{thm:GMclique}, \ref{thm:SL4connected}, \ref{thm:SL4clique}, \ref{thm:AF4connected}, and \ref{thm:AF4clique} can be dropped if we are content to deal with abstract matroids rather than represented matroids.

\renewcommand{\thesection}{Appendix}
\section{Computing Universal Partial Fields}
\label{sec:Appendix}

\setcounter{theorem}{0}

The use of the following lemma was incorporated into the code used to compute universal partial fields.

\begin{app-lemma}
\label{app-lem:1and-1}
Let $P_0$ be a matrix over $\GF(4)$ that contains a submatrix of the form $[1,\alpha,\alpha,1]^T$ (up to field isomorphism and permuting rows). If $\mathcal{M}(\Phi_{P_0})$ is algebraically equivalent to a template $\mathcal{M}(\Phi_{P_0}')$ over a partial field $\P$, then the corresponding submatrix of $P_0'$ must be $[-1,-x,x,1]^T$ for some $x\in\P-\{0,1\}$.
\end{app-lemma}

\begin{proof}
Consider the following submatrix of $[I_4|D_4|P_0]$.
\[\begin{bmatrix}
1&0&0&0&1&1\\
0&1&0&0&0&\alpha\\
0&0&1&0&0&\alpha\\
0&0&0&1&1&1\\
\end{bmatrix}\]
If the elements represented by the second and third columns are contracted, we see that the elements represented by the last two columns become a parallel pair. In the corresponding submatrix of $P_0'$, the nonzero entries of the fifth column are $1$ and $-1$. Therefore, the entries of $P_0'$ corresponding to the $1$s of the last column above, must be a $1$ and a $-1$ also.
\end{proof}

The function \texttt{complete\_template\_representation} takes as input a matrix \texttt{P0} over $\GF(4)$. It returns a pair of matrices \texttt{A4} and \texttt{Avar}. The matrix \texttt{A4} is the matrix $[I_r|D_r|P_0]$ over $\GF(4)$, where each column of $P_0$ has been scaled so that the last nonzero entry is $1$. The matrix \texttt{Avar} has entries from a polynomial ring $\mathbb{Z}[z_0,z_1,\ldots,z_n]$ for some $n$, and it is of the form $[I_r|D_r|P_0']$ for some matrix $P_0'$ with the same zero-nonzero pattern as $P_0$. The symbols $z_0,z_1,\ldots,z_n$ are indeterminates. In order for \texttt{Avar} to be a representation of $M$, there will be certain relationships between the indeterminates. The columns of $P_0'$ are scaled so that the last nonzero entry is a $1$. If there is a second $1$ in a column of $P_0$, then the corresponding entry of $P_0'$ is $-1$, by Appendix Lemma \ref{app-lem:1and-1}.

The function \texttt{zero\_determinant\_ideal} takes as input a matroid \texttt{M} and a matrix \texttt{A}, over a ring $R$, of the form $A=[I_r|A']$, where $r=r(M)$ and where the rows and columns of $I_r$ are indexed by a set $B$. For each column, indexed by $c$, and row, indexed by $b$, of $A'$ the entry $A'_{b,c}$ is nonzero if and only if the basis element with nonzero entry in row $b$ is in the $B$-fundamental circuit of $c$. For each size-$r$ subset of the ground set of $M$ that is not a basis, we compute the determinant of the corresponding square submatrix of \texttt{Avar}. (We call these size-$r$ sets \emph{nonbases}.) In order for \texttt{Avar} to represent \texttt{M}, these determinants should be $0$. To do this, we need the quotient ring of $R$ modulo the ideal generated by all of these determinants. The function returns a Gr\"obner basis for this ideal.

The function \texttt{check\_partial\_field} takes as input a matroid \texttt{M}, a matrix \texttt{A}, and an ordering \texttt{E} of the elements of \texttt{M}, as in the function \texttt{zero\_determinant\_ideal}. It also takes as an argument a list of generators of a multiplicative group $G$ where $\P=(R,G)$ is some partial field. These generators should be elements of the fraction field of the base ring of $A$. The function determines if \texttt{A} is a $\P$-matrix that represents \texttt{M}; it returns either \texttt{True} or \texttt{False, B, determ} where \texttt{B} is some basis of \texttt{M} such that the determinant \texttt{determ} of the corresponding submatrix of \texttt{A} is not a product of the generators. In order to ensure that we calculate the universal partial field, we may need to run this function several times. We begin only with the generators that can be deduced from the Gr\"obner basis of the ideal; for each time the function returns \texttt{False, B, determ}, we add a generator (and possibly also its inverse, although this is not always necessary) to $G$, based on \texttt{determ}. By Definition \ref{def:UPF}, the universal partial field of \texttt{M} is generated by the union of $\{-1\}$ and the set of inverses of subdeterminants of \texttt{A}. The function \texttt{check\_partial\_field} ensures that each subdeterminant of \texttt{A} is a product of the generators of the partial field. Therefore, the inverses of these products are also in the partial field. This implies that our process of adding a generator each time the function returns \texttt{False} will lead to a determination of the universal partial field.

\begin{proof}[Detailed proof of Lemma \ref{lem:I}]
The following code was used to determine the relations that must be satisfied between the nonzero entries of a matrix representing $M$. Besides $1$ and $-1$, these entries are called $z_0,z_1,\ldots,z_{11}$, as explained above. (The entries in the leftmost columns are assigned the variables first. Within a column, the uppermost entry is assigned a variable first.)

\begin{verbatim}
P0 = Matrix(GF4, [[1,1,a,a,a,a^2],
                  [a,a,a,a,a^2,a],
                  [a,a,1,1,1,1],
                  [1,0,1,0,0,0],
                  [0,1,0,1,0,0]])
(A4,Avar)=complete_template_representation(P0)
M = Matroid(A4)
E = M.groundset_list()
I = zero_determinant_ideal(M, Avar)
I
Avar.base_ring().inject_variables()
Avar.change_ring(Avar.base_ring().fraction_field())
\end{verbatim}

The function \texttt{zero\_determinant\_ideal} returned the ideal with the Gr\"obner basis $\{z_{11}^2 + z_{11} - 1$, $z_0 + z_{11}$, $z_1 - z_{11}$, $z_2 + z_{11}$, $z_3 - z_{11}$, $z_4 - z_{11}$, $z_5 + z_{11}$, $z_6 - z_{11}$, $z_7 + z_{11}$, $z_8 + z_{11}$, $z_9 - z_{11} + 1$, $z_{10} + z_{11} + 1\}$. Take the quotient ring of the fraction field of the base ring of \texttt{Avar} modulo this ideal. The fact that $z_{11}^2 + z_{11} - 1$ and $z_2+z_{11}$ are in the ideal implies that, in the quotient ring, $z_2^2-z_2-1=0$. Thus, $z_2$ is a root of $x^2-x-1$ in the quotient ring. Thus, any field over which $M$ has a representation must contain a root of this polynomial. Moreover, $z_2+1=z_2^2$ and $z_2-1=z_2^{-1}$. In fact, we also have $z_2^2 - 2z_2 = z_2^{-1}$. Therefore, although the only generators for the golden-mean partial field are $-1$ and $z_2$, the ideal allows us to include $z_2+1$, $z_2-1$, and $z_2^2 - 2z_2$ as generators. The following code returns \verb|True| and therefore confirms that \texttt{Avar} is a $\mathbb{G}$-matrix after we pass from its base ring to the quotient ring.

\begin{verbatim}
Avar2 = Avar(z0=z2,z1=-z2,z3=-z2,z4=-z2,z5=z2,z6=-z2,z7=z2,
z8=z2,z9=-z2-1,z10=z2-1,z11=-z2)
check_partial_field(M, Avar2, M.groundset_list(),
[-1,z2,z2+1,z2-1,z2^2 - 2z2],[])
\end{verbatim}
\end{proof}

\begin{proof}[Detailed proof of Lemma \ref{lem:V}]
 The following code was used to determine the relations that must be satisfied between the nonzero entries of a matrix representing $M$. Besides $1$ and $-1$, these entries are called $z_0,z_1,\ldots,z_5$, as explained above. (The entries in the leftmost columns are assigned the variables first. Within a column, the uppermost entry is assigned a variable first.)

\normalsize
\begin{verbatim}
P0 = Matrix(GF4, [[1,a  ,0],
                  [a,a^2,0],
                  [a,0  ,a^2],
                  [1,0  ,a],
                  [0,1  ,1]])
(A4,Avar)=complete_template_representation(P0)
M = Matroid(A4)
E = M.groundset_list()
I = zero_determinant_ideal(M, Avar)
I
Avar.base_ring().inject_variables()
Avar.change_ring(Avar.base_ring().fraction_field())
\end{verbatim}

The function \texttt{zero\_determinant\_ideal} returned the ideal with the Gr\"obner basis $\{z_1z_5 + z_5 + 1$, $z_0 + z_1$, $z_2 - z_5$, $z_3 + z_5 + 1$, $z_4 + z_5 + 1\}$. Take the quotient ring of the fraction field of the base ring of \texttt{Avar} modulo this ideal. If $z_0$, $z_0+1$, and $z_0-1$ are generators of a multiplicative group, then $1/z_0$, $1/(z_0+1)$, and $1/(z_0-1)$ are in the group and can also be added to the list of generators. The following code returns \verb|True| and therefore confirms that \texttt{Avar} is a $\mathbb{K}_2$-matrix after we pass from its base ring to the quotient ring.

\normalsize
\begin{verbatim}
Avar2 = Avar(z1=-z0,z2=1/(z0-1),z3=-z0/(z0-1),z4=-z_0/(z_0-1),
z_5=1/(z_0-1))
check_partial_field(M, Avar2, M.groundset_list(),
[-1,z_0,z_0-1,z_0+1,1/z_0,1/(z_0+1),1/(z_0-1)],[])
\end{verbatim}
\end{proof}

\section*{Acknowledgements}
This paper is based on the author's PhD dissertation completed at Louisiana State University under the supervision of Stefan H.M. van Zwam, without whose help this paper would not have been possible. The author also thanks James Oxley for his sketches of the geometric representations given in Figures \ref{fig:V_1} and \ref{fig:V_2} {\color{red}and} Dillon Mayhew for sharing some SageMath code that was used in Section \ref{sec:Characteristic Sets}. {\color{red}Finally, the author thanks the anonymous reviewers for carefully reading the manuscript and pointing out needed improvements.}

\end{document}